\documentclass[onefignum,onetabnum]{siamart171218}

\usepackage{lipsum}
\usepackage{amsfonts}
\usepackage{graphicx}
\usepackage{epstopdf}
\usepackage{algorithmic}
\ifpdf
  \DeclareGraphicsExtensions{.eps,.pdf,.png,.jpg}
\else
  \DeclareGraphicsExtensions{.eps}
\fi

\newtheorem{nnremark}[theorem]{\sc Remark}
\newenvironment{remark}{\begin{nnremark} \rm }{\hfill \hspace*{1pt}\hfill $\lrcorner$\end{nnremark}}

\newsiamremark{hypothesis}{Hypothesis}
\crefname{hypothesis}{Hypothesis}{Hypotheses}
\newsiamthm{claim}{Claim}

\headers{ISS of infinite-dimensional systems}{A. Mironchenko and Ch. Prieur}

\title{Input-to-state stability of infinite-dimensional systems: recent results and open questions\thanks{Submitted to the editors DATE.
\funding{A. Mironchenko was supported by German Research Foundation (DFG), grant: MI 1886/2-1.}}}

\author{Andrii Mironchenko\thanks{Faculty of Computer Science and Mathematics, University of Passau, Germany 
  (\email{andrii.mironchenko@uni-passau.de}, \url{http://www.mironchenko.com}, corresponding author).}
\and Christophe Prieur\thanks{Univ. Grenoble Alpes, CNRS, Grenoble-INP, GIPSA-lab, F-38000, Grenoble, France 
  (\email{christophe.prieur@gipsa-lab.fr}).}
}

\usepackage{amsopn}
\DeclareMathOperator{\diag}{diag}

\ifpdf
\hypersetup{
  pdftitle={ISS of infinite-dimensional systems},
  pdfauthor={A. Mironchenko and Ch. Prieur}
}
\fi

\usepackage{amsmath} 
\usepackage{amssymb}
\usepackage{tikz}

\usepackage{color}

\usepackage{subcaption}

\usepackage{pdfsync}

\usepackage{enumerate}

\usepackage{longtable} 

\usepackage{marginnote}  

\usepackage{centernot}  

\usepackage{multicol}

\usepackage{hyperref}   
\hypersetup{
    colorlinks=true,
    linkcolor=blue,
    filecolor=magenta,      
    urlcolor=blue,
		citecolor=red,
    pdftitle={Sharelatex Example},
    bookmarks=true,
		bookmarksdepth=2,
		bookmarksopen,
    pdfpagemode=FullScreen,
}

\usetikzlibrary{arrows}
\usetikzlibrary{circuits.ee.IEC}
\usetikzlibrary{graphs,decorations.pathmorphing,decorations.markings}
\usetikzlibrary{calc}

\pgfmathsetmacro\weight{1/2}
\pgfmathsetmacro\third{1/3}
\pgfmathsetmacro\twothirds{2/3}

\tikzset{degil/.style={
            decoration={markings,
            mark= at position 0.5 with {
                  \node[transform shape] (tempnode) {$/$};
                  }
              },
              postaction={decorate}
}
}

\newtheorem{Ass}{Assumption} 

\newtheorem{example}[theorem]{Example}
\newtheorem{openprob}[theorem]{Open Problem}

\numberwithin{equation}{section}

\makeatletter
\DeclareOldFontCommand{\rm}{\normalfont\rmfamily}{\mathrm}
\DeclareOldFontCommand{\sf}{\normalfont\sffamily}{\mathsf}
\DeclareOldFontCommand{\tt}{\normalfont\ttfamily}{\mathtt}
\DeclareOldFontCommand{\bf}{\normalfont\bfseries}{\mathbf}
\DeclareOldFontCommand{\it}{\normalfont\itshape}{\mathit}
\DeclareOldFontCommand{\sl}{\normalfont\slshape}{\@nomath\sl}
\DeclareOldFontCommand{\sc}{\normalfont\scshape}{\@nomath\sc}
\makeatother

\newcommand{\ccat}[3]{{#1\, \underset{#3}{\lozenge}\,{#2}}}

\newcommand{\tm}{\times}

\newcommand  \esssup {\mathop{\mathrm{ess} \sup} }

\renewcommand \diag  {\operatorname{diag}}

\newcommand \eps {\varepsilon}

\newcommand \N   {\mathbb{N}}
\newcommand \R   {\mathbb{R}}
\newcommand \C   {\mathbb{C}}
\newcommand \Z   {\mathbb{Z}}
\newcommand \Q   {\mathbb{Q}}

\newcommand{\Ah}{\hat{A}}

\newcommand{\Gh}{\hat{G}}

\newcommand \K   {\mathcal{K}}
\newcommand \Kinf{\mathcal{K_\infty}}

\newcommand \KL  {\mathcal{KL}}

\newcommand \LL  {\mathcal{L}}

\newcommand \PD   {\mathcal{P}}

\newcommand{\Uc}{\ensuremath{\mathcal{U}}}

\usepackage{csquotes}  
\newcommand\q{\enquote}

\newcommand{\lel}{\left\langle}
\newcommand{\rir}{\right\rangle}

\newcommand \qrq   {\quad\Rightarrow\quad}

\newcommand \srs   {\ \ \Rightarrow\ \ }

\newcommand \Iff   {\Leftrightarrow}

\newcommand \id  {\operatorname{id}}

\newcommand \re  {\mathrm{Re}}
\newcommand \im  {\mathrm{Im}}

\renewcommand{\ker}{{\rm Ker}\,}
\newcommand{\intt}{{\rm int}\,}

\newcommand{\mir}[1]{{\color{red} AM: #1}}

\definecolor{forestgreen}{rgb}{0.13,0.54,0.13}

\newcommand{\midset}{\;:\;}

\newcommand{\sym}{\mathrm{sym}}

\newif\ifConf

\newif\ifJournal

\newif\ifAndo

\newif\ifNOTFORBOOK

\newif\ifExtendedVersion   

\newif\ifFullVersion

\newif\ifFinal  

\usepackage{xifthen}

\newcommand{\einsnorm}[2]{\ensuremath{
    \!\!\;\!\!\!\;
    \left\bracevert\!\!\!\!\!\left\bracevert
    \!
		\ifthenelse{\isempty{#2}}{#1}{#1(#2)}
    \!
      \right\bracevert\!\!\!\!\!\right\bracevert
    \!\!\;\!\!\!\;
  }}

\usepackage{xcolor}
\definecolor{blond}{rgb}{0.98, 0.94, 0.75}
	
\newlength\mytemplen
\newsavebox\mytempbox

\makeatletter
\newcommand\mybluebox{%
    \@ifnextchar[
       {\@mybluebox}%
       {\@mybluebox[0pt]}}

\def\@mybluebox[#1]{%
    \@ifnextchar[
       {\@@mybluebox[#1]}%
       {\@@mybluebox[#1][0pt]}}

\def\@@mybluebox[#1][#2]#3{
    \sbox\mytempbox{#3}%
    \mytemplen\ht\mytempbox
    \advance\mytemplen #1\relax
    \ht\mytempbox\mytemplen
    \mytemplen\dp\mytempbox
    \advance\mytemplen #2\relax
    \dp\mytempbox\mytemplen
    \colorbox{blond}{\hspace{1em}\usebox{\mytempbox}\hspace{1em}}}

\makeatother

\makeatletter
\let\origd=\d
\renewcommand*\d{
  \relax\ifmmode
    \mathrm{d}%
  \else
    \expandafter\origd
  \fi
}\makeatother

\usepackage{mathtools}

\newcommand{\scalp}[2]{ \lel #1, #2 \rir }

\makeatletter 
\newcommand{\pushright}[1]{\ifmeasuring@#1\else\omit\hfill$\displaystyle#1$\fi\ignorespaces}
\newcommand{\pushleft}[1]{\ifmeasuring@#1\else\omit$\displaystyle#1$\hfill\fi\ignorespaces}
\makeatother 

\newcounter{syscounter}
\newenvironment{sysnum}{\begin{list}{($\Sigma{\arabic{syscounter}}$)}%
{\settowidth{\labelwidth}{($\Sigma4$)}
\settowidth{\leftmargin}{($\Sigma4$)~}%
\usecounter{syscounter}}}
{\end{list}}

\newcounter{WPcounter}

\Finalfalse  
\ExtendedVersionfalse

\makeatother
 
\begin{document}

\maketitle

\begin{abstract}
In a pedagogical but exhaustive manner, this survey reviews the main results on input-to-state stability (ISS) for infinite-dimensional systems. This property allows estimating the impact of inputs and initial conditions on both the intermediate values and the asymptotic bound on the solutions. ISS has unified the input-output and Lyapunov stability theories and is a crucial property in the stability theory of control systems as well as for many applications whose dynamics depend on parameters, unknown perturbations, or other inputs. 
In this paper, starting from classic results for nonlinear ordinary differential equations, we motivate the study of ISS property for distributed parameter systems. 
Then fundamental properties are given, as an ISS superposition theorem and characterizations of (global and local) ISS in terms of Lyapunov functions. 
We explain in detail the functional-analytic approach to ISS theory of linear systems with unbounded input operators, with special attention devoted to ISS theory of boundary control systems.
The Lyapunov method is shown to be very useful for both linear and nonlinear models, including parabolic and hyperbolic partial differential equations. 
Next, we show the efficiency of the ISS framework to study the stability of large-scale networks, coupled either via the boundary or via the interior of the spatial domain. ISS methodology allows reducing the stability analysis of complex networks, by considering the stability properties of its components and the interconnection structure between the subsystems. An extra section is devoted to the ISS theory of time-delay systems with the emphasis on techniques, which are particularly suited for this class of systems. Finally, numerous applications are considered in this survey, where ISS properties play a crucial role in their study. This survey contains recent as well as classical results on systems theory and suggests many open problems throughout the paper.
\end{abstract}



%
%
%
%
%
%
%

%
%


\setcounter{tocdepth}{1}
\tableofcontents

\section{Introduction}

The research community has been active on the controllability and stabilizability of linear partial differential equations during the past 50 years (see in particular the highlighting survey \cite{Rus78}). More recently, control problems for nonlinear partial differential equations (PDEs) have been studied by many researchers (see, e.g., \cite{Cor07}). In parallel with these developments, the concept of input-to-state stability, which reflects the robustness of nonlinear systems with respect to both initial states and external disturbances, has revolutionized the nonlinear control theory of finite-dimensional systems \cite{KoA01}.
Now we face the merge of infinite-dimensional ISS theory and PDE control methods, which is going to bring systematic methods for robust control of networks of infinite-dimensional systems. 
In this survey, we outline the key results and techniques in infinite-dimensional ISS theory as well as applications of this theory to robust control of linear and nonlinear PDEs.

\subsection{Input-to-state stability}
\label{sec:ISS-Intro}

The concept of input-to-state stability (ISS), introduced by E. Sontag in the late 1980s in \cite{Son89}, unified the Lyapunov and input-output stability theories and revolutionized the constructive nonlinear control theory for finite-dimensional systems \cite{KoA01}.
It played a major role in robust stabilization of
nonlinear systems \cite{FrK08}, design of robust 
nonlinear observers \cite{Lib07}, nonlinear
detectability \cite{SoW97,KSW01}, stability of nonlinear networked control systems \cite{JMW96, DRW07},
supervisory adaptive control \cite{HeM99b} and others.

Such an overwhelming success has become possible 
thanks to the powerful tools provided by the ISS framework. The first of them is the equivalence between ISS and existence of a smooth ISS Lyapunov function \cite{SoW95}, which naturally generalizes the theorems of Massera and Kurzweil, known for systems without inputs. The second result is an ISS superposition theorem \cite{SoW96}, which characterized ISS as the combination of global asymptotic stability in the absence of disturbances together with a kind of a global attractivity property for the system with inputs. Finally, the nonlinear small-gain theorem provided a powerful criterion for input-to-state stability of a network consisting of an arbitrary finite number of ISS components \cite{JTP94, JMW96, DRW07, DRW10}.
These results fostered the development of new techniques for the design of non-linear stabilizing controllers: 
robust backstepping \cite{KKK95}, ISS feedback redesign \cite{Son89}, stabilization via controllers with saturation \cite{Tsi97}, etc., which make the closed-loop system not only asymptotically stable in absence of disturbances, but also robust with respect to disturbances, appearing due to actuator and measurement errors, modeling errors, hidden dynamics of a system or external disturbances.

Nowadays ISS of ordinary differential equations (ODEs) is a mature theory which is recognized as a milestone in the development of the constructive nonlinear control theory and there are several excellent survey papers \cite{Son08, DES11, JLR11} as well as book chapters devoted to this paradigm \cite{KKK95, LJH14, KaJ11}.


The questions of the robust stabilization and observation, as well as robust stability analysis of coupled systems,  are as important for infinite-dimensional systems as they are for finite-dimensional ones.
Modern applications of control theory to traffic networks \cite{herrera2010evaluation,de2015grenoble}, multi-body systems (e.g., robotic arms, flexible elements) \cite{geradin2014mechanical}, adaptive optics \cite{baudouin2008robust}, fluid-structure interactions (in particular for aircraft wings
\cite{robu2011simultaneous}), etc., require methods for robust stabilization of coupled systems, described by partial differential equations (PDEs). 
In particular, the following challenging and fascinating questions come to the foreground:
\emph{How to control such complex systems, if we can physically apply controls only to the boundary of their spatial domain? How to ensure the efficiency of a control design in spite of actuator and observation errors, hidden dynamics and external disturbances? Under which conditions a coupled large-scale infinite-dimensional system is stable if all its components are stable?} 
For many classes of distributed parameter systems, it is known that small noise (even small in amplitude) can reduce the performance, alter the stability, or even destabilize the control system. Thus the robustness property has to be analyzed, and, the impact of perturbations should be described. 
The {\em input-to-state stability} property is central for many control applications in part since {\em inputs} may be either disturbance inputs or control inputs or both. In the first case, the gains in the ISS property give an estimation of the disturbance effect on the performance and on the size of the asymptotic error in the stability. 
In the second case, the knowledge of the ISS property and corresponding ISS gains is useful for the design of controllers ensuring a suitable type of stability and performance of the closed-loop system. Furthermore, ISS property combined with the small-gain methods is essential for the control of feedback interconnections.

A decade ago the development of ISS theory for infinite-dimensional systems has been started, which 
aims at giving systematic mathematical tools that are suitable for tackling the challenging issues described in the previous paragraph.
This theory exploits a broad range of techniques and tools from such diverse fields as Lyapunov methods, semigroup and admissibility theories, spectral methods, control of networks, etc., whose interplay results in efficient methods of robust stability analysis of large-scale infinite-dimensional systems.
Many works study the ISS by itself for the benefits this property gives to the closed-loop systems, see \cite{GLO19,LhS18,MKK19,MiW18b,PrM12,TPT16}, where different techniques (as in particular Lyapunov methods, monotonicity or spectral methods) are used to prove ISS properties of the control systems under consideration.  Other papers develop control design methods using ISS techniques, and exploit this property, as an intermediate step, for more general control constructions. For example, see \cite{karafyllis2018sampled,KaK19b,Event:espitia:aut:2019} for constructions of feedback laws and see \cite{observers:tac:2019} for observers construction of infinite-dimensional systems of various types, using ISS arguments.
To conclude, right now we are witnessing a fusion of the infinite-dimensional ISS theory with modern methods of PDE control, which are going to provide us with systematic methods for the design of robust controllers and observers for coupled infinite-dimensional systems with heterogeneous components.
In this survey, we give an overview of the main results in this booming theory to a broad mathematical audience.

\subsection{Outline of the paper}
\label{sec:Outline}

In the rest of this section, we motivate the input-to-state stability property by considering several examples of linear and nonlinear finite-dimensional systems. Next, we introduce the abstract class of infinite-dimensional systems, which encompasses a broad range of distributed parameter systems, important in practice and formally define the ISS concept.
Such an abstract approach allows us to develop a theory which
\begin{itemize}
    \item covers both linear and nonlinear systems
    \item encompasses PDEs, time-delay systems and ODE systems
    \item can be a firm basis for the study of coupled systems with components of different types (PDEs, ODEs, delay systems, etc.) and with different types of couplings (in-domain and boundary couplings).
\end{itemize} 
With this in mind, we present a unified ISS theory for a broad class of infinite-dimensional systems, which can be then refined for more specific systems classes.

In Section~\ref{sec:Fundamental-properties-ISS-systems} we derive the ISS superposition theorem and characterize local and global ISS in terms of ISS Lyapunov functions. These tools are indispensable for the theoretical and practical analysis of input-to-state stability of nonlinear control systems, both finite and infinite-dimensional. We close the section by introducing the important notion of integral input-to-state stability (iISS).

In Section~\ref{sec:Linear_systems} we proceed to the general theory of linear infinite-dimensional systems in Banach spaces with unbounded input operators. We characterize ISS in terms of exponential stability of the semigroup and the admissibility of the input operator, investigate relations between ISS and integral ISS of linear systems and explain the methods of construction of Lyapunov functions for such systems.

Next we show in Section~\ref{sec:Boundary_control_systems} how the criteria obtained in Section~\ref{sec:Linear_systems} 
can be used to study linear boundary control systems. We apply the results to Riesz-spectral systems whose space of input values is finite-dimensional.

Having presented general ISS theory of nonlinear and linear systems, we specialize ourselves to important classes of infinite-dimensional systems. 
In Section~\ref{sec:ISS_analysis_linear_nonlinear_PDEs_Lyapunov_methods} we show how Lyapunov methods combined with classic inequalities (Friedrichs, Poincare, Agmon, Jensen inequalities, etc.) can be used to obtain criteria for input-to-state stability of parabolic semilinear systems with both in-domain and boundary inputs, as well as for hyperbolic first-order systems with in-domain and boundary inputs.

In Section~\ref{sec:Interconnected_systems} we present general results on the stability of networks of infinite-dimensional systems with ISS components. We show that the whole network is stable provided certain types of small-gain conditions hold.
We formulate the small-gain results both in trajectory formulation and in the Lyapunov form, and for systems with both  ISS and integral ISS components.

In Section~\ref{sec:TDS} we briefly outline several central results in ISS theory of retarded differential equations.
In particular, we present ISS counterparts of classic Lyapunov-Krasovskii and Lyapunov-Razumikhin theorems, which allow for efficient stability analysis of nonlinear retarded differential equations.
Overview of available small-gain results is given as well.

Several applications to real-world control problems are given in Section~\ref{sec:Applications}. Some numerically tractable conditions are given there and real experimental results are reviewed, where ISS properties and Lyapunov functions play a crucial role.

The scope and richness of the ISS theory, as well as a broad range of methods which are used, make it impossible to cover all available results in one paper. In Section~\ref{sec:Further-topics} we briefly mention some further important results in ISS theory. In particular, we discuss strong and weak input-to-state stability, which play an important role in applications, ISS of infinite-dimensional Lur'e systems, ISS of monotone systems, numerical methods to construct ISS Lyapunov functions, stability of impulsive systems and practical input-to-state stability.

In a recent monograph \cite{KaK19} an ISS theory for partial differential equations with boundary controls has been proposed, which is based upon spectral analysis of boundary value problems. 
Whereas the book \cite{KaK19} puts the stress on quantitative aspects, and studies first of all classical solutions,
 we discuss in this survey qualitative aspects of the infinite-dimensional ISS theory in detail and provide an overview of a broad class of methods for ISS analysis.
We consider briefly the ISS theory for ODE systems to stress the similarities and differences to the infinite-dimensional theory. For a detailed study of the finite-dimensional ISS theory we refer to several excellent surveys on this topic \cite{Son08, DES11, JLR11}.

\def\x{1.85}
\def\y{1.9}

\begin{figure}%
\begin{tikzpicture}[text width=3.3cm, text centered, rounded corners, minimum height=1.1cm, scale = 1, transform shape]

\node [rectangle, draw](Intro)   at (0,\x) {Introduction};

\node [rectangle, draw,text width=4.3cm](Fund)   at (0,0) {Fundamentals of ISS:\\Lyapunov functions, Criteria for ISS};
\node [rectangle, draw](Couplings)   at (2.25,-2*\y) {Stability of\\ interconnected $\infty$-dim systems};
\node [rectangle, draw](TDS)   at (4.5,-\x) {ISS of nonlinear\\ time-delay systems};
\node [rectangle, draw, rounded corners](PDE)   at (0,-\x) {Lyapunov methods \\ for PDE systems};
\node [rectangle,  draw](Lin)   at (-4.5,-\x){ Linear systems\\ $\dot{x} = Ax + Bu$ };
\node [rectangle, draw](BCS)   at (-4.5,-2*\y){Boundary \\ control systems\\ $\dot{x} = \Ah x, \ \  \Gh x=u$ };

\node [rectangle, draw, text width=9cm](APP)   at (0,-3.2*\y) {Applications\\[1mm]

\begin{itemize}
    \item Robust control/observation\\[-15mm]
    \item Riesz-spectral and port-Hamiltonian systems\\[-15mm]
    \item PDE-ODE, PDE-PDE cascades 
\end{itemize}
};

\draw[->, thick] (Intro) to (Fund);
\draw[->, thick] (Fund) to (Lin);
\draw[->, thick] (Lin) to (BCS);

\draw[->, thick] (Fund) to (PDE);
\draw[->, thick] (Fund) to (TDS);
\draw[->, thick, bend angle=35, bend left] (Fund) to (Couplings);
\draw[->, thick] (PDE) to (Couplings);

\draw[->, thick] (PDE) to (APP);
\draw[->, thick, bend angle=32, bend left] (TDS) to (APP);
\draw[->, thick] (Couplings) to (APP);
\draw[->, thick] (BCS) to (APP);

\end{tikzpicture}
\caption{Topics considered in the manuscript}%
\label{fig:Paper-structure}%
\end{figure}
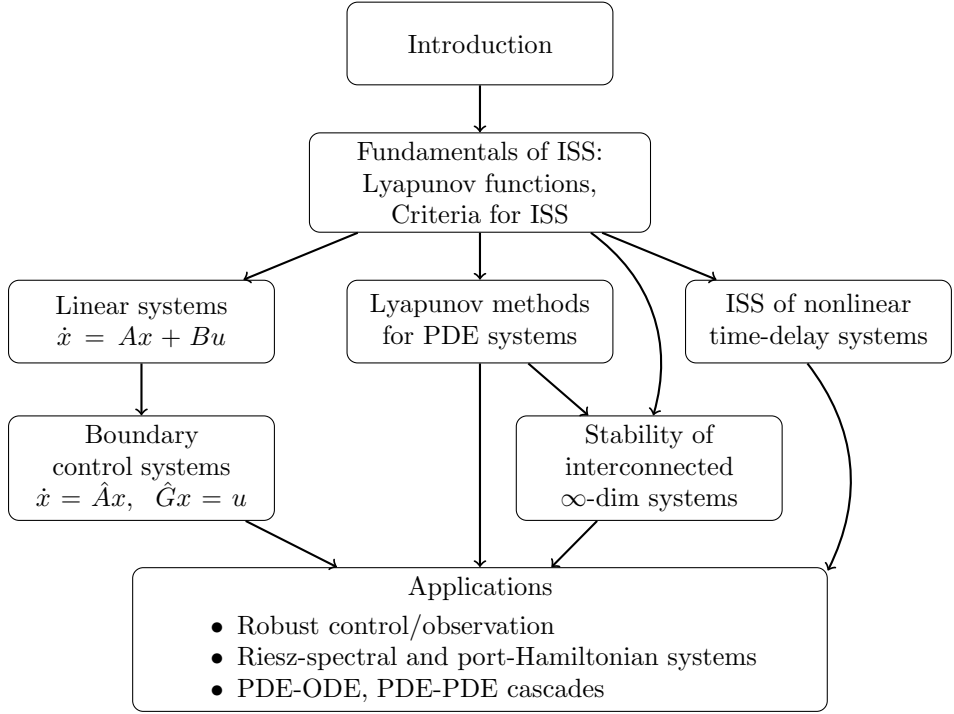

\subsection{ISS of finite-dimensional systems}
\label{sec:ISS-finite-dim-systems}

To motivate the notion of input-to-state stability, consider a linear system with additive external inputs:
\index{control system!linear}
\begin{eqnarray}
\dot{x} = Ax + Bu,
\label{eq:LinSys}
\end{eqnarray}
where $A \in\R^{n\times n}$ and $B\in\R^{n\times m}$.

The mild solution of \eqref{eq:LinSys} subject to initial condition $x \in \R^n$ and any locally Lebesgue integrable input $u \in \Uc$ is given by
the variation of constants formula
\index{variation of constants formula}
\begin{eqnarray}
\phi(t,x,u)= e^{At}x + \int_0^t e^{A(t-s)} B u(s)ds,\qquad t\geq 0.
\label{eq:LinSys_ISS}
\end{eqnarray}

For $x\in\R^n$ we denote by $|x|$ the Euclidean norm of $x$.
For $A\in\R^{n\times n}$ denote by $\|A\|$ the induced matrix norm of $A$, i.e., $\|A\|:=\sup_{|x|=1}|Ax|$.

Assume that \eqref{eq:LinSys} is asymptotically stable for a zero input $u\equiv 0$, i.e. $A$ is a Hurwitz matrix.
Then it holds that $\|e^{At}\| \leq M e^{-\lambda t}$ for some $M>0$, $\lambda >0$ and all $t \geq 0$ and consequently
\begin{eqnarray}
\label{eqn:ODEs_linear_ISS_Estimate}
|\phi(t,x,u)| &\leq& |e^{At}x| + \int_0^t \|e^{A(t-s)}\| |B u(s)|ds \nonumber\\
                            &\leq& Me^{-\lambda t}|x| + M \int_0^t e^{-\lambda (t-s)} ds \|B\| \|u\|_{\infty}\nonumber\\
                                                        &\leq& \underbrace{Me^{-\lambda t}|x|}_{\beta(|x|,t)} + \underbrace{\frac{M}{|\lambda|} \|B\| \|u\|_{\infty}}_{\gamma(\|u\|_{\infty})}.                                                        
\end{eqnarray}
Here the term $\beta(|x|,t):=Me^{-\lambda t}|x|$ describes the transient behavior of the system \eqref{eq:LinSys} and the term $\gamma(\|u\|_{\infty}):=\frac{M}{|\lambda|} \|B\| \|u\|_{\infty}$ describes the maximal asymptotic deviation of the state from the equilibrium, called the \emph{asymptotic gain} of the system.
In particular, inputs of a bounded magnitude induce only bounded asymptotic deviations of the system from the origin.

The estimate \eqref{eqn:ODEs_linear_ISS_Estimate} has been derived under an assumption that 
\eqref{eq:LinSys} is globally asymptotically stable if there are no disturbances acting at the system ($u\equiv 0$).
For nonlinear systems this requirement does not guarantee robustness with respect to sufficiently large disturbance inputs as we see 
on the following example
\begin{example}
\label{examp:0UGAS-not-ISS}
Consider the following nonlinear system
\begin{eqnarray}
\dot{x}(t) = -x(t) + x^2(t) u(t),
\label{eq:0UGAS-not-ISS}
\end{eqnarray} 
where $x(t) \in X := \R$ and  $u \in \Uc:= L_{\infty}(\R_+,\R)$.
This system is globally exponentially stable for a zero input. 
Now set $u\equiv 1$. Then for each $x_0 >1$ the solution of \eqref{eq:0UGAS-not-ISS} has a finite escape time, that is there is a time $t^*(x_0)$ so that 
$\lim_{t\to t^*(x_0)-0}|\phi(t,x_0,1)| = +\infty$.
Thus, \emph{nonlinear systems which are globally asymptotically stable in the absence of disturbances may be not robust against sufficiently large inputs}.
\end{example}

\begin{example}
\label{examp:Coupling_two_UGAS_systems_is_not_UGAS}
Consider the system \eqref{eq:0UGAS-not-ISS} interconnected with a linear exponentially stable system via coupling $u(t):=x_2(t)$:
\begin{eqnarray*}
\dot{x}_1(t) &=& -x_1(t) + x_1^2(t) x_2(t),\\
\dot{x}_2(t) &=& -x_2(t).
\end{eqnarray*} 
Arguing similarly to Example~\ref{examp:0UGAS-not-ISS} we conclude that the coupled system has finite escape times for initial conditions with a sufficiently large norm.
Hence, \emph{the simplest cascade interconnections of globally asymptotically stable systems are not globally asymptotically stable}.
\end{example}

Examples~\ref{examp:Coupling_two_UGAS_systems_is_not_UGAS} and \ref{examp:0UGAS-not-ISS} indicate that in spite of the importance of classical global asymptotic stability theory with its powerful Lyapunov methods, this concept is far too weak for the study of robustness and stability analysis of coupled control systems. 
Input-to-state stability, defined in the next section as a nonlinear extension of the property \eqref{eqn:ODEs_linear_ISS_Estimate}, constitutes a notion which does not have drawbacks indicated in Examples~\ref{examp:0UGAS-not-ISS}, \ref{examp:Coupling_two_UGAS_systems_is_not_UGAS} and brings a bunch of efficient tools for analysis of stability and robustness.

\subsection{Input-to-state stability}
\label{sec:Systems-and-ISS-def}

We start with a general definition of a control system.
\begin{definition}
\label{Steurungssystem}
Consider the triple $\Sigma=(X,\Uc,\phi)$ consisting of 
\begin{enumerate}[(i)]  
    \item A normed vector space $(X,\|\cdot\|_X)$, called the \emph{state space}, endowed with the norm $\|\cdot\|_X$.
    \item A normed vector \emph{space of inputs} $\Uc \subset \{u:\R_+ \to U\}$          
endowed with a norm $\|\cdot\|_{\Uc}$, where $U$ is a normed vector \emph{space of input values}.
We assume that the following two axioms hold:
                    
\emph{The axiom of shift invariance}: for all $u \in \Uc$ and all $\tau\geq0$ the time
shift $u(\cdot + \tau)$ belongs to $\Uc$ with \mbox{$\|u\|_\Uc \geq \|u(\cdot + \tau)\|_\Uc$}.

\emph{The axiom of concatenation}: for all $u_1,u_2 \in \Uc$ and for all $t>0$ the concatenation of $u_1$ and $u_2$ at time $t$, defined by
\begin{equation}
\ccat{u_1}{u_2}{t}(\tau):=
\begin{cases}
u_1(\tau), & \text{ if } \tau \in [0,t], \\ 
u_2(\tau-t),  & \text{ otherwise},
\end{cases}
\label{eq:Composed_Input}
\end{equation}
belongs to $\Uc$.

    \item A map $\phi:D_{\phi} \to X$, $D_{\phi}\subseteq \R_+ \times X \times \Uc$ (called \emph{transition map}), such that for all $(x,u)\in X \tm \Uc$ it holds that $D_{\phi} \cap \big(\R_+ \times \{(x,u)\}\big) = [0,t_m)\tm \{(x,u)\} \subset D_{\phi}$, for a certain $t_m=t_m(x,u)\in (0,+\infty]$.
		
		The corresponding interval $[0,t_m)$ is called the \emph{maximal domain of definition} of $t\mapsto \phi(t,x,u)$.
		
\end{enumerate}
The triple $\Sigma$ is called a \emph{(control) system}, if the following properties hold:

\begin{sysnum}
    \item\label{axiom:Identity} \emph{The identity property:} for every $(x,u) \in X \times \Uc$
          it holds that $\phi(0, x,u)=x$.
    \item \emph{Causality:} for every $(t,x,u) \in D_\phi$, for every $\tilde{u} \in \Uc$, such that $u(s) =
          \tilde{u}(s)$ for all $s \in [0,t]$ it holds that $[0,t]\tm \{(x,\tilde{u})\} \subset D_\phi$ and $\phi(t,x,u) = \phi(t,x,\tilde{u})$.
    \item \label{axiom:Continuity} \emph{Continuity:} for each $(x,u) \in X \times \Uc$ the map $t \mapsto \phi(t,x,u)$ is continuous on its maximal domain of definition.
        \item \label{axiom:Cocycle} \emph{The cocycle property:} for all
                  $x \in X$, $u \in \Uc$, for all $t,h \geq 0$ so that $[0,t+h]\tm \{(x,u)\} \subset D_{\phi}$, we have
$\phi(h,\phi(t,x,u),u(t+\cdot))=\phi(t+h,x,u)$.
\end{sysnum}

\end{definition}

%

This class of systems encompasses control systems generated by 
evolution partial differential equations (PDEs), abstract differential
equations in Banach spaces, time-delay systems, ordinary differential equations (ODEs), switched systems, as well as important classes of well-posed coupled systems consisting of arbitrary number of finite- and infinite-dimensional components, with both in-domain and boundary couplings.
In a certain sense, Definition~\ref{Steurungssystem} is a direct generalization and a unification of the concepts of strongly continuous nonlinear semigroups (see \cite{Sho13}) with abstract linear control systems considered, e.g., in \cite{Wei89b}.
Such an abstract definition of a control system is frequently used within the system-theoretic community at least since 1960's \cite{KFA69, Wil72}. 
In a similar spirit, even more general system classes can be considered, containing output maps, time-variant dynamics, possibility for a solution to jump at certain time instants, etc., see \cite{Kar07a, KaJ11}.


\begin{definition}
\label{def:FC_Property} 
We say that a control system (as introduced in Definition~\ref{Steurungssystem}) is \emph{forward complete (FC)}, if 
$D_\phi = \R_+ \tm X\tm\Uc$, that is for every $(x,u) \in X \times \Uc$ and for all $t \geq 0$ the value $\phi(t,x,u) \in X$ is well-defined.
\end{definition}

For a wide class of control systems boundedness of a solution implies the possibility to prolong it to a larger interval, see \cite[Chapter 1]{KaJ11b}. Next we formulate this property for abstract systems:
\begin{definition}
\label{def:BIC} 
We say that a system $\Sigma$ satisfies the \emph{boundedness-implies-continuation (BIC) property} if for each
$(x,u)\in X \tm \Uc$ such that the maximal existence time $t_{m}=t_m(x,u)$ is finite, and for all $M > 0$, there exists $t \in [0,t_{\max})$ with $\|\phi(t,x,u)\|_X>M$.
\end{definition}


\begin{Ass}
\label{ass:always-forward-complete}
From now on, for simplicity we consider only forward-complete control systems, if the contrary is not mentioned explicitly.
\end{Ass}

For two sets $X,Y$ denote by $C(X,Y)$ the linear space of continuous functions, mapping $X$ to $Y$.
For the formulation of stability properties we use the standard classes of comparison functions, introduced next:
\index{comparison!functions}
\index{class!$\K$}
\index{class!$\LL$}
\index{class!$\KL$}
\index{class!$\PD$}
\index{class!$\Kinf$}
\index{positive-definite function}
\index{function!positive definite}
\begin{equation*}
\begin{array}{ll}
{\PD} &:= \left\{\gamma \in C(\R_+): \gamma(0)=0 \mbox{ and } \gamma(r)>0 \mbox{ for } r>0  \right\} \\
{\K} &:= \left\{\gamma \in \PD : \gamma \mbox{ is strictly increasing}   \right\}\\
{\K_{\infty}}&:=\left\{\gamma\in\K: \gamma\mbox{ is unbounded}\right\}\\
{\LL}&:=\{\gamma\in C(\R_+): \gamma\mbox{ is strictly decreasing with } \lim\limits_{t\rightarrow\infty}\gamma(t)=0 \}\\
{\KL} &:= \left\{\beta \in C(\R_+\times\R_+,\R_+): \beta(\cdot,t)\in{\K},\ \forall t \geq 0, \  \beta(r,\cdot)\in {\LL},\ \forall r > 0\right\}
\end{array}
\end{equation*}
Functions of class $\PD$ are also called \textit{positive definite functions}.
An up-to-date compendium of results concerning comparison functions can be found in \cite{Kel14}.

We proceed to the main concept for this survey:
\begin{definition}
\label{def:ISS}
System $\Sigma=(X,\Uc,\phi)$ is called \emph{(uniformly)  input-to-state stable
(ISS)}, if there exist $\beta \in \KL$ and $\gamma \in \Kinf$ 
such that for all $ x \in X$, $ u\in \Uc$ and $ t\geq 0$ it holds that
\begin {equation}
\label{iss_sum}
\| \phi(t,x,u) \|_{X} \leq \beta(\| x \|_{X},t) + \gamma( \|u\|_{\Uc}).
\end{equation}
\end{definition}

The uniformity in Definition~\ref{def:ISS} means that the decay rate described by the function $\beta$ depends on the norm of the initial condition only, and does not depend on the initial condition itself as well as on the input. Non-uniform ISS-like notions are briefly considered in Section~\ref{sec:WeakISS}.

We consider also a weaker property
\begin{definition}
\label{def:0-UGAS}
System $\Sigma=(X,\Uc,\phi)$ is called \emph{uniformly globally asymptotically stable at zero (0-UGAS)}, if there exists $\beta \in \KL$ 
such that for all $ x \in X$ and all $ t\geq 0$ it holds that
\begin {equation}
\label{eq:0UGAS}
\| \phi(t,x,0) \|_{X} \leq \beta(\| x \|_{X},t).
\end{equation}
\end{definition}
Substituting $u:=0$ into \eqref{iss_sum} we obtain that ISS systems are always 0-UGAS. The converse is not true, as the 0-UGAS property does not  even imply forward completeness of the system with inputs, as can be seen in Example~\ref{examp:0UGAS-not-ISS}.

For finite-dimensional systems, i.e. if $X = \R^n$, the ISS property does not depend on the choice of the norm in $\R^n$
(as all norms in $\R^n$ are equivalent):
\begin{proposition}
\label{prop:ISS-in-finite-dimensions-does-not-depend-on-norm} 
Let $|\cdot|_1$, $|\cdot|_2$ be two norms in $\R^n$ and let $X_i := (\R^n,|\cdot|_i)$, $i=1,2$.
If a control system $(X_1,\Uc,\phi)$ is ISS, then $(X_2,\Uc,\phi)$ is also ISS.
\end{proposition}


If $X$ is an infinite-dimensional space, then the claim of Proposition~\ref{prop:ISS-in-finite-dimensions-does-not-depend-on-norm} 
does not hold, as the following simple and rather academical example shows:
\begin{example}
\label{examp:ISS-wrt-different-spaces} 
By $\ell^p$, $p\in[1,\infty]$, we denote the Banach space of all real sequences $x = (x_i)_{i\in\N}$ with finite $\ell^p$-norm $|x|_p<\infty$, where $|x|_p = (\sum_{i=1}^{\infty}|x_i|^p)^{1/p}$ for $p < \infty$ and $|x|_{\infty} = \sup_{i\in\N}|x_i|$.

Consider the following nonlinear system, consisting of a countable number of same components.
\begin{eqnarray}
\label{eq:Full_system}
\dot{x}_i = -x_i^3,\quad i\in\N.
\end{eqnarray}
By a rather elementary but lengthy argument (which we omit) one can show that this system is well-posed with $X=\ell_p$ for any $p\in[1,+\infty]$, it is 0-UGAS for $X=\ell_\infty$, but it is not 0-UGAS for $X=\ell_p$ with any $p\in [1,\infty)$.
\end{example}

The inequality \eqref{iss_sum} tells in particular that the trajectory $\phi(\cdot,x,u)$ is bounded as long as $u$ is bounded.
A simple but important practical consequence of ISS is the \emph{convergent input convergent state (CICS)} property which is a kind of a folklore in ISS community:
\begin{proposition}
\label{prop:Converging_input_uniformly_converging_state}
Let $\Sigma = (X,\Uc,\phi)$ be an ISS control system. Then for any $u\in\Uc$ such that $\lim_{t \to \infty}\|u(\cdot +t)\|_{\Uc} = 0$, it holds that
\begin{eqnarray}
\lim_{t \to \infty}\sup_{\|x\|_X\leq r}\|\phi(t,x,u)\|_X = 0, \mbox{ for any } r>0.
\label{eq:Uniform-convergence}
\end{eqnarray}
\end{proposition}
%
%
%

For a normed linear space $W$ and any $r>0$ denote $B_{r,W} :=\{u \in W: \|u\|_W < r\}$ (the open ball of radius $r$ around $0$ in $W$).
If $W$ is the state space $X$, we write simply $B_r$ instead of $B_{r,X}$.

The local counterpart of the ISS property is
\begin{definition}
\label{def:LISS}
System $\Sigma=(X,\Uc,\phi)$ is called \emph{ (uniformly) locally input-to-state stable
(LISS)}, if there exist $\beta \in \KL$, $\gamma \in \Kinf$ 
and $r>0$ such that the inequality \eqref{iss_sum} holds for all $ x \in
\overline{B_r}$, $u\in \overline{B_{r,\;\Uc}}$ and $ t\geq 0$.
\end{definition}


\begin{figure}[ht]
    \centering
    \hspace{-5mm}
\begin{subfigure}{0.44\textwidth}
\centering
\begin{tikzpicture}[scale = 0.85, transform shape]
\draw[->,thick] (-0.7,0) --(5,0);
\draw[->,thick] (0,-0.5) --(0,4.5);

\node(trans) at (2.5,1.5) {\small\color{red} $\beta(\|x\|_X,t)$};

\node(x) at (1.4,0.6) {\small $\|\phi(t,x,u)\|_X$};
\node(t) at (5,-0.3) {\small $t$};


\draw[thick, domain=0:5]  plot (\x, {2*exp(-0.6*\x)+2.5*\x*exp(-1.5*\x)});

\draw[dashed, thick, red, domain=0:5]  plot (\x, {3*exp(-0.5*\x)});
\end{tikzpicture}
    \caption{Typical solution of an ISS system with $u\equiv 0$}
\end{subfigure}
\quad\ \ 
\begin{subfigure}{0.48\textwidth}
\centering
\begin{tikzpicture}[scale = 0.85, transform shape]
\draw[->,thick] (-0.7,0) --(5,0);
\draw[->,thick] (0,-0.5) --(0,4.5);
\draw[dashed,thick] (0,1) --(5,1);

\node(gain) at (5.7,1) {\small $\gamma(\|u\|_{\Uc})$};
\node(trans) at (0.8,1.4) {\small $\beta(\|x\|_X,t)$};
\node[red] (ISS_sum) at (3,2.8) {\small $\beta(\|x\|_X,t)+\gamma(\|u\|_{\Uc})$};

\node(x) at (1.8,0.4) {\small $\|\phi(t,x,u)\|_X$};
\node(t) at (5,-0.3) {\small $t$};

\draw[-,thick] (0,2) to [out=90,in=130](1,2.5) to [out=-50,in=145](2,2)
to [out=-30,in=145](3,0.2) to [out=-30,in=145](4,1.2)
to [out=-40,in=145](5,0.8);

\draw[dashed, thick,  domain=0:5]  plot (\x, {3*exp(-0.5*\x)});
\draw[dashed, thick, red, domain=0:5]  plot (\x, {1+3*exp(-0.5*\x)});
\end{tikzpicture}
    \caption{Typical solution of an ISS system with $u\not\equiv 0$}
\end{subfigure}
\end{figure}

One of the most common choices for $\Uc$ is the space $\Uc:=PC_b(\R_+,U)$ of globally bounded piecewise-continuous functions which are right-continuous, with the norm $\| \cdot \|_{\Uc} := \sup\limits_{0 \leq s \leq \infty} \|u(s)\|_U$.
In this case the causality property of $\Sigma$ justifies the following equivalent definition of the (L)ISS property, which is often used in the literature (note the difference to the classical definition of ISS: the supremum is taken only over the interval $[0,t]$):
\begin{proposition}{\cite[Proposition 1]{DaM13}}
\label{prop:Causality_Consequence}
Let $\Sigma:=(X,\Uc,\phi)$ be a control system with $\Uc:=PC_b(\R_+,U)$. Then $\Sigma$ is ISS iff 
$\exists \beta \in \KL$ and $\gamma \in \Kinf$, such that
the following inequality holds for all $x\in X$, all $u\in\Uc$ and all $t\geq 0$:
\begin {equation}
\label{iss_sum_equiv}
\begin {array} {lll}
\| \phi(t,x,u) \|_{X} \leq \beta(\| x \|_{X},t) + \gamma( \sup\limits_{0 \leq s \leq t} \|u(s)\|_U).
\end {array}
\end{equation}
\end{proposition}


\section{Fundamental properties of ISS systems}
\label{sec:Fundamental-properties-ISS-systems}

We would like to start our exposition with several results, which are valid for a broad class of infinite-dimensional systems introduced in 
Definition~\ref{Steurungssystem}.
We discuss an ISS superposition theorem, which provides the insights into the essence of ISS and paves the way to the proof of such important results as non-coercive ISS Lyapunov theorem (see Section~\ref{sec:ISS-Lyapunov-theory}) and small-gain theorem in the trajectory formulation, discussed in Section~\ref{sec:Interconnected_systems}.
Next, we present direct and converse ISS Lyapunov theorems, which are an indispensable tool for ISS analysis of nonlinear systems and will be a basis for subsequent sections.

\subsection{ISS superposition theorem}
\label{sec:ISS-superposition-theorem}

One of important results in finite-dimensional dynamical systems theory is that global asymptotic stability defined as a combination of global attractivity and local stability, is equivalent to the existence of a $\KL$-bound for the transition map. 
In this section, we show an analogous result for the ISS property.

\subsubsection{General results using superposition techniques}

Let us give general results establishing ISS by superposing local and global properties.

\begin{definition}
\label{def:RobustEquilibrium_Undisturbed}
Consider a system $\Sigma=(X,\Uc,\phi)$ with \emph{equilibrium point} $0\in X$, that is $\phi(t,0,0) = 0$ for all $t \geq 0$.
We say that \emph{$\phi$ is continuous at the equilibrium point (CEP)} if  for every $\eps >0$ and for any $h>0$ there is a 
$\delta = \delta (\eps,h)>0$, so that 
\begin{eqnarray}
 t\in[0,h],\ \|x\|_X \leq \delta,\ \|u\|_{\Uc} \leq \delta \; \Rightarrow \;  \|\phi(t,x,u)\|_X \leq \eps.
\label{eq:RobEqPoint}
\end{eqnarray}
In this case we  say also that $\Sigma$ has the CEP property.
\end{definition}

Forward completeness alone does not imply in general the existence of any uniform bounds on the trajectories emanating from  bounded balls which are subject to uniformly bounded inputs. Systems exhibiting such bounds deserve a special name.
\begin{definition}
\label{def:BRS}
We say that \emph{$\Sigma=(X,\Uc,\phi)$ has bounded reachability sets (BRS)}, if for any $C>0$ and any $\tau>0$ it holds that 
\[
\sup\big\{
\|\phi(t,x,u)\|_X \midset \|x\|_X\leq C,\ \|u\|_{\Uc} \leq C,\ t \in [0,\tau]\big\} < \infty.
\]
\end{definition}

In the next definitions the counterparts of \q{stability} and \q{attractivity} concepts are introduced for control systems with inputs.
\begin{definition}
\label{def:ULS} 
System $\Sigma=(X,\Uc,\phi)$ is called 
\begin{itemize}
    \item[(i)] \emph{uniformly locally stable (ULS)}, if there exist $ \sigma \in\Kinf$, $\gamma
          \in \Kinf$ and $r>0$ such that for all $ x \in \overline{B_r}$ and all $ u
          \in \overline{B_{r,\Uc}}$:
\begin{equation}
\label{GSAbschaetzung}
\left\| \phi(t,x,u) \right\|_X \leq \sigma(\|x\|_X) + \gamma(\|u\|_{\Uc}) \quad \forall t \geq 0.
\end{equation}

    \item[(ii)] \emph{uniformly globally stable (UGS)}, if there exist $ \sigma \in\Kinf$, 
$\gamma \in \Kinf$ such that for all $ x \in X$ and all $ u\in\Uc$ the estimate \eqref{GSAbschaetzung} holds.
\end{itemize}
\end{definition}


%

\begin{definition}{(see \cite[Definition 8]{MiW18b} and \cite{Mir19b})}
\label{def:bULIM}
We say that $\Sigma=(X,\Uc,\phi)$ has the \emph{uniform limit property on bounded sets (bULIM)}, if there exists
    $\gamma\in\Kinf$ so that for every $\eps>0$ and for every $r>0$ there is a $\tau = \tau(\eps,r)$ such that 
\begin{eqnarray}
\|x\|_X \leq r \ \wedge \ \|u\|_\Uc \leq r \qrq \exists t\leq\tau:\  \|\phi(t,x,u)\|_X \leq \eps + \gamma(\|u\|_{\Uc}).
\label{eq:bULIM_ISS_section}
\end{eqnarray}
We say that $\Sigma=(X,\Uc,\phi)$ has the \emph{uniform limit property (ULIM)}, if 
there is $\gamma\in\Kinf$ so that for every $\eps>0$ and for every $r>0$ there exists a $\tau = \tau(\eps,r)$ such that 
the following stronger implication holds:
\begin{eqnarray}
\|x\|_X \leq r \ \wedge \ u\in\Uc \qrq \exists t\leq\tau:\  \|\phi(t,x,u)\|_X \leq \eps + \gamma(\|u\|_{\Uc}).
\label{eq:ULIM_ISS_section}
\end{eqnarray}
\end{definition}

The bULIM property shows that all trajectories emanating from a ball $B_r$ and subject to inputs from $B_{r,\Uc}$ intersect the $\varepsilon$-neighborhood of the ball $B_{\gamma(\|u\|_{\Uc}),\Uc}$ before a time $\tau(\varepsilon,r)$, which does not depend on a particular solution. After intersecting this neighborhood, the trajectories may well leave it, but then they have to visit it again and again.

\begin{remark}
\label{rem:Motivation_for_ULIM_Property} 
The bULIM and ULIM properties are motivated by the weaker concept of the limit property, introduced in \cite{SoW96}, which in turn is an extension to the systems with inputs of the weak attractivity property, playing an important role in the classical dynamical systems theory \cite{Bha66, BhS02}.
\end{remark}

Finally, we introduce a concept which is stronger than ULIM, and reflects the classical input-output stability.
\begin{definition}{\cite[p. 1285]{SoW96}}
\label{def:UAG}
System $\Sigma=(X,\Uc,\phi)$ has the \emph{uniform asymptotic gain (UAG) property}, if there
          exists a
          $ \gamma \in \Kinf$ such that for all $ \eps, r
          >0$ there is a $ \tau=\tau(\eps,r) < \infty$ such
          that for all $u \in \Uc$ and all $x \in B_{r}$
\begin{equation}    
\label{UAG_Absch}
t \geq \tau \quad \Rightarrow \quad \|\phi(t,x,u)\|_X \leq \eps + \gamma(\|u\|_{\Uc}).
\end{equation}    
\end{definition}

The following characterization of ISS has been shown in \cite[Theorem 5]{MiW18b} and was slightly improved in \cite{Mir19b}. This characterization is motivated by the classical critetia for ISS of ODEs shown in \cite{SoW95,SoW96}.
\begin{theorem}{(\emph{ISS Superposition theorem})}
\label{thm:UAG_equals_ULIM_plus_LS}
Let $\Sigma=(X,\Uc,\phi)$ be a forward complete control system. The following statements are equivalent:
\begin{itemize}
    \item[(i)] $\Sigma$ is ISS.
    \item[(ii)] $\Sigma$ is UAG, CEP, and BRS.
    \item[(iii)] \label{cond:ULIM_ULS_BRS_is_ISS} $\Sigma$ is bULIM, ULS, and BRS.
\end{itemize}
\end{theorem}

\begin{proof}
Assume that $\Sigma$ is ISS. Then for certain $\beta\in\KL$, $\gamma\in\Kinf$ it holds that
\[
t\geq 0\ \wedge\ x\in X\ \wedge\ u\in\Uc  \qrq \| \phi(t,x,u) \|_{X} \leq \beta(\| x \|_{X},0) + \gamma( \|u\|_{\Uc}).
\]
Defining $\sigma(r):=\beta(r,0)$, we see that $\Sigma$ is uniformly globally stable and in particular it is ULS, and has BRS and CEP properties.

Now pick any $r>0$ and any $\varepsilon>0$ and define $\tau(r,\varepsilon)$ as the solution of the equation $\beta(r,\tau) = \varepsilon$, if such a solution exists and $\tau:=0$ otherwise. We have:
\begin{eqnarray*}
 t\geq \tau  \ \wedge\ \|x\|_X\leq r \ \wedge\ u\in\Uc \qrq \| \phi(t,x,u) \|_{X} \leq \beta(r,\tau) + \gamma( \|u\|_{\Uc}) \leq \varepsilon + \gamma( \|u\|_{\Uc}),
\end{eqnarray*}
which shows the UAG property (and thus also ULIM property).

The above argument shows the implications (i) $\Rightarrow$ (ii) and (i) $\Rightarrow$ (iii).
Proof of converse implications is more involved and can be found in \cite[Theorem 5]{MiW18b} and in \cite{Mir19b}. 
\end{proof}

Often it is much harder to show the ISS property directly, and it can be much easier to show the properties, listed in items (ii) and (iii) of Theorem~\ref{thm:UAG_equals_ULIM_plus_LS}. This underlines the importance of Theorem~\ref{thm:UAG_equals_ULIM_plus_LS}.
Characterizations of ISS pave the way to the proof of other important results, such as small-gain theorems \cite{DRW07, Mir19b, CaT09, DaK13}, Lyapunov-Razumikhin theory for time-delay systems (see Theorem~\ref{thm:Lyapunov-Razumikhin theorem} in this survey and more in \cite{DKM12, Tee98}), non-coercive ISS Lyapunov theorems (see Theorem~\ref{t:ISSLyapunovtheorem}), relations between ISS and nonlinear $L_2 \to L_2$ stability \cite{Son98}, connections to semiglobal ISS \cite{ASW00b}, to name a few applications.

Furthermore, Theorem~\ref{thm:UAG_equals_ULIM_plus_LS} shows that ISS is indeed a synergy of stability and some kind of a classical input-output stability, which makes ISS a framework well-suited for the description of stability properties not only for closed systems (i.e. the systems with inputs), which are the primary object of investigation of the classical dynamical systems theory, but also for open systems which interact with the environment by means of the inputs. Another theory, which is successful for the analysis of open systems is a dissipative systems theory, see \cite{Wil72, AMP16}.

Theorem~\ref{thm:UAG_equals_ULIM_plus_LS} is tight in the sense, that if either ULS or BRS or uniformity in the limit property is dropped from the item (iii) of Theorem~\ref{thm:UAG_equals_ULIM_plus_LS}, then the result of Theorem~\ref{thm:UAG_equals_ULIM_plus_LS} is no more valid, see \cite{Mir16, MiW18b} for corresponding examples.

At the same time, for more special classes of systems tighter counterparts of Theorem~\ref{thm:UAG_equals_ULIM_plus_LS} can be obtained, see \cite{MiW18b} for evolution equations in Banach spaces, \cite{SoW96} and \cite[Section~VIII]{MiW18b} for ordinary differential equations, as well as \cite{MiW17e} for preliminary results on time-delay systems. 

\subsubsection{Criteria for ISS of ordinary differential equations}

Let $X:=\R^n$, $U:=\R^m$ and $\Uc$ be the space of globally
essentially bounded functions endowed with the essential supremum
norm (here $\mu$ denotes the Lebesgue measure):
\[
\Uc:=L_{\infty}(\R_+,\R^m),\qquad
\|u\|_{\infty}:=\esssup_{t\geq 0}|u(t)| = \inf_{D \subset \R_+,\ \mu(D)=0} \sup_{t \in \R_+ \backslash D} |u(t)|.
\]
For $f: \R^n\times \R^m \to \R^n$ consider the system
\begin{eqnarray}
\dot{x} = f(x,u).
\label{eq:ODE_System}
\end{eqnarray}

First we introduce the concept of a solution for the system \eqref{eq:ODE_System}.
\begin{definition}
\label{def:Solution-of-ODE-system} 
For $\tau>0$ we say that $\zeta:[0,\tau] \to \R^n$ is a \emph{solution of \eqref{eq:ODE_System} on $[0,\tau)$} with an initial condition $x_0 \in \R^n$ and an input $u\in\Uc$, if $\zeta$ is absolutely continuous on $[0,\tau]$, satisfies the initial condition $\zeta(0) = x_0$ and the equation $\dot{\zeta}(t) = f(\zeta(t),u(t))$ holds almost everywhere on $[0,\tau)$.

We say that $\zeta:\R_+ \to \R^n$ is a \emph{solution of \eqref{eq:ODE_System} on $\R_+$} with an initial condition $x_0 \in \R^n$ and an input $u\in\Uc$, if $\zeta$ is a solution (subject to an initial condition $x_0 \in \R^n$ and an input $u\in\Uc$) of \eqref{eq:ODE_System} on $[0,s)$ for each $s>0$.
\end{definition}

Assuming that for each $x_0\in \R^n$ and $u\in\Uc$ there is a unique global solution $\phi(\cdot,x_0,u)$ corresponding to the input $u$ and the initial condition $x(0)=x_0$, one can see that $\Sigma_{ODE}:=(\R^n,\Uc,\phi)$ is a control system according to Definition~\ref{Steurungssystem}.

We need the following weaker counterpart of the bULIM property:
\begin{definition}{(see \cite{SoW96})}
\label{def:LIM}
We say that a forward complete control system $\Sigma=(X,\Uc,\phi)$ has the \emph{limit property (LIM)}, if there exists
    $\gamma\in\Kinf$ so that for every $x\in X$, every $u\in\Uc$ and for every $\varepsilon>0$ there
    exists a $t = t(\eps,x,u)$ such that 
\begin{eqnarray}
\|\phi(t,x,u)\|_X \leq \eps + \gamma(\|u\|_{\Uc}).
\label{eq:LIM_ISS_section}
\end{eqnarray}
\end{definition}
Please note that forward completeness is needed already in order to define the LIM property. If forward completeness is not known, then the  definition of LIM property has to be changed.
 
For ODEs tight characterizations of ISS in terms of stability and limit properties have been derived by Sontag and Wang in \cite{SoW96}. 
\begin{theorem}{\cite[Theorem 1]{SoW96}}
\label{thm:Characterization_in_n_dim} 
The following holds for a system \eqref{eq:ODE_System}:
\begin{center}
\eqref{eq:ODE_System} is ISS $\quad\Iff\quad$ \eqref{eq:ODE_System} is forward complete, ULS and has a LIM property.
\end{center}
\end{theorem}

\begin{proof}
Follows from Theorem~\ref{thm:UAG_equals_ULIM_plus_LS} by using that for \eqref{eq:ODE_System} the BRS property is equivalent to forward completeness by \cite[Proposition 5.1]{LSW96} and the bULIM property is equivalent to the LIM property by \cite[Proposition 13]{MiW18b}, which in turn essentially uses \cite[Corollary III.3]{SoW96}.
\end{proof}

The possibility to relax the BRS property to forward completeness as well as the uniform LIM property to the \q{plain} LIM property 
uses heavily the topological properties of ODE systems.
In finite dimensions, uniform and non-uniform notions are frequently equivalent due to local compactness of
the state space. In infinite dimensions, this uniformity becomes a requirement, see \cite{Mir16} for the examples showing that a \q{naive}
generalization of the equivalences in Theorem~\ref{thm:Characterization_in_n_dim} to infinite-dimensional systems is not possible. 

\subsection{ISS Lyapunov functions}
\label{sec:ISS-Lyapunov-theory}

The concept of a Lyapunov function is central for stability theory as the existence of a Lyapunov function is a certificate for the stability of a system and in many cases the only way to show stability of a nonlinear system is to construct a Lyapunov function for it. 
As we will see in this section, Lyapunov functions retain their significance in the context of input-to-state stability.

For a real-valued function $b:\R_+\to\R$ define the \emph{right-hand upper Dini derivative} at $t\in\R_+$ by
\begin{eqnarray*}
D^+b(t) := \mathop{\overline{\lim}} \limits_{h \rightarrow +0} \frac{b(t+h) - b(t)}{h}.
\end{eqnarray*}

Let $x \in X$ and $V$ be a real-valued function defined in a neighborhood of $x$. \emph{The Lie derivative 
of $V$ at $x$ corresponding to the input $u$ along the corresponding trajectory of $\Sigma$} is defined by
\begin{equation}
\label{ISS_LyapAbleitung}
\dot{V}_u(x):=D^+ V\big(\phi(\cdot,x,u)\big)\Big|_{t=0}=\mathop{\overline{\lim}} \limits_{t \rightarrow +0} {\frac{1}{t}\big(V(\phi(t,x,u))-V(x)\big) }.
\end{equation}


The interior of a domain $D \subset X$ we denote by $\intt(D)$.
\begin{definition}
\label{def:noncoercive_ISS_LF}
Consider $D \subset X$, such that $0 \in \intt(D)$. 
A continuous function $V:D \to \R_+$ is called a \emph{non-coercive LISS Lyapunov function} for the system $\Sigma = (X,\Uc,\phi)$,
 if there are $\psi_2,\alpha \in \Kinf$, $\sigma \in \K$ and $r>0$, such that $B_r\subset D$,
\begin{equation}
\label{LyapFunk_1Eig_nc_ISS}
0 < V(x) \leq \psi_2(\|x\|_X), \quad \forall x \in D\backslash\{0\}
\end {equation}
and the Lie derivative of $V$ along the trajectories of $\Sigma$ satisfies
\begin{equation}
\label{DissipationIneq_nc}
\dot{V}_u(x) \leq -\alpha(\|x\|_X) + \sigma(\|u\|_{\Uc})
\end{equation}
for all $x \in B_r$ and $u\in B_{r,\Uc}$.

If additionally there is a $\psi_1\in\Kinf$ so that 
\begin{equation}
\label{LyapFunk_1Eig_LISS}
\psi_1(\|x\|_X) \leq V(x) \leq \psi_2(\|x\|_X), \quad \forall x \in D,
\end {equation}
then $V$ is called a \emph{(coercive) LISS Lyapunov function} for the system $\Sigma = (X,\Uc,\phi)$.

If $D=X$, and $r=+\infty$, then $V$ is called an \emph{ISS Lyapunov function} (non-coercive or coercive respectively) for $\Sigma$.
\end{definition}
%

\begin{remark}
\label{rem:Non-coercivity_and_weak_nondegenerativity} 
By definition, coercive ISS Lyapunov functions are a subclass of non-coercive ISS Lyapunov functions (in the same way as linear systems we consider as a subclass of nonlinear systems).
\end{remark}

\begin{remark}
\label{rem:Lyapunov-functions-for-L-p} 
Although Definition~\ref{def:noncoercive_ISS_LF} of ISS Lyapunov functions is formally defined for arbitrary input spaces, for the spaces $\Uc=L_p(\R_+,U)$, where $p\in[1,+\infty)$ and $U$ is a Banach space, it is far too restrictive. For an ISS Lyapunov function concept, which is suitable for ISS analysis of control systems with integrable inputs we refer to \cite{Mir20}.
\end{remark}

\begin{remark}
\label{rem:dissipation inequality} 
The \emph{dissipation inequality} \eqref{DissipationIneq_nc} shows that an ISS Lyapunov function $V$ decreases along the trajectories as long as the norm of the state is large enough in comparison to the magnitude of the input (more precisely,  if $\|x\|_X>\alpha^{-1}\circ \sigma(\|u\|_{\Uc})$).
Definition~\ref{def:noncoercive_ISS_LF} is often called a \emph{dissipative formulation (or dissipative form) of the ISS Lyapunov function concept}. 
\end{remark}

%

\begin{remark}
\label{rem:Computation-of-dotV} 
Computation of $\dot{V}_u(x)$ by definition given in \eqref{ISS_LyapAbleitung}, requires the knowledge of  
the solution for future times, at least on a sufficiently small time-interval.
This definition is very general and fits perfectly for the ISS theory of a general class of infinite-dimensional systems. However, one of the main benefits of the Lyapunov theory is the possibility to check the stability of a system without knowledge of a solution. To retain this feature, for many important classes of systems it is possible to derive 
useful formulas for the computation of $\dot{V}_u(x)$, which do not involve the knowledge of the future dynamics.
For example, for \emph{ODE systems} \eqref{eq:ODE_System} with a Lipschitz continuous $V$, for almost all $x\in\R^n$ and all continuous $u\in\Uc$ it holds that $\dot{V}_u(x) = \nabla V(x) \cdot f(x,u(0))$, where $\nabla V$ is the gradient of $V$ and $\cdot$ is the standard inner product in $\R^n$.
For \emph{time-delay systems}, $\dot{V}_u(x)$ can be computed using the so-called \emph{Driver derivative}, see Section~\ref{sec:ISS-theory-for-TDS}.
For many \emph{PDE systems}, one can give simple formulas for $\dot{V}_u(x)$ for $x$ belonging to a dense subset of the state space $X$, and for other states the derivative can be estimated by using some kind of \emph{density arguments}, see, e.g., \cite{DaM13}.
\end{remark}

Alternatively,  (L)ISS Lyapunov functions can be defined in implication form:
\begin{definition}
\label{def:implicative-ISS-LF}
A continuous function $V:D \to \R_+$, $D \subset X$, $0 \in \intt(D)$ is called a \emph{ (coercive) LISS Lyapunov function in  implication form} for $\Sigma$, if there are functions $\psi_1,\psi_2 \in \Kinf$, $\chi \in \Kinf$, $r>0$ and  a positive definite function $\alpha$, such that
\eqref{LyapFunk_1Eig_LISS}
holds and 
for all $x \in B_r$ and all $u\in B_{r,\Uc}$ it holds that
\begin{equation}
\label{GainImplikation}
 \|x\|_X \geq \chi(\|u\|_{\Uc}) \  \Rightarrow  \ \dot{V}_u(x) \leq -\alpha(\|x\|_X).
\end{equation}
Function $\chi$ is called \emph{ISS Lyapunov gain} for $(X,\Uc,\phi)$.
If $D=X$, and $r=+\infty$, then $V$ is called an \emph{ISS Lyapunov function in implication form} for $\Sigma$.
\end{definition}

For ordinary differential equations (under reasonable regularity assumptions) a function $V:X\to\R_+$ is an ISS Lyapunov function in dissipative form (see Remark~\ref{rem:dissipation inequality}) if and only if it is an ISS Lyapunov function in  implication form, see \cite[Remark 2.4, p. 353]{SoW95}.

\begin{openprob}
\label{op:dissipative-and-implicative-ISS-LF} 
For infinite-dimensional systems, it is easy to see that if $V$ is an ISS Lyapunov function in dissipative form, then it is an ISS Lyapunov function in implication form. 
It is, however, an open question, whether the converse implication holds unless quite restrictive requirements on $\Sigma$ and Lyapunov function $V$ are imposed, see \cite[Theorem 3.4]{MiI16}.
\end{openprob}

The importance of ISS Lyapunov functions is due to the following basic result:
\begin{theorem}{\emph{(Direct Lyapunov theorem, \cite[Theorem 1]{DaM13}, \cite[Proposition 1]{MiI16})}}
\label{LyapunovTheorem}
Let $\Sigma = (X,\Uc,\phi)$ be a control system satisfying the BIC property, and let $x \equiv 0$ be its equilibrium point.
If $\Sigma$ possesses a coercive (L)ISS Lyapunov function (in either dissipative or implication form), then it is (L)ISS.
\end{theorem}

\begin{proof}
In \cite[Theorem 1]{DaM13} it was assumed that $\Sigma$ is forward complete, but the proof can be naturally adapted for a much weaker assumption of BIC property. The proof of this theorem in \cite{DaM13} is based on the so-called \emph{comparison principle}, see \cite[Lemma 3.2]{MiI16}.
\end{proof}

Construction of ISS Lyapunov functions is a challenging task, especially for systems with boundary inputs.  
Thus it is natural to ask, whether less restrictive non-coercive ISS Lyapunov functions are sufficient for ISS.
The next theorem, whose proof essentially exploits the characterizations of ISS developed in Theorem~\ref{thm:UAG_equals_ULIM_plus_LS}, shows that this is indeed the case, provided several further conditions hold, which have to be checked separately.

%
%
%
\begin{theorem}{\emph{(Direct non-coercive Lyapunov theorem, \cite[Theorem 3.7]{JMP18})}}
    \label{t:ISSLyapunovtheorem}
Let $\Sigma$ be a forward complete control system, which is CEP and BRS.    
If there exists a non-coercive ISS Lyapunov function for $\Sigma$, then $\Sigma$ is ISS.    
\end{theorem}


Note that requirements of CEP and BRS properties are necessary for the validity of Theorem~\ref{t:ISSLyapunovtheorem}, see examples in \cite{MiW19a}. Without these requirements however one can still verify certain stability properties for $\Sigma$, which are weaker than ISS, see \cite{MiW19b} for an analysis in the context of UGAS property, and \cite{JMP18} for such results in the ISS context. We state here only one of such results:
\begin{proposition}{(\cite[Proposition 8]{MiW18b})}
\label{prop:ncISS_implies_ULIM} 
Let $\Sigma=(X,\Uc,\phi)$ be a forward complete control system and assume
there exists a non-coercive ISS Lyapunov function for $\Sigma$. Then $\Sigma$ has the ULIM property.
\end{proposition}

\begin{openprob}
\label{op:other-formulations-of-noncoercive-theorems} 
Theorem~\ref{t:ISSLyapunovtheorem}, in contrast to Theorem~\ref{LyapunovTheorem}, gives an ISS criterion only for the global ISS property, and non-coercive ISS Lyapunov functions are formulated in dissipative form. Currently, there are no non-coercive Lyapunov results for the LISS property and for non-coercive ISS Lyapunov functions defined in implication form.
\end{openprob}


Applications of Lyapunov methods are very broad. We apply Lyapunov technique for linear infinite-dimensional systems 
in Section~\ref{sec:LFs_linear_systems}
and for linear and nonlinear PDEs in Section~\ref{sec:ISS_analysis_linear_nonlinear_PDEs_Lyapunov_methods}. Furthermore, ISS Lyapunov theory is very powerful for analysis of interconnected systems, as will be illustrated in Section~\ref{sec:Interconnected_systems}.

\subsection{Characterization of local ISS}
\label{sec:LISS-Characterization}

Consider infinite-dimensional evolution equations of the form
\begin{equation}
\label{InfiniteDim}
\dot{x}(t)=Ax(t)+f(x(t),u(t)),
\end{equation}
where $A: D(A)\subset X \to X$ generates a strongly continuous semigroup (also called $C_0$-semigroup)  $T(\cdot)$ of boun\-ded linear operators, $X$ is a Banach space, $U$ is a normed linear space of input values, and $f:X\times U \to X$.
As the space of admissible inputs, we consider the space $\Uc:=PC_b(\R_+,U)$ of globally bounded, piecewise continuous functions from $\R_+$ to $U$ which are right-continuous.
We refer to the excellent books \cite{CuZ95, Paz83, JaZ12} for an overview of the semigroup theory and of the theory of evolution equations.

\begin{definition}
\label{def:Mild-solution}
For a certain $\tau>0$ we call a function $x:[0,\tau] \to X$ a \emph{mild solution of \eqref{InfiniteDim} on $[0,\tau]$}, if $x  \in C([0,\tau],X)$ and it solves
the integral equation
\begin{equation}
\label{InfiniteDim_Integral_Form}
x(t)=T(t) x(0) + \int_0^t T(t-s) f\big(x(s),u(s)\big)ds. 
\end{equation}
Here the integral is understood in the sense of Bochner, see \cite[Chapter III, Sec. 3.7]{HiP00} or \cite[Appendix A]{JaZ12}, which is resolved to the standard Riemann integral of $X$-valued maps as what we integrate is a piecewise-continuous $X$-valued function.

We say that $x:\R_+\to X$ is a \emph{mild solution of \eqref{InfiniteDim} on $\R_+$}, if it is a mild solution of 
\eqref{InfiniteDim} on $[0,\tau]$ for all $\tau>0$.
\end{definition}

In this paper, we always consider mild solutions, which is a more usual choice. Another possibility would be to consider classical solutions. For relations between ISS w.r.t. mild solutions and ISS w.r.t. classical solutions a reader may consult \cite[Proposition 2.11]{Sch20}.

\begin{Ass}
\label{Assumption1} 
The nonlinearity $f$ satisfies the following properties:
\begin{itemize}
    \item[(i)] $f:X \times U \to X$ is Lipschitz continuous on bounded subsets of $X$, i.e.
$\forall C>0 \; \exists L_f(C)>0$, such that $\forall x,y \in B_C$, $\forall v \in B_{C,U}$, it holds that
\begin{eqnarray}
\|f(y,v)-f(x,v)\|_X \leq L_f(C) \|y-x\|_X.
\label{eq:Lipschitz}
\end{eqnarray}
    \item[(ii)] $f(x,\cdot)$ is continuous for all $x \in X$.
\end{itemize}
\end{Ass}

\begin{remark}
\label{rem:On-Lipschitz-continuity} 
In \cite{Mir16} and \cite{MiW18b} a somewhat stronger requirement on $f$ has been posed, but it is not essential for the results in this section.
\end{remark}

Since $\Uc=PC_b(\R_+,U)$, Assumption~\ref{Assumption1} ensures that for any $x\in X$, $u\in\Uc$ the corresponding mild solution of \eqref{InfiniteDim} exists and is unique, according to a variation of the classical existence and uniqueness theorem \cite[Proposition 4.3.3]{CaH98}. We denote by $\phi(t,x,u)$ this solution at the moment $t\in\R_+$ associated with 
an initial condition $x \in X$ at $t=0$, and input $u \in \Uc$.
The conditions
($\Sigma$\ref{axiom:Identity})-($\Sigma$\ref{axiom:Cocycle}) are satisfied
by construction and thus \emph{$\Sigma:=(X,\Uc,\phi)$ is a control system according to Definition~\ref{Steurungssystem}}.

A local counterpart of the 0-UGAS notion is
\begin{definition}
\label{def:0-UAS}
System $\Sigma=(X,\Uc,\phi)$ is called \emph{uniformly asymptotically stable at zero (0-UAS)}, if there exist $\beta \in \KL$ and $r>0$ such that 
\begin {equation}
\label{eq:0UAS}
\|x\|_X\leq r \ \wedge \ t\geq 0 \qrq \| \phi(t,x,0) \|_{X} \leq \beta(\| x \|_{X},t).
\end{equation}
\end{definition}

In \cite[Theorem 4]{Mir16} it was shown that under mild assumptions on the nonlinearity of a system \eqref{InfiniteDim}, the LISS property is equivalent to the 0-UAS.
\begin{theorem}
\label{Characterization_LISS}
Let Assumption~\ref{Assumption1} hold, $f(0,0)=0$, and let \eqref{InfiniteDim} has a CEP property.
Furthermore, let there exist $\sigma \in \K$ and $r >0$ so that:
\begin{eqnarray*}
\|v\|_U \leq r \ \wedge \ \|x\|_X\leq r \qrq \|f(x,v)-f(x,0)\|_X \leq \sigma(\|v\|_U).
\end{eqnarray*}
Then for the system \eqref{InfiniteDim} the following properties are equivalent:
\begin{enumerate}
    \item[(i)] 0-UAS,
    \item[(ii)] Existence of a coercive locally Lipschitz continuous LISS Lyapunov function,
    \item[(iii)] LISS.
\end{enumerate}
\end{theorem}


\begin{openprob}
The proof of Theorem~\ref{Characterization_LISS} is based on converse Lyapunov theorems for the 0-UAS property, see \cite{Mir16}, and thus it does not give a precise estimate for the region in $X \times \Uc$ for which the LISS estimate \eqref{iss_sum} is valid (for given $\beta,\gamma$).
The estimation of the region in which the LISS estimate is valid is a hard problem, which has not been touched so far for infinite-dimensional systems. 
For ODE systems this problem has been approached by means of a numerical construction of 
a continuous piecewise affine LISS Lyapunov functions using the linear programming method \cite{LBG15}.
\end{openprob}

\subsection{Converse Lyapunov theorems}
\label{sec:Converse-ISS-Lyapunov-Theorem}

In most cases, construction of an ISS Lyapunov function is the only practical way to prove ISS of a given nonlinear system.
However, before we start searching for an ISS Lyapunov function of a given system, a natural question appears whether such a function exists at all. Results ensuring that this is the case are called \emph{converse Lyapunov theorems} and for classical asymptotic stability such results date back to the pioneering works of Kurzweil \cite{Kur56} and Massera \cite{Mas56}.
Converse Lyapunov theorems for input-to-state stability of general nonlinear ODE systems have been proved in an influential paper \cite{LSW96}. Motivated by \cite{LSW96}, in \cite{MiW17c} a corresponding result for evolution equations with bi-Lipschitz nonlinearities has been shown.


%
%

\begin{theorem}{\emph{(Converse ISS Lyapunov theorem, \cite[Theorem 5]{MiW17c})}}
\label{ISS_Converse_Lyapunov_Theorem}
Let $f:X \times U \to X$ be bi-Lipschitz continuous on bounded balls, i.e.
$f$ is Lipschitz continuous on bounded balls from the Banach space $(X\times U, \|\cdot\|_X+\|\cdot\|_U)$ to the
space $X$.

Then \eqref{InfiniteDim} is ISS if and only if there is a Lipschitz continuous on bounded balls coercive ISS Lyapunov function for \eqref{InfiniteDim}.
\end{theorem}

\begin{proof}
The proof of this result in \cite[Theorem 5]{MiW17c} is based on a reduction of \eqref{InfiniteDim} to a related system with disturbances, proposed in \cite{SoW95} and relies on a converse Lyapunov theorem for uniform global asymptotic stability of such systems, shown in \cite[Section 3.4]{KaJ11b}. 
\end{proof}

Theorem~\ref{ISS_Converse_Lyapunov_Theorem} is an existence result, and does not give a precise form of an ISS Lyapunov function. For linear systems, constructive converse Lyapunov theorems can be shown, see Section~\ref{sec:LFs_linear_systems}.

\begin{openprob}
For general nonlinear systems and even for general linear boundary control systems converse ISS Lyapunov theorems are not available now and constitute an important and complex open problem.
\end{openprob}

\subsection{Integral input-to-state stability}
\label{sec:iISS}

In spite of all advantages of the ISS framework, for some practical systems, input-to-state stability is too restrictive.
This is because ISS excludes systems whose state stays bounded as long as the magnitude of applied inputs and of initial states remains below a specific threshold, but becomes unbounded when the input magnitude or the magnitude of an initial state exceeds the threshold. Such behavior is frequently caused by saturation and limitations in actuation and processing rate. 
The idea of integral input-to-state stability (iISS) is to capture such nonlinearities \cite{Son98, ASW00}. 

In this section we assume that $\Sigma$ is a forward complete control system in the sense of Definition~\ref{Steurungssystem} with $\Uc$ which is a linear subspace of a space $L_{1,loc}(\R_+,U)$ of locally Bochner integrable $U$-valued functions.

\begin{definition}
\label{def:iISS}
A forward complete system $\Sigma=(X,\Uc,\phi)$ is called \emph{integral input-to-state stable (iISS)} if there exist $\theta \in \Kinf$, $\mu \in \K$ and $\beta \in \KL$ such that for all $(t, x, u) \in \R_+ \tm X \tm \Uc$ it holds that
\begin{equation}
\label{iISS_Estimate}
\|\phi(t,x,u)\|_X \leq \beta(\|x\|_X,t) + \theta\left(\int_0^t \mu(\|u(s)\|_U)ds\right).
\end{equation}
\end{definition}

If the integral in the right-hand side of \eqref{iISS_Estimate} diverges, we consider it as equal to $+\infty$.
%

In the ISS theory of ODE and delay systems as well as for many other classes of evolution equations a natural choice of the input space is $L_\infty$ space or the space $PC_b$ of piecewise continuous functions. For such choices of the input space the ISS estimate \eqref{iss_sum} provides an upper bound on the response of the system with respect to the maximal magnitude of the applied input. In contrast to that, \emph{integral ISS gives an upper bound of the response of the system with respect to a kind of the energy fed into the system, described by the integral in the right-hand side of \eqref{iISS_Estimate}.}

\begin{definition}
\label{def:BECS} 
We say that a forward complete system $\Sigma := (X,\Uc,\phi)$ has \emph{bounded energy-convergent state (BECS)} property if there is $\xi\in\K$ such that 
\begin{eqnarray}
\int_0^\infty \xi(\|u(s)\|_U)ds <\infty \qrq \forall r>0\ \ 
\lim_{t \to \infty}\sup_{\|x\|_X\leq r}\|\phi(t,x,u)\|_X = 0.
\label{eq:BECS} 
\end{eqnarray}
\end{definition}

Similarly to Proposition~\ref{prop:Converging_input_uniformly_converging_state}, we have the following basic properties of iISS systems, belonging to the folklore of the ISS theory:
\begin{proposition}
\label{prop:iISS-implies-bounded energy convergent state} 
Let $\Sigma$ be an iISS control system. Then $\Sigma$ is 0-UGAS and satisfies BECS property with $\xi:=\mu$.
\end{proposition}

{
}

\begin{remark}{(iISS and forward completeness)}
\label{rem:iISS-and-Forward-completeness} 
In this paper we take forward completeness of control systems for granted (Assumption~\ref{ass:always-forward-complete}). 
Alternatively, one could require in the iISS definition that the estimate \eqref{iISS_Estimate} holds only for $(t, x, u) \in D_\phi \subset \R_+ \tm X \tm \Uc$.
If $\Sigma$ satisfies the BIC property (see Definition~\ref{def:BIC}), then an easy argument shows that $\Sigma$ is forward complete.
The same argument could be applied for ISS property as well.

Finally, note that iISS is stronger than forward completeness together with 0-UGAS, even for nonlinear ODE systems, see an example in \cite[Section~V]{ASW00}.
\end{remark}


For linear systems with bounded input operators the notions of ISS and integral ISS are equivalent, see Theorem~\ref{thm:ISS-criterion-linear-systems-bounded-operators}. For nonlinear systems as well as for linear systems with unbounded input operators the situation is more complex.

\begin{definition}\label{def:iISSV}
Consider a control system $\Sigma:=(X,\Uc,\phi)$ with the input space $\Uc:=PC_b(\R_+,U)$.
A continuous function $V:X \to \R_+$ is called a \emph{non-coercive iISS Lyapunov function} for $\Sigma$, if there exist
$\psi_2 \in \Kinf$, $\alpha \in \PD$ and $\sigma \in \K$ 
such that 
\begin{equation}
\label{LyapFunk_1Eig_iISS_noncoercive}
0 < V(x) \leq \psi_2(\|x\|_X), \quad \forall x \in X\backslash\{0\}
\end{equation}
and Lie derivative of $V$ along the trajectories of the system \eqref{InfiniteDim} satisfies 
\begin{equation}
\label{DissipationIneq}
\dot{V}_u(x) \leq -\alpha(\|x\|_X) + \sigma(\|u(0)\|_U),\quad x\in X,\ u\in\Uc.
\end{equation}
If additionally there is a $\psi_1\in\Kinf$ so that 
\begin{equation}
\label{LyapFunk_1Eig_iISS_coercive}
\psi_1(\|x\|_X) \leq V(x) \leq \psi_2(\|x\|_X), \quad \forall x \in X,
\end {equation}
then $V$ is called a \emph{(coercive) iISS Lyapunov function}.
\end{definition}

Note that in the definition of an iISS Lyapunov function the decay rate $\alpha$ is a $\PD$-function, and in the definition of an ISS Lyapunov function $\alpha \in\Kinf$.


The direct Lyapunov theorem for the iISS property reads as follows:
\begin{proposition}{\cite[Proposition 1]{MiI16}}
\label{PropSufiISS}
Consider a control system $\Sigma{:=}(X,\Uc,\phi)$ with $\Uc:=PC_b(\R_+,U)$.
If there is a coercive iISS Lyapunov function for $\Sigma$, then $\Sigma$ is iISS.
\end{proposition}

\begin{remark}
\label{rem:Why-PC-functions-in-iISS} 
Please note that for general input spaces $\Uc$, allowed by Definition~\ref{Steurungssystem}, the term $u(0)$ in \eqref{DissipationIneq} may have no sense, e.g.,\ this is the case for $\Uc=L_\infty(\R_+,\R)$.
To allow for control systems with general input spaces the definition of an iISS Lyapunov function has to be changed. If we mimic the definition of the ISS Lyapunov function, and require in the definition of an iISS Lyapunov function instead of \eqref{DissipationIneq} a validity of the inequality
\begin{equation}
\label{DissipationIneq-iISS-2}
\dot{V}_u(x) \leq -\alpha(\|x\|_X) + \sigma(\|u\|_\Uc),
\end{equation}
then, using the same argument as in the proof of Proposition~\ref{PropSufiISS}, it is possible to obtain a certain variation of the iISS property, which is however possibly weaker than iISS.
\end{remark}

\begin{definition}
\label{def:ISS-wrt-small-inputs}
System $\Sigma=(X,\Uc,\phi)$ is called \emph{input-to-state stable
(ISS) with respect to small inputs}, if there exist $\beta \in \KL$, $\gamma \in \Kinf$ and $R>0$
such that for all $ x \in X$, $ u\in \Uc$: $\|u\|_\Uc\leq R$ and $ t\geq 0$ it holds that
\begin {equation}
\label{iss-small-inputs-sum}
\| \phi(t,x,u) \|_{X} \leq \beta(\| x \|_{X},t) + \gamma( \|u\|_{\Uc}).
\end{equation}
\end{definition}

Although integral ISS shares several nice properties with ISS systems, there are also important distinctions.
\begin{example}
\label{examp:iISS-not-strong-iISS} 
Let $X=\R$ and $\Uc= PC_b(\R_+,\R)$. Consider the system
\begin{eqnarray}
\dot{x} = -\frac{x}{1+x^2} + u.
\label{eq:iISS-not-strong-iISS}
\end{eqnarray}
This system is iISS, but it does not have convergent input-convergent state property (compare to Proposition~\ref{prop:Converging_input_uniformly_converging_state}) and constant inputs of arbitrarily small magnitude induce unbounded trajectories, provided that an initial state is chosen sufficiently large.
\end{example}

Obstructions demonstrated in Example~\ref{examp:iISS-not-strong-iISS} motivate the following strengthening of the iISS property, which is well-studied for ODE systems \cite{CAI14, CAI14b}:
\begin{definition}
\label{def:strong-iISS} 
System $\Sigma=(X,\Uc,\phi)$ is called \emph{strongly integral input-to-state stable (strongly iISS)}, if $\Sigma$ is iISS and ISS with respect to small inputs.
\end{definition}

\begin{example}
\label{examp:1dim_bilinear_system} 
A simple example of a strongly integral ISS system which is not ISS is given by a one-dimensional bilinear system
\begin{eqnarray}
\dot{x} = -x + xu,
\label{eq:SimpleBilinSys}
\end{eqnarray}
with $x(t) \in\R$ and $u\in PC_b(\R_+,\R)$.
Clearly, for $u\equiv 2$ the solution is unbounded, which is not possible for an ISS system. At the same time, \eqref{eq:SimpleBilinSys}
is iISS, which can be shown by using an iISS Lyapunov function $V(x) = \ln(1+x^2)$, $x\in\R$.
Trivially, \eqref{eq:SimpleBilinSys} is also ISS with respect to small inputs, and thus strongly iISS.
\end{example}

Further results on integral input-to-state stability of infinite-dimensional systems are outlined in Sections~\ref{sec:iISS of linear systems}, \ref{sec:Bilinear_systems}, \ref{sec:ISS-SGT-iISS-couplings}.

\section{ISS of linear systems}
\label{sec:Linear_systems}

In this section we study ISS of linear systems with unbounded input operators, which contain linear boundary control systems as a special case, see Section~\ref{sec:BCS-as-linear-systems}.
The ISS analysis of such systems is important in particular for analyzing robustness of PDE systems w.r.t.\ boundary disturbances 
as well as for robust stabilization by means of boundary controllers.
The linearity of control systems allows for efficient functional-analytic criteria for ISS, which are not available for general nonlinear systems 
and makes it possible to derive constructive converse Lyapunov theorems for important subclasses of linear systems.
In this section, we explain these methods in detail.

\subsection{ISS of linear systems with bounded input operators}
\label{sec:Linear_systems_bounded_ops}

Consider first linear control systems with bounded input operators of the form
\begin{equation}
\label{eq:Linear_System}
\dot{x}(t)=Ax(t)+ Bu(t), \quad t > 0,
\end{equation}
where $A$ is the generator of a strongly continuous semigroup $T:=(T(t))_{t\ge 0}$ on a Banach space $X$ and $B\in L(U,X)$ for some Banach space $U$.

We would like to introduce several spaces of Bochner integrable functions, which we will use as the input spaces in this section. 

\begin{definition}
\label{def:Spaces-of-Bochner-integrable functions} 
Let $W$ be a Banach space, $I$ be a closed subset of $\R$ and $p\in[1,+\infty)$. We define the following spaces (see \cite[Definition A.1.14]{JaZ12} for details)
\begin{subequations}
\label{eq:Bochner-L-spaces}
\begin{eqnarray}
M(I,W) &:=& \{f: I \to W: f \text{ is strongly measurable} \},\\
L_p(I,W) &:=& \{f \in M(I,W): \|f\|_{L_p(I,W)}:=\Big(\int_I\|f(s)\|^p_Wds\Big)^{1/p} < \infty \},\\
L_\infty(I,W) &:=& \{f \in M(I,W): \|f\|_{L_\infty(I,W)}:=\esssup_{s\in I}\|f(s)\|_W < \infty \}.
\end{eqnarray}
\end{subequations}
Identifying the functions, which differ on a set with a Lebesgue measure zero, the spaces $L_p(I,W)$, $p\in[1,+\infty]$ are Banach spaces.

We define also for $p\in[1,+\infty]$ the spaces $L_{p,loc}(\R_+,W)$ of \emph{locally Bochner-integrable functions}, consisting of functions $f:\R_+\to W$ so that the restriction $f_I$ of a map $f$ to any compact interval $I \subset \R_+$ belongs to $L_p(I,W)$.
\end{definition}

\begin{definition}
\label{def:Mild-solution-linear-system} 
For every $x \in X$ and every $u\in \Uc:=L_{1, loc}(\R_+,U)$, the function $\phi(\cdot,x,u):\R_+\rightarrow X$, defined by
\begin{eqnarray}
\phi(t,x,u):=T(t)x+\int_0^t T(t-s)B u(s)ds,\quad t\ge 0,
\label{eq:Lin_Sys_mild_Solution}
\end{eqnarray}
is called a {\em mild solution} of \eqref{eq:Linear_System}. 
\end{definition}

The integral in \eqref{eq:Lin_Sys_mild_Solution} and on the following pages is understood in the sense of Bochner.
For a general introduction to the Bochner integral, we refer to \cite{ABH11}.

\begin{remark}
\label{rem:L_p-loc-spaces} 
Note that the space $L_{1, loc}(\R_+,U)$ is not a normed linear space, as some of its elements have \q{infinite} $L_1$-norm (that is, the $L_1$-norm cannot be introduced for these spaces), and hence according to Definition~\ref{Steurungssystem}, $L_{1, loc}(\R_+,U)$ is not allowed as an input space to a control system. However, ISS of a control system w.r.t.\ inputs in $L_1(\R_+,U)$ implies ISS of the system with $L_{1, loc}(\R_+,U)$, due to causality of the control systems, similarly to Remark~\ref{prop:Causality_Consequence}.
Therefore we slightly abuse the use of Definition~\ref{Steurungssystem}, and allow for the input spaces $L_{p,loc}(\R_+,U)$.
\end{remark}

\begin{lemma}
\label{lem:Systems-with-bounded-input-operators-are-well-posed} 
Let $\Uc:=L_{1, loc}(\R_+,U)$, and let $\phi$ be defined by \eqref{eq:Lin_Sys_mild_Solution}. 
Then $\Sigma:=(X,\Uc,\phi)$  is a control system in the sense of Definition~\ref{Steurungssystem}.
\end{lemma}

\begin{proof}
For a fixed $u\in\Uc$ the map $s \mapsto T(t-s)B u(s)$ is Bochner integrable and  
$t\mapsto \int_0^t T(t-s)B u(s)ds$ is continuous, see \cite[Proposition 1.3.4]{ABH11}.\footnote{If $X$ is Hilbert space, a simpler argument can be found in \cite[Example A.1.13, Lemma 10.1.6]{JaZ12}.}
As for a fixed $x\in X$ the map $t \mapsto T(t)x$ is continuous, $\phi$ is continuous w.r.t.\ $t$ as well.
Verification of the other axioms of a control system is straightforward.
\end{proof}

For ease of terminology we identify $\Sigma=(X,\Uc,\phi)$ with \eqref{eq:Linear_System}.
Next we give exhaustive criteria for ISS of linear systems \eqref{eq:Linear_System} with bounded input operators.
\begin{theorem}{(\cite[Proposition 3, Lemma 1]{DaM13}, \cite[Proposition 4]{MiI16}, \cite[Propositions 6,7]{MiW17c})}
\label{thm:ISS-criterion-linear-systems-bounded-operators} 
The following statements are equivalent for a system \eqref{eq:Linear_System} with $B\in L(U,X)$:
\begin{itemize}
    \item[(i)] \eqref{eq:Linear_System} is ISS,
    \item[(ii)] \eqref{eq:Linear_System} is iISS, 
    \item[(iii)] \eqref{eq:Linear_System} is 0-UGAS,
    \item[(iv)] The semigroup $T$ is exponentially stable,
    \item[(v)] The function $V:X\to \R_+$ defined by 
    \begin{eqnarray}
V(x)=\int_0^{\infty} \|T(t) x\|_X^2 dt
\label{eq:LF_LinSys_Banach}
\end{eqnarray}
is a non-coercive ISS Lyapunov function for \eqref{eq:Linear_System},
    \item[(vi)] The function $V^\gamma:X\to \R_+$ defined by 
\begin{equation}
    \label{eq:eqnorm}
    V^\gamma (x) := \max_{s\geq 0} \| e^{\gamma s} T(s) x\|_X,
\end{equation}
is a globally Lipschitz coercive ISS Lyapunov function for \eqref{eq:Linear_System}.
Here $\gamma\in (0,\lambda)$, and $\lambda\in\R$ is so that $\|T(t)\|_X\leq Me^{-\lambda t}$ for some $M>0$ and all $t\geq 0$,
    \item[(vii)] There is a non-coercive ISS Lyapunov function for \eqref{eq:Linear_System}.
\end{itemize}
\end{theorem}

\begin{proof}
(i) $\Iff$ (iii) $\Iff$ (iv) is an easy computation similar to that in Example~\ref{sec:ISS-finite-dim-systems},
see \cite[Proposition 3, Lemma 1]{DaM13}.
Equivalence (i) $\Iff$ (ii) follows from  \cite[Proposition 4]{MiI16}.
Equivalences (i) $\Iff$ (v) $\Iff$ (vi) are shown in \cite[Propositions 6, 7]{MiW17c}.
Clearly, (v) implies (vii), and (vii) implies (i) by Theorem~\ref{t:ISSLyapunovtheorem} (CEP an BRS properties for a linear system \eqref{eq:Linear_System} are always fulfilled).
\end{proof}

\begin{remark}
\label{rem:Coercivity-of-quadratic-LFs-in-special-cases} 
The Lyapunov function defined by \eqref{eq:LF_LinSys_Banach} is in general non-coercive, but it is coercive for some special classes of systems.
\end{remark}

Theorem~\ref{thm:ISS-criterion-linear-systems-bounded-operators} shows that
the question about ISS of a control system \eqref{eq:Linear_System} with $B\in L(U, X)$ can be reduced to the analysis of the exponential stability of a semigroup $T$, generated by $A$, which is a classical problem in semigroup theory.
On the other hand, for each input-to-state stable linear system with a bounded input operator, there is a non-coercive as well as a coercive ISS Lyapunov function.
The Lyapunov function \eqref{eq:LF_LinSys_Banach} is rather standard, see, e.g., \cite[Theorem 5.1.3]{CuZ95}, and in a Hilbert space setting it is obtained via the solution of an operator Lyapunov equation. The fact that such natural Lyapunov functions, as \eqref{eq:LF_LinSys_Banach} are non-coercive, is one of the reasons motivating the study of non-coercive Lyapunov functions.

\begin{remark}\label{REM:1}
For finite-dimensional linear systems \eqref{eq:LinSys}  0-UGAS (and thus ISS) is equivalent to the \emph{strong stability of a semigroup $T$}, i.e. to the property that for all $x\in X$ it holds that $\phi(t,x,0)=T(t)x \to 0$ as $t\to\infty$.
For infinite-dimensional linear systems \eqref{eq:Linear_System} strong stability of a semigroup $T$ is much weaker than exponential stability, even for bounded $A$. Moreover, in this case, \eqref{eq:Linear_System} may have unbounded solutions in the presence of a bounded input, see, e.g., \cite[p. 247]{MaP11}.
\end{remark}

\subsection{ISS of linear systems with unbounded input operators}
\label{sec:ISS of linear systems}

Assume that $X$ and $U$ are Banach spaces, $p\in [1,\infty]$ and $\Uc:=L_p(\R_+,U)$.
Let also $A:D(A) \subset X\to X$ be an infinitesimal generator of a strongly continuous semigroup $T:=(T(t))_{t\geq 0}$ on $X$ with a nonempty resolvent set $\rho(A)$.
Consider again the equation \eqref{eq:Linear_System} but now assume that $B$ is an unbounded linear operator from $U$ to $X$.
Unbounded input operators naturally appear in boundary or point control of linear systems, see, e.g., \cite{JaZ12}. 

Define the extrapolation space $X_{-1}$ as the completion of $X$ with respect to the norm 
$ \|x\|_{X_{-1}}:= \|(aI -A)^{-1}x\|_X$ for some $a \in \rho(A)$.
$X_{-1}$ is a Banach space (see \cite[Theorem 5.5, p. 126]{EnN00}) and different choices of $a \in \rho(A)$ generate equivalent norms on $X_{-1}$, see \cite[p. 129]{EnN00}.
As we know from the representation theorem \cite[Theorem 3.9]{Wei89b}, the input operator $B$ must satisfy the condition $B\in L(U,X_{-1})$ in order to give rise to a well-defined control system (at least for $p<\infty$). We assume this for all $p\in[1,+\infty]$.

Lifting of the state space $X$ to a larger space $X_{-1}$ brings further good news: the semigroup  $(T(t))_{t\ge 0}$ extends uniquely to a strongly continuous semigroup  $(T_{-1}(t))_{t\ge 0}$ on $X_{-1}$ whose generator $A_{-1}:X_{-1}\to X_{-1}$ is an extension of $A$ with $D(A_{-1}) = X$, see, e.g.,\ \cite[Section II.5]{EnN00}.
Thus we may consider equation \eqref{eq:Linear_System} on the Banach space $X_{-1}$:
\begin{equation}
\label{eq:Linear_System_extrapolated}
\dot{x}(t)=A_{-1}x(t)+ Bu(t), \quad t > 0,
\end{equation}
and mild solutions of this extrapolated system are given by the variation of constants formula \eqref{eq:Lin_Sys_mild_Solution}, which takes 
for every $x \in X$, $u\in L_{1, loc}(\R_+,U)$  and every $t\geq 0$ the form 
\begin{eqnarray}
\phi(t,x,u)&:=&T_{-1}(t)x+\int_0^t T_{-1}(t-s)B u(s)ds\\
        &=&T(t)x+\int_0^t T_{-1}(t-s)B u(s)ds.
\label{eq:Lifted_Lin_Sys_mild_Solution}
\end{eqnarray}
The last transition is due to the fact that $T_{-1}(t)x = T(t)x$ for all $x\in X$ and $t\geq 0$.

The lifting comes however at a price that now the solution $\phi(t,x,u)$ has values in $X_{-1}$.
The formula \eqref{eq:Lifted_Lin_Sys_mild_Solution} defines an $X$-valued function only in case if the value of the integral in \eqref{eq:Lifted_Lin_Sys_mild_Solution} belongs to the state space $X$, despite the fact that what we integrate is in $X_{-1}$.
This motivates the following definition: 
\begin{definition}
\label{def:q-admissibility} 
The operator $B\in L(U,X_{-1})$ is called a {\em $q$-admissible control operator} for $(T(t))_{t\ge 0}$, where $1\le q\le \infty$, if
there is a $t>0$ so that
\begin{eqnarray}
u\in L_q(\R_+,U) \qrq \int_0^t T_{-1}(t-s)Bu(s)ds\in X.
\label{eq:q-admissibility}
\end{eqnarray}
\end{definition}
Define the operators $\Phi(t): \Uc \to X$, $\Phi(t)u := \int_0^t T_{-1}(t-s)Bu(s)ds$.
Note that for any $B \in L(U,X_{-1})$ operators $\Phi(t)$ are well-defined as maps from $U$ to $X_{-1}$ for all $t$. The next result shows that from admissibility of $B$ it follows that the image of $\Phi(t)$ is in $X$ for all $t \geq 0$ and $\Phi(t)\in L(U,X)$ for all $t$.
\begin{proposition}{(\cite[Proposition 4.2]{Wei89b}, \cite[Proposition 4.2.2]{TuW09})}
\label{prop:Restatement-admissibility} 
Let $X, U$ be Banach spaces and let $p \in [1,\infty]$ be given. Then $B \in L(U,X_{-1})$ is $q$-admissible if and only if
for all $t>0$ there is $h_t>0$ so that for all $u \in L_{q,loc}(\R_+,U)$ it holds that $\Phi(t)u \in X$ and 
\begin{equation}
\label{eq:admissible-operator-norm-estimate}
\left\| \int_0^t T_{-1}(t-s)Bu(s)\,ds\right\|_X \le h_t \|u\|_{L_q([0,t],U)}.
\end{equation}
\end{proposition}

%

If $B\in L(U,X_{-1})$ is a $q$-admissible operator, then $\phi$ given by \eqref{eq:Lifted_Lin_Sys_mild_Solution} is well-defined on $\R_+\times X\times \Uc$, with $\Uc:=L_q(\R_+,U)$.
The next natural question is whether $\Sigma:=(X,\Uc,\phi)$ is a control system in the sense of Definition~\ref{Steurungssystem}.
For $q<\infty$ it has an affirmative answer:
\begin{proposition}
\label{prop:q-admissibility-implies-continuity} 
Assume that $B\in L(U,X_{-1})$ is a $q$-admissible input operator with $q<\infty$.
Then $\Sigma:=(X,\Uc,\phi)$ with $\Uc:=L_q(\R_+,U)$ is a forward-complete control system in the sense of Definition~\ref{Steurungssystem}.
\end{proposition}

\begin{proof}
Using Proposition~\ref{prop:Restatement-admissibility} one can show that $(T,\Phi)$, where $\Phi:=\{\Phi(t):t\in\R_+\}$ is an abstract linear control system in the sense of \cite[Definition 2.1]{Wei89b}. Continuity of $\phi$ in the $X$-norm can be inferred from 
\cite[Proposition 2.3]{Wei89b} (here an assumption that $q\neq \infty$ is essential).
After some additional effort, the claim follows.
\end{proof}

For $\infty$-admissible control operators we have the following result:
\begin{proposition}
\label{prop:infty-admissibility-implies-continuity} 
Assume that $B\in L(U,X_{-1})$ is an $\infty$-admissible input operator and that $\phi$ is continuous w.r.t.\ time in the norm of $X$. 
Then $\Sigma:=(X,\Uc,\phi)$ with $\Uc:=L_\infty(\R_+,U)$ is a forward-complete control system in the sense of Definition~\ref{Steurungssystem}.
\end{proposition}

\begin{proof}
The proof goes along the lines of the proof of Proposition~\ref{prop:q-admissibility-implies-continuity}.
\end{proof}

\begin{remark}
\label{rem:Admissibility-and-forward-completeness} 
In the language of general control systems, $q$-admissibility means precisely the forward-completeness of control systems for all inputs from $L_q$-space, which is in view of Proposition~\ref{prop:Restatement-admissibility} equivalent to BRS property.
\end{remark}

Although $\phi$ given by \eqref{eq:Lifted_Lin_Sys_mild_Solution} is continuous w.r.t.\ time in the norm of $X_{-1}$ (as a mild solution), but it does not imply that $\phi$ is continuous in the $X$-norm.
Question whether $\infty$-admissibility of $B\in L(U,X_{-1})$ implies the continuity of the solution map $\phi$ w.r.t.\ time in the $X$-norm is open for 30 years (\cite[Problem 2.4]{Wei89b}). Recently, this was positively answered for an important subclass of analytic systems on Hilbert spaces:
\begin{theorem}{(Follows from \cite[Theorem 1]{JSZ17})}
\label{thm:Contiuity-of-a-map} 
Let $A$ generate an exponentially stable analytic semigroup on a Hilbert space $X$, which is similar to a contraction semigroup. 
Assume further that $\dim (U) <\infty$. 

Then any $B\in L(U,X_{-1})$ is an $\infty$-admissible operator, and the corresponding map $\phi$ is continuous in the norm of $X$ for any $u \in \Uc:=L_\infty(\R_+,U)$. In particular, $\Sigma:=(X,\Uc,\phi)$ is a control system in the sense of Definition~\ref{Steurungssystem}.
\end{theorem}

\begin{remark}
\label{rem:Relations-between-admissibility-classes} 
Lemma~\ref{prop:Restatement-admissibility} and H\"older's inequality show that $q$-admissibility of $B\in L(U,X_{-1})$ implies $p$-admissibility of $B$ for any $p>q$. 
At the same time, in general, $q$-admissibility of $B$ does not imply $p$-admissibility for any $p<q$.
In particular, in \cite[Example 5.2]{JNP18} an example of an $\infty$-admissible operator is given, 
which is not $p$-admissible for any $p<\infty$.
Thus, admissibility gives us a measure for how bad-behaved an unbounded input operator may be, and $\infty$-admissible operators are the \q{worst} type of operators, which still allows for well-posedness of a linear control system \eqref{eq:Linear_System} and at the same time give
the possibility to obtain ISS with respect to supremum norm of the inputs, which is the most classical type of ISS estimates.
Bounded operators are always 1-admissible, and if $X$ is reflexive, then $1$-admissibility of $B$ is equivalent to the fact that $B\in L(U, X)$, see 
\cite[Theorem 4.8]{Wei89b}.
\end{remark}

\begin{definition}
\label{def:infinite-time-admissibility} 
A $q$-admissible operator $B$ is called \emph{infinite-time $q$-admissible operator}, if the constant $h:=h_t$ in Lemma~\ref{prop:Restatement-admissibility} does not depend on $t$.
\end{definition}

\begin{remark}
\label{rem:Admissibility_and_infinite-time_admissibility} 
If the semigroup $T$ is exponentially stable, i.e. if \eqref{eq:Linear_System} is 0-UGAS, then 
$B$ is infinite-time $q$-admissible if and only if $B$ is $q$-admissible, see \cite[Lemma 2.9]{JNP18} 
or \cite[Lemma 1.1]{Gra95} (in a dual form).
\end{remark}

Let us proceed to the ISS analysis. For linear systems it is natural to consider the following strengthening of the input-to-state stability:
\begin{definition}
\label{def:eISS}
System $\Sigma=(X,\Uc,\phi)$ is called \emph{exponentially  input-to-state stable
(eISS) with a linear gain}, if there exist $M,\lambda, G>0$
such that 
\begin {equation}
\label{eq:eiss_sum}
 x \in X \ \wedge \ u\in\Uc \ \wedge \ t\geq 0 \qrq \| \phi(t,x,u) \|_{X} \leq M e^{-\lambda t}\| x \|_{X} + G \|u\|_{\Uc}.
\end{equation}
\end{definition}

Now we can characterize ISS of \eqref{eq:Linear_System}, see, e.g., \cite[Proposition 2.10]{JNP18}:
\begin{theorem}
\label{thm:ISS-Criterion-lin-sys-with-unbounded-operators}
Let $X$ and $U$ be Banach spaces and let $\Uc:=L_p(\R_+,U)$ for some $p\in[1,+\infty]$. 
If $p=+\infty$, assume further that $\phi$ is continuous w.r.t. time in the norm of $X$.
The following properties are equivalent for a system \eqref{eq:Linear_System} with above $X$ and $\Uc$:
\begin{enumerate}
    \item[(i)] \eqref{eq:Linear_System} is ISS
    \item[(ii)] \eqref{eq:Linear_System} is eISS with a linear gain
    \item[(iii)] \eqref{eq:Linear_System} is 0-UGAS $\ \wedge\ $ $B$ is infinite-time $p$-admissible
    \item[(iv)] \eqref{eq:Linear_System} is 0-UGAS $\ \wedge\ $ $B$ is $p$-admissible
\end{enumerate}
\end{theorem}

\begin{proof}
(ii) $\Rightarrow$ (i) $\Rightarrow$ (iv). Evident.

(iv) $\Rightarrow$ (iii). Follows from Remark~\ref{rem:Admissibility_and_infinite-time_admissibility}.

(iii) $\Rightarrow$ (ii). Follows from Proposition~\ref{prop:Restatement-admissibility} with $h\equiv const$ and exponential stability of a semigroup $T$.
\end{proof}

Similarly in spirit to Theorem~\ref{thm:ISS-criterion-linear-systems-bounded-operators}, 
Theorem~\ref{thm:ISS-Criterion-lin-sys-with-unbounded-operators} reduces the ISS analysis of linear systems to the     stability analysis of the semigroup and to admissibility analysis of the input operator, which are classical functional-analytic problems with 
many powerful tools for its solution, see \cite{JaP04, TuW09}.

\subsection{Integral ISS of linear systems with unbounded input operators}
\label{sec:iISS of linear systems}

The ISS estimate is defined in terms of the norms in $X$ and $\Uc$ which is one of the reasons for the rather elegant characterization of 
ISS in terms of the exponential stability of a semigroup and admissibility of the input operator.
In contrast to this, integral ISS is defined in terms of the integration of an input function with a nonlinear scaling, see \eqref{iISS_Estimate}, which does not necessarily produce a norm. This makes the characterization of iISS more involved.

Recall that for linear systems with bounded input operators ISS and iISS are equivalent notions in view of Theorem~\ref{thm:ISS-criterion-linear-systems-bounded-operators}. This remains true for linear systems with unbounded operators, if $\Uc$ is an $L_p$ space with $p<\infty$.
\begin{proposition}
\label{prop:iISS_equals_ISS_p_less_infty} 
Let $X$ and $U$ be Banach spaces and let $\Uc:=L_p(\R_+,U)$ for some $p\in[1,+\infty)$. 
The system \eqref{eq:Linear_System} is ISS $\Iff$ \eqref{eq:Linear_System} is iISS.
\end{proposition}

\begin{proof}
For $\Uc:=L_p(\R_+,U)$ with $p\in[1,+\infty)$ the ISS estimate \eqref{iss_sum} is automatically an iISS estimate with 
$\theta(r) = \gamma(r^{1/p})$ and $\mu(r) = r^p$, $r\geq 0$.
The converse implication can be shown as in Proposition~\ref{prop:iISS_implies_ISS}.
\end{proof}

The case $p=\infty$ needs special care. First of all, integral ISS implies ISS:
\begin{proposition}{\cite[Proposition 2.10]{JNP18}}
\label{prop:iISS_implies_ISS} 
Let $X$ and $U$ be Banach spaces and let $\Uc:=L_\infty(\R_+,U)$.
If \eqref{eq:Linear_System} is iISS, then $T$ is an exponentially stable semigroup and $B$ is an infinite-time $\infty$-admissible operator. 
If additionally $\phi$ is continuous w.r.t.\ time in the norm of $X$, then \eqref{eq:Linear_System} is ISS.
\end{proposition}

\begin{proof}
Integral ISS estimate for $u:=0$ implies exponential stability of the semigroup $T$.
Furthermore, integral ISS implies well-posedness of the system \eqref{eq:Linear_System} for $\Uc:=L_\infty(\R_+,U)$, which is equivalent to $\infty$-admissibility of $B$. 
According to Remark~\ref{rem:Admissibility_and_infinite-time_admissibility}, $B$ is infinite-time $\infty$-admissible. 
As $\phi$ is continuous w.r.t.\ time in the norm of $X$, Theorem~\ref{thm:ISS-Criterion-lin-sys-with-unbounded-operators}
shows ISS of \eqref{eq:Linear_System}.
\end{proof}

\begin{remark}
\label{rem:Terminology_of_JNP18_JSZ19} 
In this survey, continuity of (mild) solutions is a part of the definition of the system, which is in accordance with the well-posedness concepts for linear infinite-dimensional control systems in \cite[Definition 2.1]{Wei89b}, \cite[p. 171]{JaZ12}.
However, in the terminology of \cite{JNP18, JSZ17} the continuity of solutions is not a part of the definition of ISS. Hence, in the terminology of \cite{JNP18, JSZ17}, 
ISS of \eqref{eq:Linear_System} w.r.t.\ $\Uc:=L_\infty(\R_+,U)$ is equivalent to the properties that $T$ is an exponentially stable semigroup and $B$ is infinite-time $\infty$-admissible operator, and Proposition~\ref{prop:iISS_implies_ISS} (now without an assumption of the continuity of $\phi$) reads as follows:
\begin{center}
iISS of \eqref{eq:Linear_System} w.r.t.\ $L_\infty(\R_+,U)$ $\qrq$ ISS of \eqref{eq:Linear_System} w.r.t. $L_\infty(\R_+,U)$.
\end{center}
\end{remark}

\begin{openprob}
\label{ob:ISS-and-iISS-for-linear-systems} 
In view of Proposition~\ref{prop:iISS_implies_ISS} we know that for \linebreak
$\Uc:=L_\infty(\R_+,U)$, integral ISS implies ISS. 
At the same time, the question whether ISS w.r.t.\ $L_\infty(\R_+,U)$ implies iISS w.r.t.\ $L_\infty(\R_+,U)$ is a challenging open problem.
It is interesting, that this problem comes against the \q{finite-dimensional} intuition, as for \emph{nonlinear} ODE systems with $\Uc:=L_\infty(\R_+,\R^m)$ it is known that under natural regularity assumptions on the right-hand side $f(\cdot)$ ISS implies integral ISS, although the converse is not true for nonlinear systems, by Example~\ref{examp:1dim_bilinear_system}.

By definition of a coercive ISS Lyapunov function for general infinite-dimensional systems, it follows that any such function is automatically a 
coercive iISS Lyapunov function, and thus a system possessing an ISS Lyapunov function is both ISS and iISS.
Hence if there are ISS control systems which are not iISS (i.e. if the answer on the above problem is negative), then these systems do not possess a coercive ISS Lyapunov function.
\end{openprob}

A partial positive answer for Open Problem~\ref{ob:ISS-and-iISS-for-linear-systems} was obtained in \cite{JSZ17} for parabolic systems on Hilbert spaces employing holomorphic functional calculus.

\begin{proposition}{\cite[Theorem 2]{JSZ17}}
\label{prop:iISS-equals-ISS-analytic-case} 
Assume that $A$ generates an exponentially stable, analytic semigroup on a Hilbert 
space $X$ which is similar to a contraction semigroup.
Let $U$ be so that $\dim (U) <\infty$, $\Uc:=L_\infty(\R_+,U)$.

The following statements are equivalent:
\begin{itemize}
    \item[(i)]   $B\in L(U,X_{-1})$,
    \item[(ii)]  \eqref{eq:Linear_System} is ISS w.r.t.\ $\Uc$,
    \item[(iii)] \eqref{eq:Linear_System} is iISS w.r.t.\ $\Uc$.
\end{itemize}
\end{proposition}

\begin{proof}
(iii) $\Rightarrow$ (ii). Integral ISS of \eqref{eq:Linear_System} w.r.t.\ $\Uc$ implies that $B$ is an $\infty$-admissible operator, and in particular, $B\in L(U,X_{-1})$.
By Theorem~\ref{thm:Contiuity-of-a-map} the triple $\Sigma:=(X,\Uc,\phi)$ is a control system, in particular, $\phi$ is continuous w.r.t.\ time in $X$-norm. Now Proposition~\ref{prop:iISS_implies_ISS} shows ISS of \eqref{eq:Linear_System} w.r.t.\ $\Uc$.

(ii) $\Rightarrow$ (i). Clear.

(i) $\Rightarrow$ (iii). See \cite[Theorem 2]{JSZ17}.
\end{proof}

\begin{remark}
\label{rem:iISS-and-Orlicz} 
A notable result \cite[Theorem 3.1]{JNP18} shows that $L_\infty$-iISS of \eqref{eq:Linear_System} is equivalent to existence of an Orlicz space $W$ so that \eqref{eq:Linear_System} is ISS with $\Uc:=W$. Due to the space limitations we omit the precise formulation of this result.
\end{remark}

\begin{table}
\center
\def\arraystretch{1.5}%
\begin{tabular}{c|c|c|c}
& \parbox[t]{1.9cm}{\centering Eq.~\eqref{eq:Linear_System},\\ 
$B$ bounded}&\parbox[t]{2.3cm}{\centering Eq.~\eqref{eq:Linear_System}, \\
$B$ unbounded\vspace{0.2cm} } &\parbox[t]{2cm}{\centering Eq.~\eqref{InfiniteDim},\\ $f$ nonlinear}\\
\hline
$\mathrm{dim}\, X<\infty$ &\multicolumn{2}{|c|}{ISS $\iff$ iISS}&ISS $\Longrightarrow \atop{\centernot\Longleftarrow}$ iISS\\
\hline
$\mathrm{dim}\, X=\infty$&ISS $\iff$ iISS& ISS $\impliedby\atop \left(\stackrel{?}{\Longrightarrow}\right)$ iISS& not known \\
\hline
\end{tabular}
\medskip

\label{table:overview}
\caption{(Taken from \cite{JNP18}). Relations between ISS and iISS with respect to $L^{\infty}$ and under assumption that $\phi$ is continuous in $X$-norm, in various settings.}
\end{table}

%
%
%

\subsection{Lyapunov functions for linear systems}
\label{sec:LFs_linear_systems}

In this section, we investigate the applicability of Lyapunov methods to the analysis of linear systems with unbounded input operators.
We start with good news:
\begin{proposition}
\label{prop:for-linear-systems-ncISSLF-implies-ISS} 
Assume that $X,U$ are Banach spaces, $q\in [1,+\infty]$ and $B$ is a $q$-admissible operator.
If $q=+\infty$, assume additionally that the mild solution $\phi$ of \eqref{eq:Linear_System} is continuous w.r.t.\ time in the norm of $X$.
Then existence of a non-coercive ISS Lyapunov function for \eqref{eq:Linear_System} implies ISS of \eqref{eq:Linear_System}.
\end{proposition}

\begin{proof}
Under assumptions of this proposition, \eqref{eq:Linear_System} is a well-defined control system, see Propositions~\ref{prop:q-admissibility-implies-continuity}, \ref{prop:infty-admissibility-implies-continuity}. Furthermore, it trivially satisfies CEP and BRS properties (see Definitions~\ref{def:RobustEquilibrium_Undisturbed} and \ref{def:BRS} and Remark~\ref{rem:Admissibility-and-forward-completeness}).
The claim follows by Theorem~\ref{t:ISSLyapunovtheorem}.
\end{proof}

Proposition~\ref{prop:for-linear-systems-ncISSLF-implies-ISS} enables us to use non-coercive ISS Lyapunov functions as in item (v) of Theorem~\ref{thm:ISS-criterion-linear-systems-bounded-operators}.
However, we have \q{to compensate} the influence of unbounded input operators, which poses further requirements on Lyapunov functions.
Next, we state a recent result of this type:

In what follows we denote the Hilbert space adjoint of an operator $A:X\to X$ by $A^*$.
The following result is due to \cite{JMP20}:
\begin{theorem}
\label{thm:Gen_ISS_LF_Construction}
Let $(A,D(A))$ be  the generator of a strongly continuous semigroup $(T(t))_{t\ge 0}$ on a complex Hilbert space $X$.

Assume that there is an operator $P\in L(X)$ satisfying the following conditions:
\begin{itemize}
    \item[(i)] \label{item:Gen_Converse_Lyap_Theorem_1} $P$ satisfies
\begin{eqnarray}
\re\scalp{Px}{x}_X > 0,\quad x\in X\backslash\{0\}.
\label{eq:Positivity}
\end{eqnarray}
\item[(ii)] \label{item:Gen_Converse_Lyap_Theorem_3}  Im$\,(P) \subset D(A^*)$.
\item[(iii)] \label{item:Gen_Converse_Lyap_Theorem_4} $PA$ has an extension to a bounded operator on $X$, that is, $PA\in L(X)$. 
\item[(iv)] \label{item:Gen_Converse_Lyap_Theorem_2} $P$ satisfies the Lyapunov inequality 
\begin{eqnarray}
\re\scalp{(PA+A^*P)x}{x}_X \leq  -\scalp{x}{x}_X,\quad x\in D(A).
\label{eq:LyapIneq}
\end{eqnarray}
\end{itemize}
Then
\begin{equation}
\label{Lyap}
 V(x) := \re\, \langle Px,x\rangle_X
 \end{equation}
is a non-coercive ISS Lyapunov function for \eqref{eq:Linear_System}
with any $\infty$-admissible input operator $B\in L(U,X_{-1})$.
In particular, \eqref{eq:Linear_System} is ISS for such $B$.
\end{theorem}

\begin{remark}
\label{rem:Why-Real-parts} 
We have to take the real parts of the expressions in \eqref{Lyap} and \eqref{eq:LyapIneq}, as we deal with complex Hilbert spaces and we do not assume that $P$ is a positive operator on $X$.
\end{remark}

Note that properties (i) and (ii) in Theorem~\ref{thm:Gen_ISS_LF_Construction} taken together are equivalent to exponential stability of the semigroup $T$.
To see how Theorem~\ref{thm:Gen_ISS_LF_Construction} can be applied, we state
\begin{corollary}
\label{cor:Self-adjoint-A} 
Let $X$ be a complex Hilbert space and let $(A,D(A))$ be a negative definite self-adjoint operator on $X$, that is 
$\scalp{Ax}{x} <0$ for all $x\in D(A)\backslash\{0\}$ and $D(A) = D(A^*)$ with $Ax = A^*x$ for all $x\in D(A)$.
Then $ V(x) := - \langle A^{-1}x,x\rangle_X$ is a non-coercive ISS Lyapunov function for \eqref{eq:Linear_System} with any $\infty$-admissible operator $B$.
\end{corollary}

\begin{proof}
Take $P:=-A^{-1}$. Since $A = A^*$, for all $x \in X$ we have $\scalp{A^{-1}x}{x}\in\R$, as
\[
\scalp{A^{-1}x}{x} = \scalp{A^{-1}x}{A^*A^{-1}x} = \scalp{AA^{-1}x}{A^{-1}x} = \scalp{x}{A^{-1}x}.
\]
For any $x \in X\backslash\{0\}$ there is $y \in D(A) \backslash\{0\}$ so that $Ax = y$ and thus 
\[
\re\scalp{Px}{x} = -\scalp{A^{-1}x}{x} = -\scalp{y}{Ay} = -\scalp{Ay}{y} >0.
\]
This shows condition (i) in Theorem~\ref{thm:Gen_ISS_LF_Construction}. As $\im(P) = D(A) = D(A^*)$, condition (ii) holds as well. Conditions (iii) and (iv) are trivially satisfied. 
Theorem~\ref{thm:Gen_ISS_LF_Construction} shows the claim.
\end{proof}

\begin{remark}
\label{rem:Analyticity-of-a-self-adjoint-generator} 
Assumptions of Corollary~\ref{cor:Self-adjoint-A} imply that $A$ generates an analytic exponentially stable semigroup. 
In particular, Corollary~\ref{cor:Self-adjoint-A} can be applied to construct a non-coercive ISS Lyapunov function for a heat equation with Dirichlet boundary input, see \cite[Example 5.1]{JMP18}.
\end{remark}

Corollary~\ref{cor:Self-adjoint-A} motivates that the choice $P:=-A^{-1}$ can be useful in more general situations.  The next result shows that this is indeed the case.
\begin{proposition}{\cite[Proposition 4.1]{JMP18}}
\label{prop:Conv_ISS_LF_Theorem_LinOp}
Let $(A,D(A))$ be  the generator of an exponentially stable strongly continuous semigroup $(T(t))_{t\ge 0}$ on a (complex) Hilbert space $X$.
Further, assume that $D(A)\subseteq D(A^\ast)$ and the inequality 
\begin{equation}\label{eqn:a2}
 \re\, \langle A^*A^{-1}x, x\rangle_X + \delta \|x\|_X^2\ge 0
\end{equation}
holds for some $\delta<1$ and every $x\in X$, and $  \re\,\langle Ax,x\rangle_X <0$ for  every $x\in D(A)\backslash\{0\}$.

 Then
\begin{equation}\label{Lyap-special-case}
 V(x) := - \re\, \langle A^{-1}x,x\rangle_X
 \end{equation}
is a non-coercive ISS Lyapunov function for any $\infty$-admissible operator $B\in L(U,X_{-1})$.
\end{proposition}

\begin{openprob}
\label{op:non-coercive-ISS-LFs} 
In spite of several positive results described in this section, our present knowledge of Lyapunov methods for linear systems with unbounded operators is insufficient, and many questions are open. 
The very first question is whether ISS of \eqref{eq:Linear_System} implies the existence of a (coercive or non-coercive) ISS Lyapunov function.
It is known that for linear hyperbolic systems of conservation laws it is possible to construct a coercive ISS Lyapunov function \cite{TPT18}, and at the time it is unknown, whether coercive ISS Lyapunov functions exist for a heat equation with Dirichlet boundary conditions, or, more generally, for systems governed by analytic semigroups generated by operators with an unbounded spectrum.

Also, it is not known whether there are linear ISS control systems for which there is no coercive ISS Lyapunov function, but non-coercive ISS Lyapunov functions do exist.
As ISS can be characterized as the exponential stability of a semigroup combined with the admissibility of the input operator, a Lyapunov characterization of the admissibility could probably be helpful. Such a characterization exists for the case of 2-admissibility, see 
\cite[Theorem 3.1]{HaW97}.
\end{openprob}


\subsection{Diagonal systems}
\label{sec:diagonal_systems}

We have already characterized ISS for infinite-dimen-sional linear systems in terms of exponential stability of a semigroup and admissibility of the corresponding input operator. Now we turn our attention to parabolic diagonal systems, which is a well-studied (see, e.g., \cite[Sections 2.6, 5.3]{TuW09}) class of linear systems for which efficient criteria of admissibility are available.

For $q<\infty$ denote by $\ell_q$ the Banach space of sequences $\{a_k\}_{k\in\N}$, $a_k\in\C$ so that $\sum_{k=1}^\infty |a_k|^q<\infty$.
\begin{definition} 
\label{def:q-Riesz-basis}
Let $X$ be a separable Hilbert space. 
We say that a sequence of vectors $\{\phi_k\}_{k\in\N}$ is a \emph{$q$-Riesz basis} of $X$
if $\{\phi_k\}_{k\in\N}$ is a Schauder basis of $X$ and 
for certain constants $c_1,c_2>0$ and all sequences $\{a_k\}_{k\in\N} \in \ell_{q}$ we have
\[ 
c_1\sum_{k=1}^\infty |a_k|^q \le \left\| \sum_{k=1}^\infty a_k \phi_k \right\|^q_X\le  c_2 \sum_{k=1}^\infty |a_k|^q.
\]
\end{definition}

Assume that $U=\mathbb C$ and $1\leq q < \infty$. Further, let $X$ be a separable Hilbert space and let $A$ possess a $q$-Riesz basis of eigenvectors $\{\phi_k\}_{k\in \N}$ with eigenvalues $\{\lambda_k\}_{k\in\N}$.

We further assume that the sequence $\{\lambda_k\}_{k\in\N}$ lies in $\mathbb C$ with $\sup_{k\in\N}\re( \lambda_{k})<0$ and that there exists a constant $R>0$ such that $|\im\, \lambda_{k}|\le R|\re\, \lambda_{k}|$, $k\in \N$.

Hence the linear operator $A \colon D(A)\subset X \rightarrow X$, defined by 
\[ A\phi_k = \lambda_k \phi_k, \quad k\in\N, \qquad D(A)= \Big\{\sum_{k=1}^\infty x_k\phi_k: \{x_k\}_{k\in \N}\in \ell_q \ \wedge \ \sum_{k=1}^\infty |x_k \lambda_k|^q <\infty\Big\},\]
generates an analytic strongly continuous semigroup $(T(t))_{t\ge 0}$ on $X$
(see \cite[Theorem 1.3.4]{Hen81} and note that $-A$ is a sectorial operator).
This semigroup is given by $T(t)\phi_k= {\rm{e}}^{t\lambda_k}\phi_k$, for all $t\geq 0$ and $k\in\N$.

In \cite[Theorem 4.1]{JNP18} the following result has been shown:
\begin{proposition}
\label{prop:Diagonal-Systems} 
Let $\dim(U)<\infty$ and $\Uc=L_\infty(\R_+,U)$. Assume that the operator $A$ possesses a $q$-Riesz basis of $X$ that consists  of eigenvectors 
$\{\phi_k\}_{k\in \N}$  with eigenvalues $\{\lambda_k\}_{k\in\N}$ lying in a sector in the open left half-plane $\mathbb C_-$ with \linebreak  $\sup_{k\in\N}\re( \lambda_{k})<0$ and $B\in L(\C,X_{-1})$.

Then the system \eqref{eq:Linear_System} is an $L_\infty$-iISS and $L_\infty$-ISS control system in the sense of Definition~\ref{Steurungssystem}. 
\end{proposition}

\begin{proof}
Assume first that $U=\C$. By \cite[Theorem 4.1]{JNP18} it follows that $B$ is an $\infty$-admissible operator, and 
the mild solution $\phi$ of \eqref{eq:Linear_System} satisfies the ISS and iISS estimates. The latter fact implies that $B$ is a so-called 
zero-class $\infty$-admissible operator, i.e. in \eqref{eq:q-admissibility} it holds that $h_t \to 0$ as $t\to 0$.
Finally, \cite[Proposition 2.5]{JNP18} shows that $\phi$ is continuous.
This shows for $U=\C$ the triple $\Sigma:=(X,\Uc,\phi)$ is a control system in the sense of Definition~\ref{Steurungssystem} and 
furthermore $\Sigma$ is ISS.

The case of general finite-dimensional $U$ can be reduced to the one-dimensional case \cite[Proposition 4]{JSZ17}.
\end{proof}

Proposition~\ref{prop:Diagonal-Systems} gives one more (in addition to Proposition~\ref{prop:iISS-equals-ISS-analytic-case}) 
 class of systems for which $L_\infty$-iISS and $L_\infty$-ISS are equivalent concepts.

%

%
%

A simple criterion guaranteeing that $B$ belongs to $L(\C,X_{-1})$ can be found in \cite[p. 882]{JNP18}.
In particular, Proposition~\ref{prop:Diagonal-Systems} shows ISS of a one-dimensional heat equation with a Dirichlet boundary condition, see \cite{JNP18}.

\subsection{Bilinear systems}
\label{sec:Bilinear_systems}

One of the simplest classes of nonlinear control systems are bilinear systems which form a bridge between the linear and the nonlinear theories and are important in a number of applications as biochemical reactions, quantum-mechanical processes \cite{PaY08,BDK74}, reaction-diffusion-convection processes controlled by means of catalysts \cite{Kha03}, etc.

It is easy to see that most of bilinear systems are not ISS (consider Example \ref{examp:1dim_bilinear_system}), but at the same time, it was shown that all bilinear finite-dimensional 0-UGAS systems are iISS \cite{Son98}, and even strongly iISS \cite[Corollary 2]{CAI14}.
These results have been extended in \cite{MiI16, MiW15} for generalized bilinear distributed parameter systems with the bounded bilinear term. Here we present these results, and put them into the perspective of the strong iISS property. 

Consider a special case of systems \eqref{InfiniteDim} of the form
\begin{equation}
\label{BiLinSys}
\begin{array}{l}
{\dot{x}(t)=Ax(t)+ Bu(t) + C(x(t),u(t)),} \\
x(0)=x_0,
\end{array}
\end{equation}
where $B \in L(U,X)$, and $C: X \times U \to X$ satisfies the Assumption~\ref{Assumption1} and furthermore
\begin{align}
\exists \xi \in \K:\quad   x\in X,\ u\in U \qrq \|C(x,u)\|_X \leq \|x\|_X \xi(\|u\|_U).
\label{eq:BilinOperator}
\end{align}
As for systems \eqref{InfiniteDim}, we assume that inputs belong to the space $\Uc:=PC_b(\R_+,U)$.

Next we present a criterion for strong integral input-to-state stability of \eqref{BiLinSys}. The equivalence between iISS and strong iISS seems to be a new result, and therefore we provide a full proof of this part.
\begin{proposition}[{\cite[Theorem 4.2]{MiI16}, \cite[Proposition 5]{MiW15}}]
\label{ConverseLyapunovTheorem_BilinearSystems}
Let \eqref{BiLinSys} satisfy the assumption \eqref{eq:BilinOperator}. 
The following statements are equivalent:
\begin{itemize}
    \item[(i)]   \eqref{BiLinSys} is strongly iISS,
    \item[(ii)]   \eqref{BiLinSys} is iISS,
    \item[(iii)]  \eqref{BiLinSys} is 0-UGAS,
    \item[(iv)]  $A$ generates an exponentially stable semigroup,
    \item[(v)] A function $W:X \to\R_+$, defined by
    \begin{eqnarray}
W(x)=\ln \Big(1 + \int_0^{\infty} \|T(t)x\|_X^2 dt \Big),\quad x\in X
\label{eq:LF_BiLinSys_Banach}
\end{eqnarray}
is a non-coercive iISS Lyapunov function for \eqref{BiLinSys}.
\end{itemize}
\end{proposition}

\begin{proof}
Equivalences (ii) $\Iff$ (iii) $\Iff$ (iv) $\Iff$ (v) are covered by 
\cite[Theorem 4.2]{MiI16}, \cite[Proposition 5]{MiW15}.
By definition, (i) implies (ii). 

(ii) $\Rightarrow$ (i).
To show strong iISS of \eqref{BiLinSys} it remains to show that \eqref{BiLinSys} is ISS w.r.t. small inputs.
%
%
We show this by verifying that
\[
V^\gamma(x):=\max_{r\geq 0}\|e^{\gamma r}T(r)x\|_X,\quad x \in X
\]
is a coercive ISS Lyapunov function for \eqref{BiLinSys} subject to the input space $\overline{B_{R,\Uc}}=\{u\in\Uc:\|u\|_\Uc \leq R\}$ for $R>0$ small enough.
As in Theorem~\ref{thm:ISS-criterion-linear-systems-bounded-operators} (see  \eqref{eq:eqnorm}), we assume here that $\gamma \in (0,\lambda)$, where  
$\lambda>0$ is so that $\|T(t)\|\leq Me^{-\lambda t}$ for a certain $M>0$ and all $t\geq 0$ (recall, that $T$ is an exponentially stable semigroup).

First note, that $V^\gamma(x)\geq \|x\|_X$ for any $x\in X$ and that $V^\gamma$ is globally Lipschitz continuous (see \cite[Proposition 7]{MiW17c} for details).
To obtain an infinitesimal estimate, we compute, using the triangle inequality ($V^\gamma$ is a norm):
    \begin{align*}
    \dot{V}^\gamma_u(x) =& \mathop{\overline{\lim}} \limits_{h \rightarrow +0} {\frac{1}{h}\big(V^\gamma(\phi(h,x,u))-V^\gamma(x)\big) } \\
    = & \mathop{\overline{\lim}} \limits_{h \rightarrow +0}
    \frac{1}{h}\Big( V^\gamma \Big(T(h) x + \int_0^h{T(h-s) B u(s)ds} \\
		&\qquad\qquad\qquad\qquad\qquad+ \int_0^h{T(h-s) C\big(\phi(s,x,u),u(s)\big)ds}\Big) - V^\gamma(x) \Big)  \\
    \leq & \mathop{\overline{\lim}} \limits_{h \rightarrow +0}
    \frac{1}{h}\Big( V^\gamma \big(T(h) x\big) + V^\gamma\Big(\int_0^h{T(h-s) B u(s)ds}\Big) \\
		&\qquad\qquad\qquad\qquad\qquad+ V^\gamma\Big(\int_0^hT(h-s) C\big(\phi(s,x,u),u(s)\big)\Big)- V^\gamma(x) \Big).
\end{align*}				

Now let $u\in\Uc$ be so that $\|u\|_\Uc\leq R$ for a certain $R>0$ which will be specified later.
It holds that:
\begin{align*}
\mathop{\overline{\lim}} \limits_{h \rightarrow +0}&\frac{1}{h} V^\gamma\Big(\int_0^hT(h-s) C\big(\phi(s,x,u),u(s)\big)ds\Big)\\
&\leq
\mathop{\overline{\lim}} \limits_{h \rightarrow +0}\frac{1}{h} \max_{r\geq 0}e^{\gamma r}\|T(r)\|\int_0^h \|T(h-s)\| \| C\big(\phi(s,x,u),u(s)\big)\|_Xds \\
&\leq
\mathop{\overline{\lim}} \limits_{h \rightarrow +0}\frac{1}{h} \max_{r\geq 0}Me^{(\gamma - \lambda) r}\int_0^h Me^{-\lambda(h-s)} \|\phi(s,x,u)\|_X \xi(\|u(s)\|_U)ds \\
&\leq \mathop{\overline{\lim}} \limits_{h \rightarrow +0}\frac{1}{h} M^2\int_0^h \|\phi(s,x,u)\|_X \xi(\|u(s)\|_U)ds \\
& =  M^2 \|\phi(0,x,u)\|_X \xi(\|u(0)\|_U) \leq M^2 \xi(R) \|x\|_X  \leq M^2 \xi(R) V^\gamma(x).
\end{align*}		
Here we have used the continuity of $\phi$ with respect to time as well as piecewise continuity of $u$.

With this estimate and arguing as in \cite[Proposition 7]{MiW17c}, we obtain that
    \begin{align*}
    \dot{V}^\gamma_u&(x) \leq  - \gamma \ V^\gamma(x) + V^\gamma (Bu(0)) + M^2 \xi(R) V^\gamma(x).
    \end{align*}
Choosing $R>0$ so that  $M^2\xi(R) <\gamma$, we see that $V^\gamma$ is an ISS Lyapunov function for 
\eqref{BiLinSys} for inputs in $\overline{B_{R,\Uc}}$, and thus \eqref{BiLinSys} is ISS for inputs in $\overline{B_{R,\Uc}}$ by Theorem~\ref{LyapunovTheorem}.
%
\end{proof}

\begin{remark}
\label{rem:non-coercive-ISS-Lyapunov-functions} 
Note that right now there are no results proving that existence of a non-coercive iISS Lyapunov function (possibly under several further restrictions) implies iISS of a control system. The implication (v) $\Rightarrow$ (iv) has been shown by means of a standard Datko Lemma \cite[Lemma 8.1.2]{JaZ12}, \cite[Lemma 5.1.2]{CuZ95}.
\end{remark}

\begin{remark}
\label{rem:Bilinear-systems-with-admissible-operators} 
\emph{Bilinear systems with admissible control operators} have been studied in \cite{JaS18b} and the results have been applied to controlled Fokker-Planck equation. Nevertheless, iISS theory of bilinear systems remains much less developed than the ISS theory of linear systems with admissible input operators.
\end{remark}
\section{Boundary control systems}
\label{sec:Boundary_control_systems}

For many natural and engineering systems the interaction of a system with its environment (by controls, inputs, and outputs) occurs at the boundary of the system. Examples for such behavior are given by diffusion equations \cite{AWP12}, vibration of structures \cite{CuZ95}, transport phenomena, etc., with broad applications in robotics \cite{EMJ17}, aerospace engineering \cite{BCC16, PGC13}, and additive
manufacturing \cite{DiK15,HMR15}. Wide classes of port-Hamiltonian systems can be formulated as boundary control systems as well, see \cite{JaZ12,ScJ14}.

The development of the theory of general boundary control systems has been initiated in the pioneering work \cite{Fat68}, and was brought further forward by \cite{Sal87}. In the literature there are several ways how to define a boundary control system, see, e.g., \cite{Fat68}, \cite{CuZ95, JaZ12}, 
\cite{Sal87,TuW09} and \cite{EmT00}. The differences between various methods are discussed in \cite{LGE00}.
We follow here the strategy due to \cite{JaZ12}, with some motivation from \cite{EmT00}.

\subsection{Boundary control systems as systems with admissible operators}
\label{sec:BCS-as-linear-systems}

Let $X$ and $U$ be Banach spaces. Consider a system
\begin{subequations}
\label{eq:BCS}
\begin{align}
\dot{x}(t) =& {\Ah} x(t), \qquad x(0) = x_0, \label{eq:BCS-1}\\
{\Gh}x(t) =& u(t),    \label{eq:BCS-2}
\end{align}
\end{subequations}
where the \emph{formal system operator} $\Ah: D( \Ah ) \subset X \to X$ is a linear operator, the control function $u$ takes values in $U$, and the \emph{boundary operator} $\Gh : D( \Gh ) \subset X \to U$ is linear and satisfies $D(\Ah) \subset D(\Gh)$.

Equations \eqref{eq:BCS} look rather differently than the classic linear infinite-dimensional systems, studied previously: 
\begin{equation}
\dot{x}(t) = Ax(t) + Bu(t),
\quad
x(0) = x_0.
\label{eq:BCS-standard-linear-systems}
\end{equation}
where $A$ is the generator of a strongly continuous semigroup, and $B$ is either bounded or admissible input operator.

In order to use for the system \eqref{eq:BCS} the theory which we developed for linear systems \eqref{eq:BCS-standard-linear-systems}, we would like to transform \eqref{eq:BCS} into the form \eqref{eq:BCS-standard-linear-systems}. This can be done only under some additional assumptions.
\begin{definition}
\label{def:BCS}
The system \eqref{eq:BCS} is called a~\emph{boundary control system (BCS)} if the following conditions hold:
\begin{enumerate}
\item[(i)] The operator $A : D(A) \to X$ with $D(A) = D({\Ah} ) \cap \ker({\Gh})$ and
    \begin{equation}
    Ax = {\Ah}x \qquad \text{for} \quad x\in D(A)
        \label{eq:BSC-ass1}
    \end{equation}
    is the infinitesimal generator of a $C_0$-semigroup $(T(t))_{t\geq0}$ on $X$;
\item[(ii)] There is an operator $G \in \mathcal{L}(U,X)$ such that for all $u \in U$ we have $Gu \in D({\Ah})$,
    ${\Ah}G \in \mathcal{L}(U,X)$ and         
        \begin{equation}
    {\Gh}Gu = u, \qquad u\in U.
        \label{eq:BSC-ass2}
    \end{equation}
\end{enumerate}
The operator $G$ in this definition is sometimes called a \emph{lifting operator} (note that $G$ is not uniquely defined by the properties in the item (ii)).
\end{definition}

Item (i) of the definition shows that for $u\equiv 0$ the equations \eqref{eq:BCS} are well-posed.
In particular, as $A$ is the generator of a certain strongly continuous semigroup $T(\cdot)$, for any $x\in D(A)$ it holds that $T(t)x \in D(A)$ and thus $T(t)x \in \ker({\Gh})$ for all $t\geq 0$, which means that \eqref{eq:BCS-2} is satisfied.

Item (ii) of the definition implies in particular that the range of the operator ${\Gh}$ equals $U$, and thus the values of inputs are not restricted.

\begin{definition}
\label{def:Classical-solution-for-Boundary-Control-System} 
Consider a BCS \eqref{eq:BCS}.
A function $x : [0, \tau] \to X$ is called a~\emph{classical solution of \eqref{eq:BCS} on $[0,\tau]$} if $x$ is continuously differentiable, $x(t) \in D({\Ah})$ for
all $t \in [0, \tau]$, and $x(t)$ satisfies \eqref{eq:BCS} for all $t \in [0, \tau]$.

The function $x : [0,\infty) \to X$ is called a~\emph{classical solution of \eqref{eq:BCS} on $[0,\infty)$} if $x$ is a~classical solution on $[0, \tau]$ for every $\tau > 0$.
\end{definition}

The following lemma due to E. Hille will be useful, see \cite[Theorem 3.7.12]{HiP00}.
\begin{lemma}
\label{lem:Commutation_Closed_integral}
Let $X$ be a Banach space and let $A:D(A) \subset X\to X$ be a closed linear operator.
Let $f: [0,\tau] \to X$ be Bochner integrable so that $\im f \subset 
D(A)$ and $Af$ is again Bochner integrable. Then
\begin{eqnarray}
A \int_0^t f(s) ds = \int_0^t Af(s) ds.
\label{eq:Commutation_Closed_integral}
\end{eqnarray}
\end{lemma}

The next theorem gives a representation for the (unique) solution of \eqref{eq:BCS} \emph{for smooth enough inputs}.
\begin{theorem}
\label{thm:BCS-classical-solution-Representation}
Consider the boundary control system \eqref{eq:BCS}.
For all $u$ $\in$ \\$C^2([0,\tau],U)$, all $x_0\in X$: $x_0 - Gu(0) \in D(A)$ and all $\tau\geq 0$ the classical solution $\phi(\cdot,x_0,u)$ of \eqref{eq:BCS} on $[0,\tau]$ exists, is unique and can be represented as
\begin{subequations}
\label{eq:BCS-solution-for-smooth-inputs}
\begin{eqnarray}
\phantom{aaaa}\phi(t,x_0,u) &=& T(t)\big(x_0-Gu(0)\big)  + \int_0^t T(t-r)\big({\Ah}Gu(r)-G\dot{u}(r)\big) dr + Gu(t) \label{eq:BCS-solution-for-smooth-inputs-1}\\
&=& T(t)x_0 + \int_0^t T(t-r){\Ah}Gu(r)dr - A\int_0^t T(t-r)Gu(r) dr \label{eq:BCS-solution-for-smooth-inputs-2}\\
&=& T(t)x_0 + \int_0^t T_{-1}(t-r)({\Ah}G - A_{-1}G)u(r)dr \label{eq:BCS-solution-for-smooth-inputs-3},
\end{eqnarray}
\end{subequations}
where $A_{-1}$ and $T_{-1}$ are the extensions of the infinitesimal generator $A$ and of the semigroup $T$ to the extrapolation space $X_{-1}$.
Furthermore, $A_{-1}G \in L(U,X_{-1})$ (and thus ${\Ah}G - A_{-1}G \in L(U,X_{-1})$).
\end{theorem}

\begin{proof}
For \eqref{eq:BCS-solution-for-smooth-inputs-1} see \cite[Theorem 11.1.2]{JaZ12}.
For \eqref{eq:BCS-solution-for-smooth-inputs-2} see \cite[Lemma 13.1.5]{JaZ12}.
These results are stated in \cite{JaZ12} for Hilbert spaces, but the argumentation is valid for Banach spaces as well.

The representation formula \eqref{eq:BCS-solution-for-smooth-inputs-3} should be well-known for the specialists in the semigroup theory, but the authors could not find a reference in the literature and decided to include the argument to this paper.

Let $A_{-1}$ be the extension of $A$ to the extrapolation space $X_{-1}$, and let $T_{-1}$ be the extrapolated semigroup, generated by $A_{-1}$.
Note that $G\in L(U,X)$, and $D(A_{-1})=X$. Thus, the operator $A_{-1}G$ is well-defined as a linear operator from Banach space $U$ to Banach space $X_{-1}$ with $D(A_{-1}G) = U$. As $A_{-1}$ is the generator of a strongly continuous semigroup, it is closed.
Thus, by \cite[Proposition A.9]{HMM13} the operator $A_{-1}G$ is closed as a product of a closed and a bounded operator. By closed graph theorem (see, e.g., \cite[Theorem A.3.49]{CuZ95}) $A_{-1}G \in L(U,X_{-1})$.

The map $r \mapsto T(t-r)Gu(r)$ is Bochner integrable in the space $X$ and thus also in $X_{-1}$ (even for any $u \in L_{1,loc}([0,\tau],U)$, see
\cite[Proposition 1.3.4]{ABH11}).
Furthermore, $T(t-r)Gu(r) \in X = D(A_{-1})$.

Recall that $A_{-1}T_{-1}(s)=T_{-1}(s)A_{-1}$ for all $s\in\R_+$ on $D(A_{-1})$, see, e.g., \cite[Theorem 5.2.2]{JaZ12}.
Consider the map 
\[
w:r \mapsto A_{-1}T_{-1}(t-r)Gu(r) = T_{-1}(t-r)A_{-1}Gu(r).
\]
Since $A_{-1}G \in L(U,X_{-1})$ and $T_{-1}$ is a strongly continuous semigroup on $X_{-1}$, the function $w$ is 
Bochner integrable on $X_{-1}$, by \cite[Proposition 1.3.4]{ABH11}.
Hence Lemma~\ref{lem:Commutation_Closed_integral} can be applied to obtain for all $u \in C^2([0,\tau],U)$ that
\begin{align*}
A\int_0^t T(t-r)Gu(r) dr &= A_{-1}\int_0^t T_{-1}(t-r)Gu(r) dr \\
&= \int_0^t A_{-1} T_{-1}(t-r)Gu(r) dr = \int_0^t  T_{-1}(t-r)A_{-1}Gu(r) dr.
\end{align*}
From this the formula \eqref{eq:BCS-solution-for-smooth-inputs-3} follows.
\end{proof}

An advantage of the representation formula \eqref{eq:BCS-solution-for-smooth-inputs-1} is in the boundedness of the operators $G$ and $\Ah G$, involved in the expression, but the disadvantage is that the derivative of $u$ is employed.
Still, the expression in the right-hand side of \eqref{eq:BCS-solution-for-smooth-inputs-1} makes sense for any $x \in X$ 
and for any $u \in H^1([0,\tau],U)$, and can be called a mild solution of BCS \eqref{eq:BCS}, as is done, e.g., in \cite[p. 146]{JaZ12}.

The formula \eqref{eq:BCS-solution-for-smooth-inputs-3} does not involve any derivatives of inputs, and again is given in terms of a bounded operator ${\Ah}G - A_{-1}G \in L(U,X_{-1})$.
Moreover, if we consider the expression in the right-hand side of \eqref{eq:BCS-solution-for-smooth-inputs-3} in the extrapolation spaces $X_{-1}$, then it makes sense for all $x \in X$ and all $u\in L_{1,loc}(\R_+,U)$, and defines a mild solution (in the space $X_{-1}$) of the equation 
\begin{eqnarray}
\dot{x} = A_{-1}x + ({\Ah}G-A_{-1}G)u.
\label{eq:BCS:State-space-representation}
\end{eqnarray}

\begin{remark}
\label{rem:BCS-are-linear-abstract-systems} 
As we have shown, any BCS \eqref{eq:BCS} over a Banach space $X$ can be reformulated as a linear system with a bounded input operator in $X_{-1}$. 
\end{remark}

As we know from Theorem~\ref{thm:BCS-classical-solution-Representation}, 
for all $u \in C^2([0,\tau],U)$, all $x_0\in X$: $x_0 - Gu(0) \in D(A)$ and for all $\tau\geq 0$
the value of the expression in rhs of \eqref{eq:BCS-solution-for-smooth-inputs-3} belongs to $X$.

However, to ensure that the integral term in \eqref{eq:BCS-solution-for-smooth-inputs-3} belongs to $X$ for less regular $u$, we have to require that the input operator 
\begin{eqnarray}
B:={\Ah}G-A_{-1}G
\label{eq:BCS-Input-Operator}
\end{eqnarray}
has some sort of admissibility. 

\begin{remark}
\label{rem:Uniqueness-of-B-boundary-control-systems} 
The operator $B$ is uniquely defined by a boundary control system, and does not depend on the choice of the lifting operator $G$, see \cite[Proposition 2.8]{Sch20}.
\end{remark}

\begin{definition}
\label{def:Mild-solution-for-Boundary-Control-System} 
If $B$ defined by \eqref{eq:BCS-Input-Operator} is $q$-admissible for $A$, then for each $x_0\in X$ and each $u\in L_q(\R_+,U)$
the function $\phi(\cdot,x_0,u):\R_+\to X$, defined by \eqref{eq:BCS-solution-for-smooth-inputs-3} is called the \emph{mild solution of BCS \eqref{eq:BCS}}.
\end{definition}

%
%


\begin{corollary}
\label{cor:ISS-for-BCS} 
Consider a boundary control system \eqref{eq:BCS} and assume that the corresponding input operator $B$ defined by 
\eqref{eq:BCS-Input-Operator} is $q$-admissible for some $q\in [1,+\infty)$. Assume that $A$ generates an exponentially stable semigroup.
Then for any $p\in[q,+\infty]$
the system $\Sigma:=(X,\Uc,\phi)$ with $\Uc:=L_p(\R_+,U)$ is an ISS control system (with respect to the norm in $\Uc$).
\end{corollary}

\begin{proof}
$B$ defined by \eqref{eq:BCS-Input-Operator} is $q$-admissible for some $q\in [1,+\infty)$, then it is $p$-admissible for any $p\in[q,+\infty]$ and furthermore, the map $\phi$ is continuous w.r.t.\ time. 
Proposition~\ref{prop:q-admissibility-implies-continuity} $\Sigma:=(X,\Uc,\phi)$ with $\Uc:=L_p([0,\infty),U)$, for all $p\in[q,+\infty]$ is a forward-complete control system in the sense of Definition~\ref{Steurungssystem},
and by Theorem~\ref{thm:ISS-Criterion-lin-sys-with-unbounded-operators} $\Sigma$ is ISS.
\end{proof}

\begin{remark}
\label{rem:Weak solutions of boundary control systems}
In this paper we consequently use the notion of mild solutions of boundary control systems. 
In some papers the concepts of weak solutions and strong solutions are used. 
For a detailed discussion of the relationship between all these solution concepts we refer to \cite[Propositions 2.9, 2.11, Remark 2.10]{Sch20}. 
\end{remark}

\subsection{Spectral-based methods and related techniques}
\label{sec:Spectral method}

In \cite{KaK16b, KaK17a, KaK17b} the ISS analysis of linear parabolic PDEs with Sturm-Liouville operators over a 1-dimensional spatial domain has been performed using two different methods: (i) the spectral decomposition of the solution, and (ii) the approximation of the solution by means of a  finite-difference scheme. 
This made possible to avoid differentiation of the boundary disturbances, and to obtain ISS of classical solutions w.r.t.\ $L_\infty$ norm of disturbances, as well as in weighted $L_2$ and $L_1$ norms. 
An advantage of these methods is that this strategy can be applied also to other types of linear evolution PDEs. At the same time, for multidimensional spatial domains, the computations can become quite complicated.

The method initiated in \cite{KaK16b, KaK17a, KaK17b} was further developed in the monograph \cite{KaK19} to provide constructive and effective methods for the ISS analysis and the control design. Many techniques are based on dedicated Lyapunov functions. In the first part of this book semilinear hyperbolic PDEs with a constant transport velocity are studied. Two different methodologies that allow the derivation of ISS estimates for hyperbolic PDEs are presented: the ISS Lyapunov function for the PDE model and an equivalent model written by integral delay equations (IDEs). First ISS Lyapunov functions are derived providing estimates written in terms of the spatial $L_p$-norm of the state
(with $p\in(1,\infty)$). Then ISS properties are derived for hyperbolic systems given as IDEs. ISS properties are derived for this class of delay systems and Lyapunov-like functions are provided. In the second part of \cite{KaK19}, parabolic PDEs are considered. Such infinite-dimensional systems are first written in terms of the Sturm-Liouville operator and then interconnected with ODEs, globally Lipschitz nonlinearities, and non-local terms. Then some derivations of ISS estimates for both the spatial $L_2$ and $H_1$-norms are proven. Two different methodologies are given: one based on the eigenfunction expansion and the other exploiting ISS Lyapunov functions. Some ISS estimates in the spatial $L_2$-norms are first provided, allowing less regular inputs for ISS than with classical solutions. Both internal and boundary perturbations are tackled in the ISS estimates, assuming a lower bound of the principal eigenvalue of the Sturm-Liouville operator (but without the knowledge of all the set of eigenvalues). As far as the $H_1$-norm is concerned, estimates are proven with different boundary conditions and boundary disturbances, exploiting computations on the eigenvalue series. Then ISS Lyapunov functions are computed providing ISS estimates in $L_2$-norms. The boundary conditions could be of different types, as the Robin type or Dirichlet type with or without any disturbance. Finally, the last part of the book \cite{KaK19} deals with the small-gain analysis and feedback interconnections. Different possible interconnections are possible such as two PDEs (also of different nature, e.g., hyperbolic or parabolic) or one PDE and one ODE.

\subsection{Applications to Riesz-spectral systems}
\label{sec:Riesz-spectral_systems}

In Section~\ref{sec:BCS-as-linear-systems} we have shown that every BCS can be understood as a linear system with an input operator $B\in L(U,X_{-1})$ given by \eqref{eq:BCS-Input-Operator}, and if this operator is admissible, 
ISS of BCS can be shown by applying the theory developed in Section~\ref{sec:Linear_systems}, see Corollary~\ref{cor:ISS-for-BCS}.

However, the computation of the input operator $B$ using the formula \eqref{eq:BCS-Input-Operator} may be awkward in practice. Other methods for the computation of $B$ can be used, see \cite[Section 10.1]{TuW09}, \cite{EmT00} and \cite[Proposition 2.9]{Sch20}.
Furthermore, in some situations the admissibility of $B$ and ISS of \eqref{eq:BCS} can be obtained without computation of the operator $B$.

As an example we consider a class of Riesz-spectral boundary control systems, studied in \cite{LSZ18}, \cite{LhS18}.
\begin{definition}{(see \cite[Definition 2.3.1]{CuZ95})}
\label{def: Riesz-spectral operator}
Let $X$ be a Hilbert space and $A:~D(A) \subset X \rightarrow X$ be a linear, closed operator. For $n \in \N$, let $\lambda_n$ be the eigenvalues of $A$ and $\phi_n \in D(A)$ the corresponding eigenvectors. $A$ is called a \emph{Riesz-spectral operator} if
\begin{enumerate}
    \item $\left\{ \phi_n , \; n \in \mathbb{N} \right\}$ is a 2-Riesz basis;
    \item the closure of $\{ \lambda_n , \; n \in \mathbb{N} \}$ is totally disconnected, i.e.\ 
    no two points $\lambda,\mu \in \overline{\{\lambda_n:n\in\N\}}$ can be connected by a segment entirely lying in $\overline{\{\lambda_n:n\in\N\}}$.
\end{enumerate}
\end{definition}

By \cite[Theorem 2.3.5]{CuZ95}, the spectrum $\sigma(A)$ of a Riesz-spectral operator $A$ is given by $\sigma(A):=\overline{\{\lambda_n:n\in\N\}}$,
and the growth bound of a semigroup $T$, generated by $A$ can be computed as 
\begin{eqnarray}
\omega_0:= \sup_{i\in\N} \re\lambda_i <0,
\label{eq:Growth-bound}
\end{eqnarray}
that is, $A$ satisfies the \emph{spectrum determined growth assumption}.

As an application of Proposition~\ref{prop:Diagonal-Systems}, we obtain
\begin{proposition}
\label{prop:Riesz-spectral-sys} 
Let $X$ be a separable Hilbert space and consider BCS \eqref{eq:BCS}. Assume that the operator $A$ is a Riesz-spectral operator, which generates an exponentially stable analytic semigroup. 
Furthermore, assume that $\dim (U) <\infty$ and $\Uc:=L_\infty(\R_+,U)$.
Then $B$ is an admissible operator and the systems 
\eqref{eq:BCS} and \eqref{eq:BCS:State-space-representation} are ISS with these $X$ and $\Uc$.
\end{proposition}

\begin{proof}
Follows from Proposition~\ref{prop:Diagonal-Systems}.
\end{proof}

We note that in the papers \cite{LSZ18}, \cite{LhS18} a different method has been employed for the study of ISS of Riesz-spectral BCS, which is a modification of the spectral method from Section~\ref{sec:Spectral method}. 
The essence of a method is a decomposition of $X$ w.r.t.\ the Riesz basis $\{\phi_k:k\in\N\}$.


\subsection{Remark on nonlinear boundary control systems}

Well-posedness of linear boundary control systems has been studied for more than 50 years, see \cite{TuW14} for a survey. At the same time, the study of nonlinear systems with boundary controls is a much younger subject. For some recent references, we refer to \cite{TuW14} and \cite{HCZ19}.
In \cite{Sch20} a semigroup approach has been used to analyze input-to-state stability of a class of analytic boundary control systems with nonlinear dynamics and a linear boundary operator. 
Stabilization of linear port-Hamiltonian systems by means of nonlinear boundary controllers has been studied in \cite{Aug16}. 
Recently ISS stabilization of linear port-Hamiltonian systems through nonlinear boundary feedback has been investigated in \cite{ScZ18}. See also \cite{PTS16,TMP18,MPW19a} for recent results on saturated boundary control results. We are unable to cover these results in this survey due to space limitations.
\section{Lyapunov methods for ISS analysis of PDE systems}
\label{sec:ISS_analysis_linear_nonlinear_PDEs_Lyapunov_methods}

Construction of an ISS Lyapunov function is, as a rule, the most efficient and realistic method to prove ISS of nonlinear systems. However, verification of the dissipation inequality \eqref{DissipationIneq_nc} even for simple Lyapunov functions requires further tools, as inequalities for functions from $L_p$ and Sobolev spaces 
(Friedrichs, Poincare, Agmon, Jensen inequalities, and their relatives, see Appendix~\ref{sec:Inequalities}), linear matrix inequalities (LMIs), etc.

When using Lyapunov method and designing Lyapunov functions, quadratic Lyapunov candidates are often introduced to prove ISS properties, see in particular \cite{BaC16,KaK19} where various PDEs are considered. The main difference between all the Lyapunov function candidates is in the choice of the norm in the stability analysis, obtained from the Lyapunov function candidates (more specifically the choice of the space $X$ and the norm $\|\cdot\|_X$ in the Definition \ref{def:noncoercive_ISS_LF}). As an example, compare the $H^2$-norm considered for the quasilinear hyperbolic systems in \cite{CBA08} with the $W^{1,2q}_0$-norm considered for a semilinear parabolic equation in \cite{MiI15b}. 

In this section, we discuss several methods for construction of Lyapunov functions for verification of the ISS property as well for the analysis of closely related robust stability concepts. To do that, we first consider parabolic systems with distributed and boundary inputs (in Subsections~\ref{sec:ISS_analysis_linear_nonlinear_PDEs_Lyapunov_methods:parabolic},~\ref{sec:Lyapunov methods for semilinear parabolic systems with boundary inputs}), and then hyperbolic systems (in Subsections \ref{sec:ISS_analysis_linear_nonlinear_PDEs_Lyapunov_methods:hyperbolic} and \ref{sec:ISS_analysis_linear_nonlinear_PDEs_Lyapunov_methods:hyperbolic:time-varying}).
In due course, we explain how ISS Lyapunov methods can be used to design robust stabilizing controllers.
In Section~\ref{sec:Interconnected_systems} we discuss the Lyapunov method combined with the small-gain technique.

\subsection{ISS Lyapunov methods for parabolic systems}
\label{sec:ISS_analysis_linear_nonlinear_PDEs_Lyapunov_methods:parabolic}

Throughout this section, as motivated and studied in \cite{MaP11}, we consider PDEs of the form
\begin{equation}
\label{rossi}
\begin{array}{rcl}
\frac{\partial x}{\partial t}(t,z) & = & \frac{\partial^2 x}{\partial z^2}(t,z) 
+ \Delta(t,z) \frac{\partial x}{\partial z}(t,z) 
 + f(x(t,z)) + u(t,z) ,
\end{array}
\end{equation} 
with $z \in [0,1]$ and $x(t,\cdot) \in X:= L_2([0,1], \mathbb{R}^n)$ for all $t \geq 0$, 
where\linebreak $\Delta:\R_+\times [0,1] \rightarrow \mathbb{R}$ is continuous and bounded in norm, where $f:\mathbb{R}^n\rightarrow \mathbb{R}^n$ is a continuously differentiable function,  
and $u$ belongs to the space $\Uc$ of continuous in space and piecewise-continuous and right-continuous in time
 functions 
($u$ is typically unknown and represents disturbances). 

Under made assumptions the system \eqref{rossi} can be represented in a form~\eqref{InfiniteDim} for suitably defined $A$ and $f$, and 
it naturally gives rise to a control system $\Sigma:=(X,\Uc,\phi)$ as in Definition~\ref{Steurungssystem}, where $\phi$ is the (mild) solution map of \eqref{rossi}.

Similarly to Definition \ref{def:noncoercive_ISS_LF}, let us introduce the notions of weak and exponential Lyapunov functions for undisturbed systems, that we will consider in this subsection (see also \cite[Def.\/ 3.62]{LGM99}).
\begin{definition}
\label{def:Lyapu}
A continuous function $V:X \to \R_+$ is called a \emph{(coercive) weak Lyapunov function} for (\ref{rossi}), if there are functions $\psi_1,\psi_2 \in \Kinf$, such that
\begin{equation}
\psi_1(\|x\|_X) \leq V(x) \leq \psi_2(\|x\|_X), \quad \forall x \in X
\end {equation}
and the Lie derivative of $V$ along the trajectories of (\ref{rossi}) satisfies
\begin{equation}
\dot{V}_u(x) \leq 0
\end{equation}
for all $x \in X$ and when $u$ is identically equal to zero.
The function $V$ is said to 
be a {\em (coercive) exponential Lyapunov function} for (\ref{rossi}), if, additionally, for $u \equiv 0$, there exists a $\lambda_1>0$ such that, for all solutions of (\ref{rossi}), for all $t \geq 0$,
$$
\dot{V}_u(x(t)) \leq - \lambda_1 V(x(t)) .
$$
\end{definition}

\begin{remark}
Let us recall that, having a weak Lyapunov function, asymptotic stability of the undisturbed system can be often established via the celebrated Barba-shin-Krasovskii-LaSalle invariance principle (see, e.g.,\ \cite[Theorem 3.64]{LGM99}). Moreover, when an exponential Lyapunov function is known for (\ref{rossi}), then (\ref{rossi}) is 0-UGAS.
\end{remark}

\subsubsection{Constructions of Lyapunov functions}
\label{prob}

In this section, we give several constructions of Lyapunov functions for 
the system 
\begin{equation}
\label{1}
\begin{array}{rcl}
\frac{\partial x}{\partial t}(t,z) & = & \frac{\partial^2 x}{\partial z^2}(t,z) + f(x(t,z)), \quad z \in (0,1),
\end{array}
\end{equation} 
where $x\in X:= L_2([0,1],\R^n)$ and $f$ is a nonlinear function of class $C^1$. Furthermore, we suppose that the boundary conditions satisfy the assumption:
\begin{Ass}
\label{A2}
The boundary conditions are such that, for all $t\geq 0$,
\begin{equation}
\label{3}
\begin{array}{rcl}
& \mbox{either} \quad
|x(t,1)| \left|\frac{\partial x}{\partial z}(t,1)\right|
= |x(t,0)| \left|\frac{\partial x}{\partial z}(t,0)\right| = 0 , &\\
& \mbox{or} \quad
x(t,1) = x(t,0) \;  \mbox{and} \;  \frac{\partial x}{\partial z}(t,1) = 
\frac{\partial x}{\partial z}(t,0) .&
\end{array}
\end{equation}  
\end{Ass}

The problem of the proof of the existence of solutions of (\ref{1}) under Assumption~\ref{A2} 
is an important issue that has been tackled in the literature depending on the regularity of the function $f$. Consider, e.g.,\ \cite[Chap. 15]{taylor2013partialIII} 
for local (in time) existence of solution for sufficiently small (with respect to the existence time) and smooth function $f$. The global (in time) existence of solutions holds as soon as $f$ is globally Lipschitz (see \cite[Chap. 6]{Paz83}
among other references). When $f$ is superlinear, the finite escape time phenomenon 
may occur (see for instance 
\cite[Chap. 5]{bebernes2013mathematical} or \cite{merle1997stability}). In this section, we do not consider this issue and the results presented here are valid, as long as there exists a solution.

\paragraph{Weak Lyapunov function for the system (\protect\ref{1})}
\label{f3m}

To prepare the construction of coercive exponential Lyapunov functions of the forthcoming
sections, we recall how a weak Lyapunov function
can be constructed for the system (\ref{1}) under Assumption~\ref{A2}.

For the construction of a coercive exponential Lyapunov function, the following assumption is useful:
\begin{Ass}
\label{A1}
There is a symmetric positive definite matrix $Q$ such that 
the function $W_1$ defined by, for all $\xi\in \mathbb{R}^n$,
\begin{equation}
\label{2}
\begin{array}{rcl}
 W_1(\xi)  & := & - \frac{\partial V}{\partial \xi}(\xi) f(\xi),
\end{array}
\end{equation}  
with $V(\xi) = \frac{1}{2} \xi^\top Q \xi$, is nonnegative.
\end{Ass}

Some comments on Assumptions \ref{A2} and  \ref{A1} follow.
\begin{remark}
{\em 1.} Assumption \ref{A1} is equivalent to claiming 
that $V$ is a weak Lyapunov function for the ordinary 
differential equation 
\begin{equation}
\label{gk3}
\dot{\xi} = f(\xi)
\end{equation}  
with $\xi \in \mathbb{R}^n$.
This implies that \eqref{gk3} is globally stable. 

{\em 2.} Assumption \ref{A2} is satisfied in particular if the Dirichlet or Neumann 
conditions or the periodic conditions, i.e. $x(t,0) = x(t,1)$ and 
$\frac{\partial x}{\partial z}(t,0) = \frac{\partial x}{\partial z}(t,1)$ 
for all $t\geq 0$ (see \cite{chen1989convergence}), are satisfied.
 
{\em 3.} Since $Q$ is positive definite, there exist two positive real values $q_1$ 
and $q_2$ such that, for all $\xi \in \mathbb{R}^n$, 
\begin{equation}
\label{qt}
q_1 |\xi|^2 \leq V(\xi) \leq q_2 |\xi |^2 .
\end{equation}  
The constants $q_1$ and $q_2$ will be used in the constructions of coercive exponential 
Lyapunov functions we shall perform later. \end{remark}

The construction we perform below is given in \cite{MaP11}
\begin{lemma}
\label{lemma0} 
Under Assumptions \ref{A2} and \ref{A1}, the function
\begin{equation}
\label{ae1}
\begin{array}{rcl}
\widetilde V(x) & = & \displaystyle\int_{0}^{1} V\big(x(z)\big) dz
\end{array}
\end{equation}  
is a weak Lyapunov function whose derivative 
along the trajectories of (\ref{1}) satisfies
\begin{equation}
\label{9}
\begin{array}{rcl}
\dot {\widetilde V}(x) & = & 
- \displaystyle\int_{0}^{1} \frac{\partial x^\top}{\partial z}(t,z)  Q \frac{\partial x}{\partial z}(t,z) dz
- \displaystyle\int_{0}^{1}  W_1\big(x(t,z)\big) dz .
\end{array}
\end{equation}  
\end{lemma}

\paragraph{Coercive exponential Lyapunov function for the system (\protect\ref{1}): first result}
\label{secs2}

In this paragraph, we show that the function $\widetilde V$ given in (\ref{ae1}) is a coercive exponential Lyapunov function for (\ref{1}) when this system is associated with special families of boundary conditions or when $W_1$ is larger than a positive definite quadratic function. 
In \cite{MaP11}, the following result is proven: 
\begin{theorem}
\label{theorem9}
Assume that the system (\ref{1}) satisfies Assumptions \ref{A2} and \ref{A1} 
and that one of the following properties is satisfied:

(i) there exists a constant $\underline{\alpha} > 0$ such that, for all $\xi \in \R^n$,
$$
W_1(\xi) \geq \underline{\alpha} |\xi|^2 ,
$$
 
(ii) $x(t,0) = 0$ for all $t \geq 0$, 

(iii) $x(t,1) = 0$ for all $t \geq 0$.

\noindent
Then the function $\widetilde V$ given in (\ref{ae1}) is a coercive exponential Lyapunov function for the system (\ref{1}).
\end{theorem}

\paragraph{Coercive exponential Lyapunov function for the system (\protect\ref{1}): second result}
\label{secs3}

One can check easily that Assumptions \ref{A2} and \ref{A1} alone do not 
ensure that the system (\ref{1}) admits the zero solution as an asymptotically 
stable solution.\footnote{
More precisely, we can construct examples of 
systems (\ref{1}) which are not asymptotically stable when Assumption \ref{A1} is satisfied and 
Assumption \ref{A2} holds with the Neumann boundary conditions. For example the system $\frac{\partial x}{\partial t}=\frac{\partial ^2 x}{\partial z^2}$ where $x(t,z)\in \mathbb{R}$ with Neumann boundary conditions at $z=0$ and $z=1$ admits all constant functions as solutions and thus it is not asymptotically stable in $L_2$-norm.}
Therefore an extra assumption must be introduced to guarantee that a coercive exponential Lyapunov function exists. 
In Section~\ref{secs2} we have exhibited simple conditions which ensure that $\widetilde V$ is a coercive exponential Lyapunov function. In this section, we introduce 
a new assumption, less restrictive than the condition (i) of 
Theorem \ref{theorem9},  which ensures that a coercive exponential Lyapunov
function different from $\widetilde V$ can be constructed. \\

\begin{Ass}
\label{A3} 
There exist a nonnegative function $M:\R^n \rightarrow \R$ of class $C^2$, and a continuous function $W_2:\R^n \rightarrow \R$ such that 
\begin{equation}
\label{2bis}
M(0) = 0 \; , \; \frac{\partial M}{\partial \xi}(0) = 0  ,
\end{equation}  
\begin{equation}
\label{2las}
\frac{\partial M}{\partial \xi}(\xi) f(\xi) \leq - W_2(\xi) 
\; , \; \forall \xi \in \mathbb{R}^n ,
\end{equation}  
\begin{equation}
\label{las}
\left|\frac{\partial^2 M}{\partial \xi^2}(\xi)\right| \leq \frac{q_1}{2}
 \; , \; \forall \xi \in \mathbb{R}^n ,
\end{equation}  
and there exists a constant $q_3 > 0$ such that 
$W_1 + W_2$ is positive definite and
\begin{equation}
\label{qtrois}
W_1(\xi) + W_2(\xi) \geq q_3 |\xi|^2  \; , \; \forall \xi \in \mathbb{R}^n : |\xi| \leq 1 ,
\end{equation}  
where $W_1$ is the function defined in (\ref{2}).
\end{Ass}

We are ready to state the following result (see \cite{MaP11} for a proof): 
\begin{theorem}
\label{theorem1}
Under Assumptions \ref{A1} to \ref{A3}, there exists a function $k \in\Kinf$, of class $C^2$ such that $k'$ is positive, 
$k''$ is nonnegative and the function 
\begin{equation}
\label{60}
\overline{V}(x) = \displaystyle\int_{0}^{1} k\Big(V\big(x(z)\big) + M\big(x(z)\big)\Big) dz
\end{equation}  
is a coercive exponential Lyapunov function for (\ref{1}).
\end{theorem}

\begin{remark} Assumption \ref{A3} seems to be restrictive. In fact, it can 
be significantly relaxed. Indeed, if the system 
\begin{equation}
\label{1bd}
\dot{\xi} = f(\xi)
\end{equation}
is locally exponentially stable and  satisfies one of Matrosov's conditions 
which ensure that a coercive exponential Lyapunov 
function can be constructed then one can construct a function $M$ which satisfies 
Assumption \ref{A3}. 
For constructions of coercive exponential Lyapunov functions under Matrosov's conditions,
the reader is referred to \cite{MalisoffMazenc:book:09}.
\end{remark}

\subsubsection{ISS property for a family of PDEs}
\label{inde}

In the previous section, we presented a construction of Lyapunov functions for parabolic PDEs without 
uncertainties and without convection term. In this section, we show how this technique 
can be used to estimate the impact of uncertainties on the solutions of 
PDEs with a convection term and uncertainties of the form \eqref{rossi} with $\Delta(t,z):=D_1 +v(t,z)$:
\begin{eqnarray}
\frac{\partial x}{\partial t}(t,z)  =  \frac{\partial^2 x}{\partial z ^2}(t,z) 
+ \big[D_1 + v(t,z)\big]\frac{\partial x}{\partial z}(t,z) + f\big(x(t,z)\big) + u(t,z),
\label{34}
\end{eqnarray} 
where $D_1$ is a constant matrix, $v$ is an unknown continuous matrix function
and $u\in\Uc$, where $X:=L_2([0,1],\R^n)$ and $\Uc$ is as defined in the beginning of the Section~\ref{sec:ISS_analysis_linear_nonlinear_PDEs_Lyapunov_methods:parabolic}.

To cope with the presence of a convection term and the uncertainty $v$, 
we introduce the following assumption: 
\begin{Ass}
\label{B9}  
There exists a nonnegative real number 
$\delta$ such that 
\begin{equation}
\label{isis}
|v(t,z)| \leq \frac{\delta}{\|Q\|} \; , \; \forall z \in [0,1] \; , t \geq 0 , 
\end{equation} 
where $Q$ is the  symmetric positive definite matrix in Assumption \ref{A1}.
Moreover, the matrix $Q D_1$ is symmetric.
\end{Ass}
Moreover, we replace Assumption \ref{A3} by a more restrictive assumption: 
\begin{Ass}
\label{C3} 
There exists a nonnegative function $M : \R^n\rightarrow\R$ such 
that, for all $\xi \in \mathbb{R}^n$,
\begin{equation}
\label{gis}
\begin{array}{rcl}
M(0) = 0 \; , \quad 
\frac{\partial M}{\partial \xi}(\xi) f(\xi) & = & - W_2(\xi) ,
\end{array}
\end{equation}  
where $W_2$ is a nonnegative function and there exist $c_a > 0$, $c_b > 0$ and $c_c > 0$ 
such that, for all $\xi \in \mathbb{R}^n$,  the inequalities
\begin{equation}
\label{a3b}
\left|\frac{\partial M}{\partial \xi}(\xi)\right| \leq c_a |\xi|
\; , \;  
\left|\frac{\partial^2 M}{\partial \xi ^2}(\xi)\right| \leq c_b ,
\end{equation}  
\begin{equation}
\label{14}
|\xi|^2 \leq c_c [W_1(\xi) + W_2(\xi)] ,
\end{equation}  
where $W_1$ is the function defined in (\ref{2}), are satisfied.
\end{Ass}

\begin{remark}
If $f$ is linear and $\dot{\xi} = f(\xi)$ is exponentially stable, then 
Assumption~\ref{C3} is satisfied with a positive definite quadratic function 
as function $M$.
\end{remark}

We are ready to state the following result (see \cite{MaP11} for a proof).
\begin{theorem}
\label{theorem6}
Assume that the system (\ref{34}) satisfies Assumptions \ref{A1}, \ref{B9}
and \ref{C3} and is associated with boundary conditions satisfying
\begin{equation}
\label{dirichlet}
x(t,1) = x(t,0) \; \ \mbox{ and } \; \ \frac{\partial x}{\partial z}(t,1) = 
\frac{\partial x}{\partial z}(t,0) \; , \; \forall t \geq 0.
\end{equation}  
Then the function 
\begin{equation}
\label{60bis}
\overline{V}(x) = \displaystyle\int_{0}^{1} \big[K V(x(z)) + M(x(z))\big] dz
\end{equation}  
with 
\begin{equation}
\label{tres}
K = \max\left\{1, \frac{2 c_b}{q_1}, \frac{8 c_c c_a^2 (\|D_1\| + 1)^2}{q_1}\right\}
\end{equation}  
is a coercive ISS Lyapunov function satisfying, along the trajectories of (\ref{34}), 
\begin{equation}
\label{1isi}
\dot{\overline{V}}(x) \leq - \lambda_1\overline{V}(x(t,z)) 
+ \lambda_2 \displaystyle\int_{0}^{1} |u(t,z)|^2 dz
\end{equation}  
for some positive constants $\lambda_1$, $\lambda_2$, provided that $\delta$ in
Assumption \ref{B9} satisfies 
\begin{equation}
\label{tres2}
\delta \leq \min\left\{\|Q\|, \frac{\sqrt{q_1}}{2\sqrt{2 c_c K}}\right\} .
\end{equation}  
\end{theorem}

\subsection{Lyapunov methods for semilinear parabolic systems with boundary inputs}
\label{sec:Lyapunov methods for semilinear parabolic systems with boundary inputs}

In this section, we show how the Lyapunov method can be applied for the analysis of semilinear parabolic systems with \emph{boundary inputs of Neumann type}. To the knowledge of authors, the first results of this kind have been reported in \cite{ZhZ18}. We present the method from \cite{ZhZ18} on a representative example, which shows the essence of the technique. 
Consider the Ginzburg-Landau equation (see \cite[p. 276]{ZhZ18})
\begin{eqnarray}
\frac{\partial x}{\partial t}(t,z) = \mu \frac{\partial^2 x}{\partial z^2}(t,z) + a x(t,z) - x^3(t,z),\quad t>0,\ z\in(0,1),
\label{eq:Semilin-parabolic-Neumann-input}
\end{eqnarray}
where $\mu>0$ is the diffusion coefficient. We investigate \eqref{eq:Semilin-parabolic-Neumann-input} subject to the Neumann boundary input at $z=0$ and to Dirichlet boundary condition at $z=1$:
\begin{subequations}
\label{eq:Neumann-input-parabolic-PDE}
\begin{eqnarray}
\frac{\partial x}{\partial z}(t,0) &=& u(t) ,\quad t>0, \label{eq:Neumann-input-parabolic-PDE-1}\\
x(t,1)&=& 0, \quad t>0. \label{eq:Neumann-input-parabolic-PDE-2}
\end{eqnarray}
\end{subequations}

For consistency with the original paper we assume that $u \in \Uc:=C^2(\R_+,\R)$, and the solutions we understand in a classical sense, 
but the norm in the state space will be chosen as $L_2(0,1)$-norm
and the norm in the input space $\Uc$ will be chosen as $L_\infty(\R_+)$-norm.
The use of classical solutions calls for adaptation of the ISS concept, and hence we require for the ISS property the validity of the estimate \eqref{iss_sum} only for smooth enough $x$.
The precise definition of the solution concept, state space as well as the proof of well-posedness of the PDE model \eqref{eq:Semilin-parabolic-Neumann-input}, \eqref{eq:Neumann-input-parabolic-PDE}
can be found in \cite{ZhZ18}.

In order to derive conditions for ISS of 
\eqref{eq:Semilin-parabolic-Neumann-input}, \eqref{eq:Neumann-input-parabolic-PDE}, we use the Lyapunov function candidate
\begin{eqnarray}
V(x):=\|x\|^2_{L_2(0,1)} = \int^1_0 x^2(z) dz.
\label{eq:L2_LF_Linear_diffusion}
\end{eqnarray}


First let us compute the Lie derivative of $V$ for $x$ smooth enough:
\begin{eqnarray}
\label{eq:Partial-integration}
\dot{V}(x) &=& 2\int^1_0 x(z)\frac{\partial x}{\partial t}(z) dz =  2 \int^1_0 x(z) \Big( \mu \frac{\partial^2 x}{\partial z^2}(z) + a x(z) - x^3(z)\Big) dz \nonumber\\
&=&  2\mu \Big(\frac{\partial x}{\partial z}(z)x(z)\Big)\Big|_{z=0}^{z=1} - 2\mu \int^1_0 \Big( \frac{\partial x}{\partial z}(z)\Big)^2 dz + 2aV(x) - 2 \int^1_0 \big(x^2(z)\big)^2dz.
\end{eqnarray}

Using \eqref{eq:Neumann-input-parabolic-PDE} as well as Jensen's inequality \eqref{ineq:Jensen} for the last term, we obtain 
\begin{eqnarray*}
\dot{V}(x)
 &\leq & 2\mu |u| |x(0)|  - 2\mu \int^1_0 \Big( \frac{\partial x}{\partial z}(z)\Big)^2 dz + 2aV(x) - 2 V^2(x).
\end{eqnarray*}
Using Cauchy's inequality \eqref{ineq:Cauchy-with-eps} and afterwards Agmon's inequality \eqref{ineq:Agmon} (here we make use of \eqref{eq:Neumann-input-parabolic-PDE-2})  we obtain for any $\varepsilon>0$:
\begin{eqnarray}
\hspace{10mm} 2\mu |u| |x(0)| \leq \varepsilon |x(0)|^2 + \frac{\mu^2}{\varepsilon} |u|^2
 \leq \varepsilon \Big(\|x\|^2_{L_2(0,1)} + \Big\|\frac{dx}{dz}\Big\|^2_{L_2(0,1)}\Big) 
+ \frac{\mu^2}{\varepsilon} |u|^2. 
\label{eq:Estimate-boudnary-term}
\end{eqnarray}
Using this estimate in \eqref{eq:Partial-integration} and rearranging the terms we obtain
\begin{eqnarray}
\label{eq:Partial-integration-2}
\dot{V}(x) 
&\leq& (\varepsilon - 2\mu )\int^1_0 \Big( \frac{\partial x}{\partial z}(z)\Big)^2 dz + (\varepsilon + 2a)V(x) - 2 V^2(x) + \frac{\mu^2}{\varepsilon} |u|^2 
\end{eqnarray}
Assuming that $\varepsilon < 2\mu$, we can use Poincare's inequality \eqref{Wirtinger_Variation_Ineq} for the first term to obtain that
\begin{eqnarray}
\label{eq:Partial-integration-3}
\dot{V}(x) 
&\leq& (\varepsilon - 2\mu )\frac{\pi^2}{4}\int^1_0 x^2(z) dz + (\varepsilon + 2a)V(x) - 2 V^2(x) + \frac{\mu^2}{\varepsilon} |u|^2 \nonumber \\
&\leq& \Big((\varepsilon - 2\mu )\frac{\pi^2}{4} + \varepsilon + 2a\Big)V(x) - 2 V^2(x) + \frac{\mu^2}{\varepsilon} |u|^2.
\end{eqnarray}
To ensure that the dissipation inequality holds, we have to assume that \linebreak
 $(\varepsilon - 2\mu )\frac{\pi^2}{4} + \varepsilon + 2a \leq 0$. As $\varepsilon>0$ can be chosen arbitrarily small, we obtain the following sufficient condition for ISS of the system 
\eqref{eq:Semilin-parabolic-Neumann-input}, \eqref{eq:Neumann-input-parabolic-PDE}:
\begin{eqnarray}
a< \frac{\mu\pi^2}{4}.
\label{eq:Obtained-criterion}
\end{eqnarray}
The term $- 2 V^2(x)$ in \eqref{eq:Partial-integration-3} shows that outside of the neighborhood of the equilibrium the convergence rate of the system \eqref{eq:Semilin-parabolic-Neumann-input}, \eqref{eq:Neumann-input-parabolic-PDE}
is faster than exponential.

\subsubsection{Tightness of obtained estimates}
\label{sec:Tightness-of-results} 

Linearizing the system \eqref{eq:Semilin-parabolic-Neumann-input}-\eqref{eq:Neumann-input-parabolic-PDE} for $u\equiv 0$ near the equilibrium, one can see that (we omit the standard computations), for
$a$ and $\mu$ as in \eqref{eq:Obtained-criterion}, the linearized system is exponentially stable, and for 
$a$ and $\mu$ satisfying
\begin{eqnarray}
a> \frac{\mu\pi^2}{4}
\label{eq:Obtained-criterion-unstable}
\end{eqnarray}
the system \eqref{eq:Semilin-parabolic-Neumann-input}-\eqref{eq:Neumann-input-parabolic-PDE} is unstable (even locally near the equilibrium).
This can be done, e.g.,\ by analyzing the spectral properties of the corresponding Laplacian operator, similarly to \cite[Section 1.3]{Hen81}.

Obtained results we summarize into the following proposition:
\begin{proposition}
\label{prop:Semilinear-Parabolic-system-with-Dirichlet-boundary-input} 
Consider the system \eqref{eq:Semilin-parabolic-Neumann-input}-\eqref{eq:Neumann-input-parabolic-PDE}
with $X:=L_2(0,1)$ and $\Uc:=C^2(\R_+,\R)$ (with the supremum norm).

If the condition \eqref{eq:Obtained-criterion} holds, then \eqref{eq:Semilin-parabolic-Neumann-input}-\eqref{eq:Neumann-input-parabolic-PDE} is ISS.

If the condition \eqref{eq:Obtained-criterion-unstable} holds, then 
\eqref{eq:Semilin-parabolic-Neumann-input}-\eqref{eq:Neumann-input-parabolic-PDE} is not 0-UAS.
\end{proposition}

Thus our analysis (motivated by \cite{ZhZ18}) provides fairly tight results for ISS of the system \eqref{eq:Semilin-parabolic-Neumann-input}-\eqref{eq:Neumann-input-parabolic-PDE}. This shows that the Lyapunov method combined with the sharp versions of inequalities for the $L_p$ spaces, as Poincare's, Agmon's inequalities, etc., is an efficient method for analyzing ISS of semilinear systems with boundary inputs and for computation of the precise uniform decay rate of the solutions of a system.

Note that the choice of applied inequalities depends on the type of the system which is to be analyzed. 
Friedrich's and Poincare's inequality are tailored for the Laplacian operator, which is what we need in our case.
At the same time, for sharp analysis of, e.g., Kuramoto-Sivashinskiy equation other types of inequalities can be used, see, e.g.,\ \cite{LiK01}.
Derivation of such inequalities uses spectral properties of the linear operators which generate the semigroups of the systems which we analyze. 
This is one of the ways how \emph{linear} operator theory helps us to study stability of \emph{nonlinear} systems.

\subsubsection{Generalizations} 

This strategy has been extended in \cite{ZhZ18} for a class of systems of the form 
\begin{eqnarray}
\frac{\partial x}{\partial t} = \mu \frac{\partial^2 x}{\partial z^2} + f\Big(t, z, x, \frac{\partial x}{\partial z}\Big),\quad 
t >0,\ z \in [0,1].
\label{eq:Semilin-parabolic-General}
\end{eqnarray}
with the boundary conditions
\begin{subequations}
\label{eq:General-BC-parabolic-PDE}
\begin{align}
a_0 &x(t, 0) + b_0 \frac{\partial x}{\partial z}(t,0) = u (t), \quad t>0,   \label{eq:General-BC-parabolic-PDE-1}\\
a_1 &x(t, 1) + b_1 \frac{\partial x}{\partial z}(t,1) = 0 ,\quad t>0.  \label{eq:General-BC-parabolic-PDE-2}
\end{align}
\end{subequations}

Under certain assumptions on the coefficients $a_0,a_1,b_0,b_1 \in\R$ and provided \linebreak
$|f(t, z, x, \frac{\partial x}{\partial z})\cdot x(t,z)|$ has at most quadratic growth, the conditions for ISS of the system \eqref{eq:Semilin-parabolic-General}, \eqref{eq:General-BC-parabolic-PDE} in $X:=L_2(0,1)$ have been derived, see \cite[Theorems 4,5]{ZhZ18}.

A key step in the treatment of the Ginzburg-Landau equation was tackling of a boundary input $u$, which was done using Agmon's inequality, where we have essentially used that $x(1,t) = 0$ for all $t > 0$. 
In order to allow for more general types of boundary conditions \eqref{eq:General-BC-parabolic-PDE}, in \cite[Lemma 1]{ZhZ18} the following inequality has been used:
\begin{lemma}
\label{lem:ZhZ18-inequality}
Suppose that $x \in C^1([a,b],\R)$. Then the following inequalities hold: 
\begin{eqnarray}
\max_{z\in[a,b]}|x(z)|^2\leq \frac{2}{b-a} \|x\|^2_{L_2([a,b])} + (b-a)\Big\|\frac{dx}{dz}\Big\|^2_{L_2(a,b)}.
\label{eq:ZhZ18-inequality}
\end{eqnarray}
\end{lemma}
Its advantage is that the estimate in the right-hand side of \eqref{eq:ZhZ18-inequality} does not depend on the values of $x$ at the boundary, but just on the $L_2$-norm of $x$ and $\frac{dx}{dz}$.

\subsubsection{Remarks on Dirichlet boundary inputs}
\label{sec:Dirichlet-Boundary-input}

An important limitation of the method discussed in this section is that it cannot be applied to the parabolic systems with Dirichlet boundary inputs (i.e. the case $b_0=b_1=0$ is not allowed), even if the system is linear. To see the difficulties arising on this way consider the equation \eqref{eq:Semilin-parabolic-Neumann-input} with the boundary conditions \eqref{eq:Neumann-input-parabolic-PDE-2} and 
\begin{eqnarray}
x(t,0) = u(t) ,\quad t>0 
\label{eq:Dirichlet-input-parabolic-PDE}
\end{eqnarray}
instead of \eqref{eq:Neumann-input-parabolic-PDE-1}.

Doing the same steps as in \eqref{eq:Partial-integration}, we obtain that 
\[
\Big(\frac{\partial x}{\partial z}(z)x(z)\Big)\Big|_{z=0}^{z=1} = - \frac{\partial x}{\partial z}(0)u \leq \Big|\frac{\partial x}{\partial z}(0)\Big| |u|
\]
and in contrast to the example which we have treated above, we cannot bound the value $|\frac{\partial x}{\partial z}(0)|$ by the $H_1$-norm of $x$, which breaks the argumentation.

Looking at this problem from the functional-analytic point of view, we obtain that the input operator of the linearized version of \eqref{eq:Semilin-parabolic-Neumann-input}, \eqref{eq:Neumann-input-parabolic-PDE} (after recasting it as an abstract linear system \eqref{eq:Linear_System}) is 2-admissible, while the input operator of a system with a Dirichlet input \eqref{eq:Semilin-parabolic-Neumann-input}, \eqref{eq:Neumann-input-parabolic-PDE-2}, \eqref{eq:Dirichlet-input-parabolic-PDE} is not 2-admissible, which makes the analysis more challenging.

Currently there are no coercive Lyapunov functions known for the linearized system \eqref{eq:Semilin-parabolic-Neumann-input}, \eqref{eq:Neumann-input-parabolic-PDE-2}, \eqref{eq:Dirichlet-input-parabolic-PDE},
however using Corollary~\ref{cor:Self-adjoint-A} it is possible to construct a non-coercive one, see Remark~\ref{rem:Analyticity-of-a-self-adjoint-generator}.

\subsubsection{Remarks on Burgers' equation}
\label{sec:Burgers-eq}

Consider the Burgers' equation 
\begin{eqnarray}
\frac{\partial x}{\partial t}(t,z) = \mu \frac{\partial x^2}{\partial z^2}(t,z) - \nu x(t,z)\frac{\partial x}{\partial z}(t,z),\quad t>0,\ z\in(0,1),
\label{eq:Burgers-viscous}
\end{eqnarray}
where $\mu>0$ is the diffusion coefficient and $\nu>0$, subject to boundary conditions

\begin{subequations}
\label{eq:Burgers-Dirichlet-input}
\begin{eqnarray}
x(t,0) &=& 0 ,\quad t>0, \label{eq:Burgers-Dirichlet-input-1}\\
x(t,1)&=& u(t), \quad t>0. \label{eq:Burgers-Dirichlet-input-2}
\end{eqnarray}
\end{subequations}

Burgers' equation is known first of all as a simplification of the Navier-Stokes equations. However, it has also important applications to aerodynamics \cite{Col51}. An extensive treatment of Burgers' equation and its generalizations, one can find in \cite{Sac87}.

It is easy to show that \eqref{eq:Burgers-viscous}-\eqref{eq:Burgers-Dirichlet-input}
is 0-UGAS in $L_2(0,1)$-norm, by showing that $V(x) = \int_0^1x^2(z)dz$ is a strict Lyapunov function for
\eqref{eq:Burgers-viscous}-\eqref{eq:Burgers-Dirichlet-input}.
At the same time, ISS analysis of this system is much more involved. 

Using De Giorgi iteration method in \cite[Theorems 4,5]{ZhZ19} it was shown that 
\eqref{eq:Burgers-viscous}-\eqref{eq:Burgers-Dirichlet-input}
has a so-called ISS with respect to small inputs property (terminology is taken from \cite{CAI14}),
as the ISS property in \cite[Theorems 4,5]{ZhZ19} is verified under condition that
$\sup_{t\geq 0}|u(t)| \leq \frac{\mu}{\nu}$.

\begin{openprob}
\label{op:Burgers}
Is it possible to prove ISS of Burgers' equation \eqref{eq:Burgers-viscous}-\eqref{eq:Burgers-Dirichlet-input} with $X:=L_2(0,1)$ and $u \in PC_b(\R_+,\R)$ (or for another input space) without restrictions on the magnitude of inputs? 
\end{openprob}

\subsection{ISS Lyapunov methods for stabilization of stationary hyperbolic systems}
\label{sec:ISS_analysis_linear_nonlinear_PDEs_Lyapunov_methods:hyperbolic}

We consider the feedback boundary control for the class of linear hyperbolic PDEs described by the equation
\begin{subequations}\label{eq:sysHyp}
\begin{equation}\label{eq:sysHyp:dyna}
\frac{\partial x}{\partial t} (t,z) +\Lambda \frac{\partial x}{\partial z} (t,z) = 0,
\end{equation}
where $z \in [0,1]$ is the spatial variable, and $t \in \R_+$ is the time variable. The matrix $\Lambda$ is assumed to be diagonal and positive definite. We call $x:[0,1] \times \R_+ \rightarrow \R^n$ the state trajectory.
The initial condition is defined as
\begin{equation}
x(0,z) = x^0(z), \quad z \in (0,1)
\end{equation}
\end{subequations}
for some function $x^0 : (0,1) \rightarrow \R^n$.
The value of the state $x$ is controlled at the boundary $z=0$ through some input $u: \R_+ \to \R^{m}$ so that
\begin{equation}\label{eq:initCond}
x (t,0) = H x(t,1) + B u(t),
\end{equation}
where $H \in \R^{n \times n}$ and $B \in \R^{n \times m}$ are constant matrices.
The system~\eqref{eq:sysHyp}-\eqref{eq:initCond} forms a class of 1-D boundary controlled hyperbolic PDEs, for which several fundamental results can be found in \cite{BaC16}.

We consider the case when only the measurement of the state $x$ at the boundary point $z=1$ is available for control purposes at each $t \ge 0$, and this measurement is subject to some bounded disturbance.
We thus denote the output of the system by
\begin{equation}\label{eq:defOut}
y(t) = x(t,1) + d(t),
\end{equation}
where the disturbance $d \in {L}_\infty(\R_+,\R^n)$ may arise due to low resolution of the sensors, uncertain environmental factors, or errors in communication.

\subsubsection{Problem formulation}
\label{sec:prelim}

We are interested in designing a feedback control law $u$ as a function of the output measurement $y$, that is $u = \mathcal{F}(y)$ for some operator $\mathcal{F}$, which makes the closed-loop system robust with respect to the measurement disturbances in the sense of {\em disturbance-to-state stability (DSS)}\label{pageref:DSS}, which means that there are some constants $a, c, M_{x^0} > 0$, and some $\gamma\in\Kinf$ such that
\begin{equation}\label{eq:maxEstGen}
\max_{z\in [0,1]} \vert x(t,z) \vert \le c\, e^{-at} M_{x^0} + \gamma \left( \|d_{[0,t]}\|_\infty \right), \qquad t\geq 0.
\end{equation}

Here $\|d_{[0,t]}\|_\infty:=\esssup_{s\in[0,t]}|d(s)|$, and for given $z$ and $t$, $\vert x(t,z) \vert$ denotes the usual Euclidean norm of $x(t,z) \in \R^n$. The constant $M_{x^0}$ is such that it depends on some norm associated with the function $x^0$ and possibly the initial state chosen for the dynamic compensator $u$.


The DSS property ensures that for any initial conditions and in the absence of disturbances, that is for $d \equiv 0$, the maximum norm of 
$x$ (with respect to the spatial variable) decreases exponentially in time with a uniform decay rate. In the presence of nonzero disturbances, that is $d \not\equiv 0$, the maximum value of $x$ over $[0,1]$, at each time $t \ge 0$, is bounded by the maximum norm of the disturbance over the interval $[0,t]$ and an exponentially decaying term due to the initial condition of the system. 
    
\begin{remark}
\label{rem:ISS-and-DSS} 
In case if $u$ is a static controller, and $M_{x^0}$ is a $\K$-function of the $C$-norm of the state, DSS becomes precisely  exponential ISS. 
If $u$ is a dynamic controller with the internal state $\eta$, the constant $M_{x^0}$ in the definition of DSS may depend on the initial state of the controller as well, and thus DSS property guarantees a stability only of the $x$-component of the full state $(x,\eta)$, and does not tell anything about the stability of the controller subsystem.
Furthermore, the constant $M_{x^0}$ in the definition of DSS depends possibly on another type of a norm than $C$-norm and 
thus DSS is related to a kind of stability with respect to two measures concept, see, e.g.,\ \cite{TeP00}.
Hence, in general, DSS can be understood as a kind of robust stability with respect to partial state and two measures.
\end{remark}

\paragraph{Using static controllers}

In case there is no perturbation, that is, if $d \equiv 0$, one typically chooses $u(t) = K y(t)$ such that the closed-loop boundary condition
\begin{equation}\label{eq:initCondLoop}
x(t,0) = (H+BK) x(1,t)
\end{equation}
satisfies a certain dissipative condition.
This control law yields asymptotic stability of the system with respect to ${H}_2$-norm \cite{CBA08}, or ${C}^1$-norm \cite{CoronBastin15}, depending on the dissipativity criterion imposed on $H+BK$.
In the presence of perturbations $d \not \equiv 0$, one has to modify the stability criteria as the asymptotic stability of the origin can no longer be established.

One finds the Lyapunov stability criteria with ${L}_2$-norm and dissipative boundary conditions in \cite{BastinCoronAndreaIFAC08}. Lyapunov stability in ${H}_2$-norm for nonlinear systems is treated in \cite{CBA08}. Thus, the construction of Lyapunov functions in ${H}_2$-norm for the hyperbolic PDEs with static control laws can be found in the literature.

\paragraph{Need for dynamic controllers}

For hyperbolic systems, when seeking robust stabilization with measurement errors, one could see that the results in \cite{EspiGira16} provide robust stability of $x(\cdot,t)$ in ${L}_2((0,1),\R^n)$ space by using static controllers and piecewise continuous solutions. However, the DSS estimate \eqref{eq:maxEstGen} requires stability in ${C}([0,1],\R^n)$ space equipped with maximum norm, and the stabilization of the closed-loop system in $L_2$ norm does not provide us with this estimate.
Therefore we choose $X:={H}_1((0,1),\R^n)$ as the state space for the open-loop system. This choice is motivated by Agmon's inequality
\eqref{theorem:Agmon}

\begin{remark}
Agmon's inequality \eqref{theorem:Agmon} allows to get the bounds on ${C}$-norm of the state $x$ in terms of its ${H}_1$-norm. This fact is frequently used in the ISS literature, see, e.g.,\ \cite[Remark~4]{MiI15b}.
The second advantage of the $H^1$-norm (in contrast to $C$-norm) is that it is much easier to construct and differentiate the Lyapunov functions if their argument belongs to the $H^1$-space.
\end{remark}

One can obtain the estimate \eqref{eq:maxEstGen} from the Agmon's inequality \eqref{ineq:Agmon}, by ensuring that the control input $u$ is chosen such that for each $t \ge 0$:
\begin{itemize}
\item The solution $x(t)$ belongs to ${H}_1((0,1),\R^n)$;
\item It holds that $\vert x(t,0)\vert$ and $\| x(t) \|_{{H}_1((0,1),\R^n)}$ are bounded by the size of the disturbance $\|d_{[0,t]}\|_\infty$ plus some exponentially decreasing term in time.
\end{itemize}

To achieve these objectives, the use of static controllers of the form $u(t) = K y(t)$, will result in solutions $x$ which are not differentiable with respect to spatial variable due to (possibly discontinuous) disturbances, and hence the solutions are not contained in ${H}_1((0,1),\R^n)$. To remedy this problem, we propose the use of dynamic controllers for stabilization, which smoothen the discontinuity effect of the perturbations
 (see \cite{TPT16} for the use of a static controller).

\paragraph{Using dynamic controller for ${H}_1$-regular solutions}

\begin{figure}[ht!]
\centering
\begin{tikzpicture}[xscale = 0.5, yscale = 0.6, circuit ee IEC, every info/.style={font=\footnotesize}]
\draw (-0.5,0) node [rectangle, rounded corners, draw, minimum height =0.65cm, text centered] (sys) 
{$ \left\{ \begin{aligned}& \partial_t X (z,t) +\Lambda \partial_z X (z,t) = 0\\ &X(0,t) = HX(1,t) + Bu(t) \end{aligned}\right .$};
\draw (0,-5) node [rectangle, rounded corners, draw, minimum height =1cm, text centered] (obs) 
{\small $\left\{ \begin{aligned} \dot \eta (t)& = R\eta(t)  + S y(t) \\ u(t)&= K \eta (t)\end{aligned}\right .$};
\coordinate (upSamp) at ([xshift=1.5cm]sys.east);
\coordinate (bl) at ([xshift=-4cm]obs.west);
\coordinate (tr) at ([xshift=1.5cm]upSamp);
\coordinate (ftr) at ([xshift=2.5cm]tr);
\coordinate (br) at ([yshift=-5cm]tr);
\coordinate (tl) at ([yshift=5cm]bl);
\coordinate (upr) at ([yshift=-1.25cm]tr);
\coordinate (downr) at ([yshift=-3.5cm]tr);
\coordinate (upl) at ([yshift=-1.25cm]tl);
\coordinate (downl) at ([yshift=-3.5cm]tl);
\draw (tr) node [circle, draw, minimum height =0.65cm] (sum) {};

\draw [thick, ->] (sys.east) --node[anchor=south] {$X(1,t)$} (sum.west);
\draw [thick,->] (ftr) node[anchor=south east] {$d(t)$}-- (sum.east);
\draw [thick,-] (sum.south)--(br);
\draw [thick,->] (br) -- node[anchor=north] {$\qquad\qquad y(t) = X(1,t) + d(t)$} (obs.east);
\draw [thick, -] (obs.west) -- (bl);
\draw [thick,-] (bl)--(tl);
\draw [thick,->] (tl) -- (sys.west);
\end{tikzpicture}
\caption{Control architecture used for stabilization of hyperbolic system in the presence of disturbances.}
\label{fig:loopQuantHyp}
\end{figure}
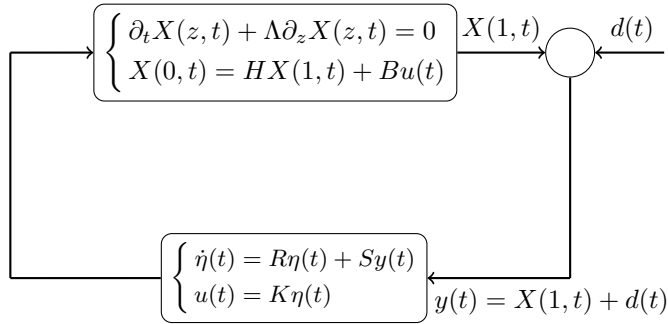

For system class \eqref{eq:sysHyp}, \eqref{eq:initCond}, \eqref{eq:defOut}, we are interested in designing control inputs $u$ that are absolutely continuous functions of time, so that their derivative is defined a.e. in Lebesgue sense.
For such inputs, we seek a solution $x\in {C}\big([0,T], {H}_1((0,1),\R^n)\big)$.

More precisely, we consider the problem of designing a dynamic controller with ODEs, which has the form
\begin{subequations}\label{eq:contGen}
\begin{align}
\dot \eta(t) &= R \eta (t) + S y(t), \label{eq:contGena}\\
u(t) & = K \eta(t), \label{eq:contGenb}
\end{align}
\end{subequations}
where the matrices $R \in \R^{n \times n}$, $S \in \R^{n \times n}$, and $K\in \R^{m \times n}$, need to be chosen appropriately.
Intuitively speaking, by using such a controller, the discontinuities of the output $y$ are integrated via equation \eqref{eq:contGena} which results in $u$ being absolutely continuous.
The addition of a dynamic controller introduces a coupling of ODEs and PDEs in the closed loop which makes the analysis more challenging. 

The result on existence and uniqueness of solutions for the closed-loop system \eqref{eq:sysHyp}, \eqref{eq:initCond}, \eqref{eq:defOut}, \eqref{eq:contGen}, is formally developed in Section~\ref{sec:cont}, resulting in certain regularity of the closed-loop solutions, which is important to obtain appropriate estimates.
Afterward, in Section~\ref{sec:lyap}, we use Lyapunov function based analysis to design the parameters of the controller \eqref{eq:contGen}, and derive conditions on the system and controller data which establish the DSS estimate \eqref{eq:maxEstGen}.

\subsubsection{Existence of solutions}\label{sec:cont}

The objective of this section is to present a result on existence and uniqueness of solutions for the closed-loop system \eqref{eq:sysHyp}, \eqref{eq:initCond}, \eqref{eq:defOut}, \eqref{eq:contGen}.
Solvability of hyperbolic PDEs is a well-studied topic.
For the intermediate results, we refer the reader to \cite[Appendix~A]{BaC16} and \cite[Chapter~3]{CuZ95}.
In \cite{BaC16}, the authors first present results with ${H}_1$-regularity for the autonomous system with $u = 0$.
The results for ODE coupled with hyperbolic PDE with $d=0$ with ${L}_2$ regularity are also proven.
However, in these works, with $d \in {L}_\infty([0,\tau],\R^n)$, which introduces certain discontinuities, the solutions with ${H}_1$-regularity are not discussed.
On the other hand, the well-posedness results are presented for systems with dynamics described by infinitesimal generators of
strongly continuous semigroups.

In this section, we present a result on well-posedness of the ODE-PDE coupled system, 
due to 
\cite{TPT17}, based on the results described in \cite[Appendix~A]{BaC16} and \cite[Chapter~3]{CuZ95}.
To do so, we start by constructing the operator ${A}$ as follows:
\begin{subequations}
\begin{gather}
\begin{split}
D({A}) := \Bigg \{ (\varphi,\eta) \in {H}_1\big((0,1),\R^n\big) \times \R^n;
 \  \begin{pmatrix} \varphi (0) \\ \eta \end{pmatrix} = \begin{bmatrix} H & BK \\ 0 & I\end{bmatrix} \begin{pmatrix}\varphi(1) \\ \eta \end{pmatrix} \Bigg\},
\end{split} \\
{A} \begin{pmatrix} \varphi \\ \eta \end{pmatrix} := \begin{pmatrix}- \Lambda \varphi_z \\ R \eta \end{pmatrix}.
\end{gather}
\end{subequations}

Next, we introduce the {\em perturbation operator} $B$ as:
\[
{B} : {H}_1\big((0,1),\R^n\big) \times \R^n\to {H}_1\big((0,1),\R^n\big) \times \R^n,
\qquad
{B} \begin{pmatrix} \varphi \\ \eta\end{pmatrix} =  \begin{pmatrix} 0 \\ S \varphi (1) \end{pmatrix}.
\]

Using these operators ${A}$ and ${B}$, and letting $\tilde{x} = (x, \eta)^\top$, one can write the closed-loop system \eqref{eq:sysHyp}, \eqref{eq:initCond}, \eqref{eq:defOut} and \eqref{eq:contGen} as follows:
\begin{subequations}\label{eq:deInfDim}
\begin{gather}
\dot{\tilde{x}} = {A} \tilde{x} + {B} \tilde{x} + \widetilde d,\\
\tilde{x}(0) = \tilde{x}^0 \in D({A}),
\end{gather}
\end{subequations}
where $\widetilde d = \begin{pmatrix} 0 \\ S \, d \end{pmatrix}$.
We now prove a result on the well-posedness of system \eqref{eq:deInfDim}. Because $d$ is possibly discontinuous, the classical solutions (where $\dot x$ is continuous) do not exist, and one must work with the notion of a mild solution \cite[Definition~3.1.6]{CuZ95}.

%

The well-posedness of \eqref{eq:deInfDim} is proven in \cite{TPT17} and uses some results of \cite[Chapter~3]{CuZ95}. To invoke these results, it is successively proven that
\begin{enumerate}
\item the operator ${A}$ generates a strongly continuous semigroup;
\item The operator ${B}$ is a bounded linear operator;
\item The operator ${A} + {B}$ generates a strongly continuous semigroup.
\end{enumerate}

This is the sketch of proof of the following existence and uniqueness result (see \cite{TPT17} for a complete proof):
\begin{theorem}\label{thm:sol}
For a given $\tau > 0$, and $d \in {L}_\infty([0,\tau],\R^n)$, there is a unique {\em mild} solution to system \eqref{eq:deInfDim} on $[0,\tau]$. Equivalently, for each $(x^0, \eta^0) \in {H}_1\big((0,1),\R^n\big) \times \R^n$, the closed-loop system \eqref{eq:sysHyp}, \eqref{eq:initCond}, \eqref{eq:defOut} and \eqref{eq:contGen} has a unique {\em mild} solution in the space 
${C}\big([0,\tau],{H}_1\big((0,1),\R^n\big) \times \R^n\big)$.
\end{theorem}

\begin{remark}
In \cite{TPT17}, it is dealt in the previous result with weak solutions and not mild solutions. For relations between both notions, refer to \cite[Section 3.1]{CuZ95}.
\end{remark}

\begin{remark}
The so-called compatibility conditions on the initial condition \linebreak
$(x^0,\eta^0)$, required for ${H}_1$-regularity, are imposed by requiring that $(x^0,\eta^0)$ belongs to $D({A})$. Such a condition is essential and hence the choice of $\eta^0$ depends upon $x^0$.
It is noted that in \cite{CBA08}, the authors propose {two} compatibility conditions for the initial state because they seek solution $x \in {H}_2 \big((0,1),\R^n\big)$. We only need solutions where $x$ is ${H}_1$-regular, so only one such condition appears in our analysis.
\end{remark}



\subsubsection{Closed loop and stability analysis}\label{sec:lyap}

As a solution to the problem formulated in Section~\ref{sec:prelim}, we now provide more structure for the controller dynamics and study the stability of the closed-loop system.
The conditions on the system parameters that guarantee stability are then provided by constructing a Lyapunov-function.

\paragraph{Control architecture and closed loop}

The controller that we choose for our purposes is described by the following equations:
\begin{subequations}\label{eq:contLin}
\begin{align}
\dot \eta (t) & = -\alpha (\eta(t)- y(t)) \notag \\
& =  -\alpha \, \eta(t) + \alpha x(t,1) + \alpha d(t), \label{eq:contLina}\\
\eta(0)&=\eta^0,   \label{eq:init:eta}\\
u(t) & = K \eta(t), \label{eq:contLinb}
\end{align}
\end{subequations}
where $\eta^0 \in \R^n$ is the initial condition for the controller dynamics.
This corresponds to choosing $R = -\alpha\, I$, and $S = -R$ in \eqref{eq:contGena}.
Note that the controller \eqref{eq:contLin} is an ISS system w.r.t.\ inputs $x(t,1)$ and $d$ (as a linear finite-dimensional system with $\alpha>0$).

The conditions on the constant $\alpha >0$, and the matrix $K \in \R^{m\times n}$ will be stated in the formulation of Theorem~\ref{thm:mainISS}.

The closed-loop system is given by equations
\begin{subequations}\label{eq:closed:loop:X}
\begin{gather}
x_t (t,z) +\Lambda x_z (t,z) = 0, \quad t >0,\ z\in(0,1), \label{eq:closed:loop:X:dynamics}\\
x(0,z)=x^0(z), \quad z\in [0,1],\\
x(t,0) = H x(t,1) + BK \eta (t).\label{eq:closed:loop:X:bd}
\end{gather}
\end{subequations}

%

\paragraph{Stability result}

Let us now state a disturbance-to-state stability (DSS) result for the closed-loop dynamics \eqref{eq:contLin}, \eqref{eq:closed:loop:X}. 

Let $\mathcal{D}_+^n$ denote the set of diagonal positive definite matrices in $\mathbb{R}^n$. 
For scalars $\mu>0$ and $0 < \nu < 1$, let $\rho := e^{-\mu} - \nu^2$; let $F:= BK$, and $Q:=F^\top D^2 F$ for $D \in \mathcal{D}_+^n$; and finally, let $G:= H^\top D^2 F$.
We denote by $\Omega$ the symmetric matrix
\[
\Omega:=
\begin{bmatrix}
\rho \beta_1 D^2 & -\beta_1(G + Q) & 0\\
* & 2 \alpha \beta_3 - (\beta_1+\alpha^2\beta_2) Q  & \beta_3 I+ \alpha \beta_2 G \\
* & * & (\rho D^2 + Q + G+G^\top)\beta_2
\end{bmatrix}
\]
in which $\alpha, \beta_1, \beta_2, \beta_3$ are some positive constants, and $*$ denotes the transposed matrix block. 
The following result is proven in \cite{TPT17}.

\begin{theorem}\label{thm:mainISS}
Assume that there exist scalars $\mu,\nu > 0$, a matrix $D \in \mathcal{D}_+^{n}$, the gain matrix $K$, and the positive constants $\alpha,\beta_1,\beta_2,\beta_3$ in the definition of $\Omega$ such that
\begin{subequations}\label{eq:gainCond}
\begin{gather}
\| D(H+BK) D^{-1} \| \le \nu < 1, \label{eq:gainConda}\\
\Omega > \zeta I, \label{eq:gainCondb}
\end{gather}
\end{subequations}
for some scalar $\zeta > 0$.
Then, {\color{blue}there is  $\gamma\in\K$ so that} the closed-loop system \eqref{eq:contLin}, \eqref{eq:closed:loop:X} satisfies the DSS estimate \eqref{eq:maxEstGen} with
\begin{equation}\label{eq:defMX0}
M_{x^0} := \|x^0 \|^2_{{H}_{1}((0,1),\R^n)} + \vert \eta^0 - x(0,1)\vert^2.
\end{equation}
\end{theorem}
With our design of a controller we obtain besides DSS also further important properties of a closed-loop system.

{\begin{remark}
\label{rem:CICS-property} 
Due to the cocycle (also called semigroup) property, the conditions we impose on the system to obtain estimate \eqref{eq:maxEstGen}, also ensure that if $d(t) \to 0$, then $\max_{z\in [0,1]} \vert x(t,z) \vert$ also converges to zero as $t\to\infty$.
\end{remark}

\begin{remark}
It must be noted that, if the initial condition of the closed-loop system $x^0, \eta^0$ is chosen such that $\|x^0 \|_{{H}_{1}((0,1),\R^n)} = 0$ and $\eta^0 = 0$, then $M_{x^0} = 0$.
Indeed, since $x^0 \in {H}_1((0,1),\R^n)$ implies that $x^0$ is continuous, and $\|x^0\|^2_{{L}_2((0,1),\R^n)} = 0$ implies that $x^0 = 0$ almost everywhere on $[0,1]$, we must have $x(0,1) = 0$.
\end{remark}

\begin{remark}[DSS implies ISpS-like property]
For $M_{x^0}$ given in \eqref{eq:defMX0}, we have
\begin{align*}
M_{x^0}  \le  \|x^0\|^2_{{H}_{1}((0,1),\R^n)} + \vert x^0(1) \vert^2 + \vert \eta^0 \vert^2 
 \le \max_{z\in[0,1]} \vert x^0(z) \vert^2 + \|x^0\|_{{H}_{1}((0,1),\R^n)}^2 + \vert \eta^0 \vert^2
\end{align*}
Substituting this bound on $M_{x^0}$ in \eqref{eq:maxEstGen}, our DSS estimate leads to
\begin{multline}
\label{eq:quasi-ISpS}
\max_{z\in[0,1]} \vert x(z) \vert \le c \,e^{-at} \max_{z\in[0,1]} \vert x^0(z) \vert^2 + \gamma (\|d_{[0,t]} \|_\infty) \\+ c \,e^{-at} \left( \|x^0\|^2_{{H}_1((0,1),\R^n)} + \vert \eta^0\vert^2 \right).
\end{multline}
The estimate \eqref{eq:quasi-ISpS} is not exactly an input-to-state practical stability (ISpS) property as defined in Definition~\ref{Def:ISpS_wrt_set} (as it is also defined in \cite{JTP94}), because the offset in \eqref{eq:quasi-ISpS} depends on the initial condition size (it is constant in Definition~\ref{Def:ISpS_wrt_set}). Furthermore, the estimate  \eqref{eq:quasi-ISpS} gives only an estimate of a norm of a partial state (more precisely $x$, and not $\eta$).
\end{remark}

\begin{openprob}
In the statement of Theorem~\ref{thm:mainISS}, condition \eqref{eq:gainConda} requires 
$\inf_{D\in\mathcal{D}_+^n} \| D (H+BK) D^{-1} \| < 1$ which also appears in the more general context of nonlinear systems \cite{CBA08} when analyzing stability with respect to ${H}_2$-norm. However, the condition \eqref{eq:gainCondb} is introduced in \cite{TPT17} to compensate for the lack of proportional gain in the feedback law. It definitely restricts the class of systems that can be treated with this approach and relaxing this condition or obtaining different criteria is a topic of further investigation.
\end{openprob}

\begin{remark}
The authors in \cite{TPT17} do not provide a precise characterization of the parameters of system~\eqref{eq:sysHyp} for which \eqref{eq:gainCond} admits a solution.
For instance, assume that \eqref{eq:gainConda} holds with $K = 0$.
In that case, the matrix $\Omega$ simplifies greatly as $Q = G = 0$.
Using the Schur complement, one can immediately find the constants $\alpha, \beta_1, \beta_2,\beta_3$ that result in $\Omega$ being positive definite, and hence satisfying \eqref{eq:gainCondb}.
By applying the continuity argument for solutions of matrix inequalities with respect to parameter variations, the solution to \eqref{eq:gainCondb} will also hold for $K \neq 0$, but sufficiently small.
\end{remark}

The proof of Theorem~\ref{thm:mainISS} in \cite{TPT17} is made by construction of an exponential Lyapunov function
$V:{H}_1\big((0,1),\R^n\big) \times \R^n \to \R_+$ so that, along the trajectories of (\ref{eq:contLin})-(\ref{eq:closed:loop:X}), it holds that
\begin{equation}
\dot V (X(t),\eta(t)) \le -\sigma V(X(t),\eta(t)) +\chi \vert d(t)\vert^2,
\end{equation}
for some constant $\sigma,\chi > 0$.
This function has been constructed in \cite{TPT17} in the form
\begin{equation}
\label{eq:defLyap}
\begin{array}{rcl}
V(x,\eta) &:=& 
\underbrace{\int_0^1 x(z)^\top P_1 x(z) e^{-\mu z} dz}_{V_1(x)} 
+
\underbrace{\int_0^1 \frac{\partial x}{\partial z}(z)^\top P_2 \frac{\partial x}{\partial z}(z) e^{-\mu z}dz}_{V_2(x)} 
\\
&&+
\underbrace{\big(\eta {-} x(1)\big)^\top P_3 \big(\eta {-} x(1)\big)}_{V_3(x,\eta)}, 
\end{array}\end{equation}
where $P_1$ and $P_2$ are certain diagonal positive definite matrices and $P_3$ is some symmetric positive definite matrix. 

Let us denote the spectrum of a matrix $M \in\R^{n\times n}$ by $\sigma(M)$ and let 
$\lambda_{\min}(M):=\min_{\lambda\in\sigma(M)}|\lambda|$ and $\lambda_{\max}(M):=\max_{\lambda\in\sigma(M)}|\lambda|$.

Defining $\underline c_P:= \min_{i = 1,2,3}\{\lambda_{\min}(P_i)\}$, 
$\overline c_P:=\max_{i = 1,2,3}\{\lambda_{\max}(P_i)\}$ we see that, for all $x\in {H}_1((0,1),\R^n)$, and $\eta \in \R^n$,
\begin{multline}\label{eq:lyapPropBnd}
\underline c_P (\|x\|_{{H}_1((0,1),\R^n)}^2 {+} \vert \eta-x(1)\vert^2) \le V (x,\eta) \le 
\overline c_P (\|x\|_{{H}_1((0,1),\R^n)}^2 {+} \vert \eta-x(1)\vert^2).
\end{multline}

The rest of the proof of Theorem~\ref{thm:mainISS} is based on precise computations of the time-derivative of $V$ along the solutions to (\ref{eq:contLin})-(\ref{eq:closed:loop:X}). See \cite{TPT18,TPT17} for a complete proof and an application to quantized control. In this latter reference, 
practical stability of the system is established, and ultimate bounds on the state trajectory are derived in terms of the quantization error.

To conclude this section, let us note that a Lyapunov function with a strict decrease property has been used. It is also possible to consider a non-decreasing property combined with a LaSalle invariance as done recently in \cite{chen2017stabilization} to stabilize the string equation in the presence of disturbances.

\subsection{ISS Lyapunov methods for time-varying hyperbolic systems}
\label{sec:ISS_analysis_linear_nonlinear_PDEs_Lyapunov_methods:hyperbolic:time-varying}

\subsubsection{Basic definitions and notions}
\label{fm}

%

Throughout this section, we consider linear partial differential equations of the form
\begin{equation}
\label{hyper:eq:general}
\frac{\partial x}{\partial t}(t,z) + \Lambda (t,z)\frac{\partial x}{\partial z}(t,z)
= F(t,z) x(t,z) + d(t,z), \quad z \in (0,1),\ \  t >0,
\end{equation} 
and $\Lambda(t,z)=\diag\big(\lambda_1(t,z),\ldots, \lambda_n(t,z)\big)$ is a diagonal 
matrix in $\mathbb{R}^{n\times n}$ whose $m$ first diagonal terms are nonnegative and 
the $n - m$ last diagonal terms are nonpositive (we will say that the hyperbolicity assumption holds, when additionally the $\lambda_i$'s are never vanishing). We assume that the 
function $d$ is a disturbance of class $C^1$, $F$ is a periodic 
function with respect to $t$ of period $\tau$ and of class $C^1$, $\Lambda$ is 
a function of class $C^1$, periodic with respect to $t$ of period $\tau$.

The boundary conditions are written as
\begin{equation}
\label{boundary:hyper:eq:general}
\left(\begin{array}{c}
x_+(t,0)
\\
x_-(t,1)\end{array}\right)
 = K\left(\begin{array}{c}
x_+(t,1)
\\
x_-(t,0)\end{array}\right),
\end{equation}
where $
x = \left(\begin{array}{c} 
x_+
\\
x_-
\end{array}
\right)
$, 
$x_+ \in \mathbb{R}^m$, $x_- \in \mathbb{R}^{n - m}$, and $K \in \mathbb{R}^{n \times n}$ is a constant 
matrix.

The initial condition is
\begin{equation}
\label{init:eq}
x(0,z) = x^0(z),\; \quad z \in (0,1) ,
\end{equation}
where $x^0$ is a function $[0,1]: \rightarrow \mathbb{R}^n$ of class $C^1$ satisfying the usual zero-order and one-order compatibility conditions (see \cite{PrM12} for a precise expression).

As proved in \cite{kmit:classical:solvability:2008}, if the function $\Lambda$ is 
of class $C^1$ and satisfies the hyperbolicity 
assumption, if the function $d$ is of class $C^1$, and if the initial condition 
is of class $C^1$ and satisfies the 
compatibility conditions,
there 
exists a unique classical solution of 
the system (\ref{hyper:eq:general}), with the boundary 
conditions (\ref{boundary:hyper:eq:general}) and the initial 
condition (\ref{init:eq}), defined for all $t\geq 0$.

\subsubsection{ISS Lyapunov functions for hyperbolic systems}
\label{linear:hyper}

Before stating the main theoretical result of the section, some comments are needed.
Since, in the case where the system (\ref{hyper:eq:general}) is such that $m < n$, we 
can replace $x(t,z)$ by 
$\left(
\begin{array}{c}
x_+(t,z)
\\
x_-(1 - z,t)
\end{array}
\right)$
we may assume without loss of generality 
that $\Lambda$ is diagonal and positive semidefinite. Then the boundary conditions 
(\ref{boundary:hyper:eq:general}) become
\begin{equation}
\label{boundary:hyper:eq}
x(t,0) = K x(t,1) .
\end{equation}

Following what has been assumed for parabolic 
equations in Section \ref{sec:ISS_analysis_linear_nonlinear_PDEs_Lyapunov_methods:parabolic}, it might seem natural 
to consider the case where $\dot x= F(t,z)x$ possesses some stability properties for any $z$. 
On the other hand, there is no reason to believe that this property is always needed. 

For a matrix $M\in \R^{n\times n}$ denote by $\sym(M)$ the symmetric part of $M$, that is $\sym(M):=\frac{1}{2}(M+M^\top)$. 
We continue to denote by $\|M\|$ the norm of the matrix, induced by Euclidean norm in $\R^n$.

The above remarks lead us to introduce the following assumption: 
\begin{Ass}
\label{C3.1}
For all $t \geq 0$ and for all $z \in [0,1]$, all the entries of the diagonal matrix $\Lambda(t,z)$ are
nonnegative. 

There exist a diagonal positive definite matrix $Q$, a real 
number $\alpha \in (0,1)$, a continuous real-valued function $r$, periodic 
of period $\tau > 0$ such that 
\begin{equation}
\label{fin1}
B := \displaystyle\int_{0}^{\tau} \left[\frac{\max\{r(m),0\}}{\|Q\|} + \frac{\min\{r(m),0\}}{\lambda_{\min} (Q)}\right] dm >0, 
\end{equation}
where $\lambda_{\min} (Q)$ is the smallest eigenvalue of $Q$.
Moreover, for all $t \geq 0$ and for all $z \in [0,1]$, the following inequalities are satisfied:
\begin{eqnarray}
\label{1hype}
&
\sym\left(\alpha Q \Lambda(t,1) - K^\top Q \Lambda(t,0) K\right) \geq 0 ,
\\
\label{hy2pe}
&
\sym\left(Q \Lambda(t,z)\right) \geq r(t) I  ,
\\
\label{hype}
&
\sym\left(Q \frac{\partial \Lambda}{\partial z}(t,z) + 2 Q F(t,z)\right) \leq 0.
\end{eqnarray}
\end{Ass}

Let us introduce the function
\begin{equation}
\label{perf}
q(t) = \mu \left[\frac{\max\{r(t), 0\}}{\|Q\|} + \frac{\min\{r(t), 0\}}{\lambda_Q}\right] - \frac{\mu B}{2 \tau} ,
\end{equation}
where $Q$ and $r$ are the matrix and the function in Assumption \ref{C3.1}.

We are ready to state the following results (see  \cite[Theorem 1]{PrM12} for a proof of this result):
\begin{theorem}
\label{theo:linear}
Assume that the system (\ref{hyper:eq:general}) with the boundary 
conditions (\ref{boundary:hyper:eq}) satisfies Assumption \ref{C3.1}. 
Let $\mu$ be any real number such that 
\begin{equation}
\label{coml}
0 < \mu \leq - \ln(\alpha) .
\end{equation}
Then the function 
$V : [0,\infty)\times L_2(0,1)  \rightarrow \mathbb{R}$ defined, for all $x \in L_2(0,1)$ 
and $t \geq 0$, by
\begin{equation}
\label{def:Lyapunov:hyper}
V(t,x) = \exp\left(\frac{1}{\tau} \int_{t - \tau}^{t}\int_\ell^{t} q(m) dm d\ell\right)
\displaystyle\int_0^1 x(z)^\top Q x(z) e^{-\mu z} dz,
\end{equation}
where $q$ is the function defined in (\ref{perf}), is an (exponential) ISS Lyapunov function for the 
system (\ref{hyper:eq:general}) with the boundary conditions (\ref{boundary:hyper:eq}).
\end{theorem}

\begin{remark}
{\em 1.} Assumption \ref{C3.1} does not imply that for all fixed $z \in [0,1]$, 
the ordinary differential equation $\dot x = F(t,z) x$ is stable. 
See \cite{MaP11} for an example where this system is unstable.

{\em 2.} Let us emphasize that in Assumption \ref{C3.1} we need that the matrix $Q$ is diagonal, but we do not need that
$r$ is a nonnegative function.  

{\em 3.} Assumption \ref{C3.1} holds when, in the 
system (\ref{hyper:eq:general}), $\Lambda(t,z)$ is constant, $F(t,z)$ 
is constant, $d(t,z) = 0$  for all $z \in [0,1]$ and 
$t \geq 0$ and the boundary condition (\ref{boundary:hyper:eq:general}) 
satisfies
$$
\sym\left( Q \Lambda - K^\top Q \Lambda K\right) \geq 0 \; ,
\quad 
\sym\left( Q f\right) \leq 0
$$
for a suitable symmetric positive definite matrix $Q$. Therefore Theorem \ref{theo:linear} 
generalizes the sufficient conditions of 
\cite{DiagneBastinCoron:automatica:11} for the exponential stability of linear hyperbolic systems 
of balance laws (when $F$ is diagonally marginally stable), to the 
time-varying case and to the semilinear perturbed case (without assuming that $F$ is diagonally marginally stable).

{\em 4.} The Lyapunov function $V$ defined in (\ref{def:Lyapunov:hyper}) is a 
time-varying function, periodic of period $\tau$. 
In the case where the system is time-invariant,  one can choose a constant 
function $q$, and then it is obtained a time-invariant function (\ref{def:Lyapunov:hyper}) 
which is a quite usual Lyapunov function candidate in the context of the stability analysis 
of PDEs (see, e.g., \cite{Coron:Euler:98,CBA08,XuSallet:COCV:2002}). 

{\em 5.} Note that the notion of ISS Lyapunov functions for time-varying periodic systems has not been defined so far. This notion is employed in Theorem \ref{theo:linear} and is similar to what is done for stationary systems (see Definition \ref{def:noncoercive_ISS_LF}). The precise definition is thus skipped in the context of Theorem \ref{theo:linear}.
\end{remark}

\begin{remark}
\label{rem:Saint-Venant equation} 
The results in this section have been applied in \cite{PrM12} for the design of a stabilizing boundary feedback
for the shallow water equation.
\end{remark}

\section{Interconnected systems}
\label{sec:Interconnected_systems}

The analysis of interconnected systems is one of the cornerstones of the mathematical control theory.
Complexity of large-scale nonlinear systems makes a direct stability analysis of such systems ultimately challenging. 
Small-gain theorems help to overcome this obstruction and to study stability of a complex network based on the knowledge of the stability properties of its components.
Classical results of this nature within linear input-output theory are summarized in \cite{DeV09}.

The first nonlinear ISS small-gain theorems have been shown for couplings of two ODE systems in fundamental works \cite{JTP94, JMW96}, and these results have been generalized to the interconnections of an arbitrary finite number of nonlinear ODE systems in \cite{DRW07, DRW10}, {\color{blue} see also \cite{KaJ11} for small-gain theorems in terms of vector Lyapunov functions.}
These theorems guarantee input-to-state stability of an interconnected system, provided all subsystems are ISS and the interconnection structure, described by gains, satisfies the small-gain condition. Together with the ISS Lyapunov functions method, small-gain theorems are one of the main tools developed within the ISS theory for analysis and control of networks.

Recently these results have been fully extended to finite couplings of abstract infinite-dimensional systems 
and some results on couplings of an infinite number of systems are available.
In this section, we give an overview of these results.

\subsection{Interconnections of control systems}
\label{sec:Interconnections_definition}

Let us define the interconnections of abstract control systems. In our exposition, we follow \cite{Mir19b}, which is based in turn on the methodology introduced in \cite[Definition 3.3]{KaJ07}.

Let $(X_i,\|\cdot\|_{X_i})$, $i=1,\ldots,n$ be normed vector spaces endowed with the corresponding norms.
Define for each  $i=1,\ldots,n$ the normed vector space
\begin{eqnarray}
\phantom{aaaaa} X_{\neq i}:=X_1 \times \ldots \times X_{i-1} \times X_{i+1} \times \ldots \times X_n, \quad 
\|x\|_{X_{\neq i}} : = \Big(\sum_{j=1,\ j\neq i}^n\|x_j\|^2_{X_j}\Big)^{\frac{1}{2}}.
\label{eq:X_neq_i}
\end{eqnarray}
Let control systems $\Sigma_i:=(X_i,PC_b( \R_+,X_{\neq i})\tm \Uc,\bar{\phi}_i)$ be given and assume that each $\Sigma_i$ possesses the BIC property.
We call $X_{\neq i}$ the \emph{space of internal input values} and $PC_b( \R_+,X_{\neq i})$ the \emph{space of internal inputs} for a system $\Sigma_i$. This choice of the input space is natural as the trajectories of subsystems have to be continuous, and the space of inputs has to satisfy the concatenation axiom.

The norm on $PC_b( \R_+,X_{\neq i})\tm \Uc$ we define as
\begin{align}
\|(v,u)\|_{PC_b( \R_+,X_{\neq i})\tm \Uc} := \Big(\sum_{j \neq i}\|v_j\|^2_{PC_b( \R_+,X_{j})} + \|u\|^2_{\Uc}\Big)^{\frac{1}{2}}.
\label{eq:Norm_Full_Input}
\end{align}
Define also the normed vector space which will be the state space for the coupled system
\begin{eqnarray}
X:=X_{1}\times\ldots\times X_{n}, \qquad \|x\|_X := \Big(\sum_{i=1}^n\|x_i\|^2_{X_i}\Big)^{\frac{1}{2}},
\label{eq:State_Space_Full_System}
\end{eqnarray}
and assume that there is a map $\phi=(\phi_1,\ldots,\phi_n):D_\phi \to X$, defined over a certain domain $D_{\phi} \subseteq \R_+ \tm X \times \Uc$
so that for each $x=(x_1,x_2,\ldots,x_n) \in X$, each $u\in\Uc$ and all $t\in\R_+$ so that $(t,x,u)\in D_{\phi}$
and for every $i=1,\ldots,n$, it holds that 
\begin{eqnarray}
\phi_i(t,x,u) = \bar{\phi}_i\big(t,x_i,(v_i,u)\big),
\label{eq:Sigma_i_models_i_th_mode_of_Sigma}
\end{eqnarray}
where
\[
v_i(t) = (\phi_1(t,x,u),\ldots,\phi_{i-1}(t,x,u),\phi_{i+1}(t,x,u),\ldots,\phi_{n}(t,x,u)).
\]

\begin{definition}
\label{def:Interconnection} 
Assume that $\Sigma:=(X,\Uc,\phi)$ is a control system with the state space $X$, input space $\Uc$ and satisfying the BIC property.
Then $\Sigma$ is called \emph{a (feedback) interconnection} of systems $\Sigma_1,\ldots,\Sigma_n$.
\end{definition}

In other words, condition \eqref{eq:Sigma_i_models_i_th_mode_of_Sigma} means that if the modes $\phi_j(\cdot,x,u)$, $j\neq i$ of the system $\Sigma$ will be sent to $\Sigma_i$ as the internal inputs (together with an external input $u$), and the initial state will be chosen as $x_i$ (the $i$-th mode of $x$), then the resulting trajectory of the system $\Sigma_i$, which is $\bar{\phi}_i(\cdot,x_i,v,u)$ will coincide with the trajectory of the $i$-th mode of the system $\Sigma$ on the interval of existence of $\phi_i$.


This definition of feedback interconnections does not depend on a particular type of control systems which are coupled, and applies to large-scale systems, consisting of heterogeneous components as PDEs, time-delay systems, ODE systems, etc.
The definition also applies to different kinds of interconnections, e.g.,\ both for in-domain and boundary interconnections of PDE systems.

\subsection{Small-gain theorems in a trajectory formulation}
\label{sec:ISS-SGT-Trajectory-Formulation}

Consider $n$ forward complete ISS systems $\Sigma_i:=(X_i,PC_b( \R_+,X_{\neq i}) \tm \Uc,\bar{\phi}_i)$, $i=1,\ldots,n$, where all $X_i$, $i=1,\ldots,n$ and $\Uc$ are normed  linear spaces. 

ISS property is introduced in Section~\ref{sec:Systems-and-ISS-def} in terms of the norm of the whole input, and this is not suitable for the consideration of coupled systems, as we are interested not only in the collective influence of all inputs over a subsystem but in the influence of particular subsystems over a given subsystem.
It is much more suitable to reformulate the ISS property in the following way:
\begin{lemma}{\cite{Mir19b}}
\label{lem:ISS_reformulation_n_systems}
A forward complete system \emph{$\Sigma_i$ is ISS (in summation formulation)} if there exist $\gamma_{ij}, \gamma_i \in \K\cup\{0\},\ j=1,\ldots,n$ and $\beta_i \in \KL$, such that for all initial values $x_i \in X_i$, all internal inputs 
$w_{\neq i} := (w_1,\ldots,w_{i-1}, w_{i+1},\ldots,w_n) \in PC_b( \R_+,X_{\neq i})$, all external inputs $u \in\Uc$
and all $t \in\R_+$ the following estimate holds:
\begin{eqnarray}
\label{eq:ISS_n_sys_sum}
\|\bar{\phi}_i\big(t,x_i,(w_{\neq i}, u)\big)\|_{X_i}  \leq
 \beta_i\left(\left\|x_i\right\|_{X_i},t\right) + \sum_{j\neq i}\gamma_{ij}\left(\left\|w_j\right\|_{[0,t]}\right) + \gamma_i\left(\left\|u\right\|_{\Uc}\right).
\end{eqnarray}
\end{lemma}

We collect all the internal gains in the matrix $\Gamma^{ISS}:=(\gamma_{ij})_{i,j=1,\ldots,n}$, which we call the \emph{gain matrix}. For a given gain matrix $\Gamma^{ISS}$ define the operator $\Gamma^{ISS}_{\oplus}:\R^n_+\to\R^n_+$ by
\begin{equation}
\Gamma^{ISS}_{\oplus}(s) := \Big(
  \sum_{j=1}^n \gamma_{1j}(s_j),\ldots,\sum_{j=1}^n \gamma_{nj}(s_j)
\Big)^\top,   \quad s=(s_1,\ldots,s_n)^\top\in\R^n_+.
\label{eq:ps_gamma}
\end{equation}

Furthermore, for $\alpha_i \in \Kinf, i=1,\ldots,n$ define $D:\R_+^n\to\R_+^n$ by
\begin{equation}
  \label{eq:definition von D}
  D (s_1,\ldots,s_n)^\top := \big((\id+\alpha_1)(s_1),\ldots,(\id+\alpha_n)(s_n)\big)^\top.
\end{equation}

A fundamental role will be played by the following operator conditions:
\begin{definition}
\label{def:SGC} 
We say that a nonlinear operator $A:\R^n_+ \to \R^n_+$ satisfies
\begin{itemize}
    \item the \emph{small-gain condition}, if 
  \begin{equation}
    A(s)\not\geq s,\qquad\forall s\in\R^n_+\setminus\{0\}.
  \label{eq:SGC}
  \end{equation}
    \item the \emph{strong small-gain condition}, if there is a map $D$ as in
  \eqref{eq:definition von D}, such that
  \begin{equation}
    (A\circ D)(s)\not\geq s,\qquad\forall s\in\R^n_+\setminus\{0\}.    
  \label{eq:Strong_SGC}
  \end{equation}
\end{itemize}
\end{definition}

Now we can formulate the small-gain sufficient condition for ISS of networks of nonlinear ISS systems.
\begin{theorem}[ISS Small-gain theorem, \cite{MiW18b}]
\label{thm:ISS_SGT} 
Let $\Sigma_i:=(X_i,PC_b( \R_+,X_{\neq i}) \tm \Uc,\bar{\phi}_i)$, $i=1,\ldots,n$ be control systems, where all $X_i$, $i=1,\ldots,n$ and $\Uc$ are normed  linear spaces.
Assume that $\Sigma_i$, $i=1,\ldots,n$ are forward complete systems, satisfying the ISS estimates as in Lemma~\ref{lem:ISS_reformulation_n_systems}, and that the interconnection $\Sigma=(X,\Uc,\phi)$ is well-defined and possesses the BIC property.

If $\Gamma^{ISS}_\oplus$ satisfies the strong small gain condition \eqref{eq:Strong_SGC}, then $\Sigma$ is ISS.    
\end{theorem}

\begin{proof}
The idea of the proof is to verify the UGS and the UAG properties for the interconnection, and then use the ISS Superposition Theorem~\ref{thm:UAG_equals_ULIM_plus_LS} to show ISS of the interconnection. For ODE systems this proof strategy has been proposed in \cite{DRW07}. However, in the ODE case it was sufficient to show UGS and a non-uniform asymptotic gain property and then use the characterizations of ISS for ODEs shown in \cite{SoW96}. For infinite-dimensional systems, one has to prove the UAG property, which is a more difficult task.
\end{proof}

\subsection{Small-gain theorems in a Lyapunov formulation}
\label{sec:ISS-SGT-Lyapunov-Formulation}

As in most cases ISS of nonlinear systems is verified by the construction of an appropriate ISS Lyapunov function,
it is a natural desire, to use in the formulation of the small-gain theorems the information about the ISS Lyapunov functions for subsystems instead of using ISS estimates for subsystems.
In this section, we show that this is possible and moreover, this results in a method for the construction of ISS Lyapunov functions for the overall system if ISS Lyapunov functions for the subsystems are known.

Consider $n \in\N$ forward complete systems $\Sigma_i:=(X_i,PC_b(\R_+,X_{\neq i}) \tm \Uc,\bar{\phi}_i)$, $i=1,\ldots,n$, where all $X_i$, $i=1,\ldots,n$ and $\Uc$ are normed  linear spaces. Let also $\Sigma = (X,\Uc,\phi)$ be a well-defined interconnection of $\Sigma_i$, $i=1,\ldots,n$, as explained in Section~\ref{sec:Interconnections_definition}.

Furthermore, assume that all $\Sigma_i$, $i=1,\ldots,n$ are ISS and we know corresponding coercive ISS Lyapunov functions
 $V_i:X_i \to \R_+$, i.e. continuous functions for which there exist functions $\psi_{i1},\psi_{i2}\in\Kinf$, $\chi \in \K$ and positive definite function $\alpha_{i}$, such that
\[
\psi_{i1}(\|x_{i}\|_{X_{i}})\leq V_{i}(x_{i})\leq\psi_{i2}(\|x_{i}\|_{X_{i}}),\quad\forall x_{i}\in X_{i}
\]
and for all $x_{i}\in X_{i}$, all $x_{\neq i} = (x_1,\ldots,x_{i-1},x_{i+1},\ldots,x_n) \in X_{\neq i}$,
all $v_i \in PC_b(\R_+,X_{\neq i})$ with $v_i(0)=x_{\neq i}$ and all $u\in\Uc$ the following implication holds
\begin{equation}
\label{GainImplikationNSyst}
V_i(x_{i})\geq\max\{ \max_{j=1}^{n}\chi_{ij}(V_{j}(x_{j})),\chi_{i}(\|u\|_{\Uc})\} \ \Rightarrow\ \dot{V}_{i,v_i,u}(x_i)\leq-\alpha_{i}(V_{i}(x_{i})),
\end{equation}
where
\begin{eqnarray}
\dot{V}_{i,v_i,u}(x_{i})=\mathop{\overline{\lim}}\limits _{t\rightarrow+0}\frac{1}{t}(V_i(\bar{\phi}_{i}(t,x_{i},(v_i,u)))-V_i(x_{i})).
\label{eq:Dini-derivative_definition}
\end{eqnarray}
In analysis of the interconnection $\Sigma$, we are primarily interested in internal inputs of a specific form.
Pick arbitrary $x \in X$ and $u\in\Uc$ and define for $i=1,\ldots,n$ the following quantities:
\begin{eqnarray}
\phi_i:=\phi_i(\cdot,x,u),\quad \phi_{\neq i} = (\phi_1,\ldots,\phi_{i-1}, \phi_{i+1},\ldots,\phi_n).
\label{eq:phi_neq_i}
\end{eqnarray}
As $\Sigma$ is a well-defined interconnection, the equality 
\eqref{eq:Sigma_i_models_i_th_mode_of_Sigma} holds with the input
$\phi_{\neq i}$, containing the modes of the coupled system.
Consequently, for $v_i:= \phi_{\neq i}$ the Lie derivative \eqref{eq:Dini-derivative_definition} takes form:
\begin{eqnarray}
\dot{V}_{i,u}(x_{i}):=\dot{V}_{i,\phi_{\neq i},u}(x_{i}) = \mathop{\overline{\lim}}\limits _{t\rightarrow+0}\frac{1}{t}\Big(V_i\big(\phi_{i}(t,x,u)\big)-V_i(x_{i})\Big).
\label{eq:Dini-derivative_interconnected_case}
\end{eqnarray}

In the following we exploit the implication form as in \eqref{GainImplikationNSyst} and assume, that for all $i=1,\ldots,n$ for Lyapunov function $V_i$ of the $i$-th system the gains $\chi_{ij}$, $j=1,\ldots,n$ and $\chi_i$ are given.

Gains $\chi_{ij}$ characterize the interconnection structure of subsystems. 
Let us introduce the gain operator $\Gamma_\otimes:\R_{+}^{n}\rightarrow\R_{+}^{n}$ defined by 
\begin{equation}
\label{operator_gamma}
\Gamma_\otimes(s):=\left(\max_{j=1}^{n}\chi_{1j}(s_{j}),\ldots,\max_{j=1}^{n}\chi_{nj}(s_{j})\right),\ s\in\R_{+}^{n}.
\end{equation}

We recall the notion of $\Omega$-path (see \cite{DRW10,Rue10}), useful for investigation of stability of interconnected systems and for construction of a Lyapunov function of the whole interconnection.
\begin{definition} 
A function $\sigma=(\sigma_{1},\dots,\sigma_{n})^{T}:\R_{+}^{n}\rightarrow\R_{+}^{n}$,
where $\sigma_{i}\in\K_{\infty}$, $i=1,\ldots,n$ is called an \textit{$\Omega$-path},
if it possesses the following properties:
\begin{enumerate}
    \item $\sigma_{i}^{-1}$ is locally Lipschitz continuous on $(0,\infty)$;
    \item for every compact set $P\subset(0,\infty)$ there are finite
constants $0<K_{1}<K_{2}$ such that for all points of differentiability
of $\sigma_{i}^{-1}$ we have 
\begin{align*}
0<K_{1}\leq(\sigma_{i}^{-1})'(r)\leq K_{2},\quad\forall r\in P ;
\end{align*}
\item It holds that
\begin{align}
\label{sigma cond 2}
\Gamma_\otimes(\sigma(r))<\sigma(r),\ \forall r>0.
\end{align}
\end{enumerate}
\end{definition}

\begin{definition}
\label{def:DGC} 
We say that $\Gamma_\otimes$ satisfies the \emph{small-gain condition} if for all $s\in\R_{+}^{n}\backslash\left\{ 0\right\}$ it holds that
\begin{align}
\label{smallgaincondition}
\Gamma_\otimes(s)\not\geq s \quad \Iff \quad \exists i:  (\Gamma_\otimes(s))_i < s_i.
\end{align}
\end{definition}

\begin{remark}
\label{rem:SGC-and-omega-paths} 
As shown in \cite[Theorem 5.2]{DRW10}, if $\Gamma_\otimes$ satisfies the small-gain condition, then there is an $\Omega$-path.
An explicit construction of an $\Omega$-path with a bit weaker regularity properties is given in 
 \cite[Proposition 2.7 and Remark 2.8]{KaJ11}. 
\end{remark}

%
%
%
%

Now we can state a theorem, that provides sufficient conditions
for input-to-state stability of a network, consisting of $n$ ISS subsystems.
\begin{theorem}
\label{thm:main1} 
Consider a well-posed interconnection $\Sigma$ of $n\in\N$ control systems $\Sigma_i$ with $i=1,\ldots,n$ and assume that 
all $\Sigma_i$ are ISS with corresponding ISS Lyapunov functions $V_i$ and internal gains $\chi_{ij}$.
If the corresponding operator $\Gamma_\otimes$ defined by \eqref{operator_gamma}
satisfies the small-gain condition \eqref{smallgaincondition}, then
the whole system $\Sigma$ is ISS and possesses an ISS Lyapunov function defined by 
\begin{align}
\label{NeuLyapFun}
V(x):=\max_{i}\big\{\sigma_{i}^{-1}(V_{i}(x_{i}))\big\},
\end{align}
 where $\sigma=(\sigma_{1},\ldots,\sigma_{n})^{T}$ is an $\Omega$-path. The Lyapunov gain of the whole system is 
\[
\chi(r):=\max_{i}\sigma_{i}^{-1}(\chi_{i}(r)).
\]
\end{theorem}

\begin{proof}
This an extension of \cite[Theorem 4]{DaM13}, where this result has been formulated for semilinear systems of the form \eqref{InfiniteDim}, which is based on the results in \cite{DRW10}.
\end{proof}

\subsection{Interconnections of integral ISS systems}
\label{sec:ISS-SGT-iISS-couplings}

Small-gain theory can also be developed to study couplings of integral input-to-stable systems (please see Section~\ref{sec:iISS} for the definitions). The obtained results and the challenges encountered on this way (e.g.,\ the insufficiency of max-type Lyapunov functions, which we have used in Section~\ref{sec:ISS-SGT-Lyapunov-Formulation}), have been excellently explained in \cite{Ito13}.
The formulations and proofs of small-gain theorems for couplings of iISS systems can be extended to the infinite-dimensional setting without radical changes, see \cite{MiI15b}. Nevertheless, the application of these theorems for particular examples is more involved than in the ISS case, as although $L_p$ theory works fine for ISS systems, for iISS systems, it is not always satisfactory and often Sobolev state spaces should be used \cite{MiI15b}. 

In this section we state an iISS small-gain theorem for the semilinear system 
\begin{align}
\label{eq:interconnection}
\begin{array}{l}
\dot{x}_i(t)=A_ix_i(t)+f_i(x_1,x_2,u) , \quad i=1,2, \\
x_i(t)\in X_i , \quad u\in \Uc , 
\end{array}
\end{align}
where $X_i$ is a state space of the $i$-th subsystem, and $A_i:D(A_i) \subset X_i \to X_i$ is a generator of a strongly continuous semigroup over $X_i$. 
Let $X:=X_1\times X_2$ which is the space of $x=(x_1, x_2)$, 
and the norm on $X$ is defined as 
$\|\cdot\|_X=\|\cdot\|_{X_1}+\|\cdot\|_{X_2}$. 
In this section, we assume that 
there exist continuous functions $V_i:X_i \to \R_+$, 
$\psi_{i1},\psi_{i2} \in \Kinf$, $\alpha_i \in \K$, $\sigma_i \in \K$  and $\kappa_i \in \K\cup\{0\}$ 
for $i=1,2$ such that 
\begin{equation}
\label{LyapFunk_1Eigi}
\psi_{i1}(\|x_i\|_{X_i}) \leq V_i(x_i) \leq \psi_{i2}(\|x_i\|_{X_i}), 
\quad \forall x_i \in X_i
\end{equation}
and system \eqref{eq:interconnection} satisfies 
\begin{equation}
\label{GainImplikationi}
\dot{V}_i(x_i) \leq -\alpha_i(\|x_i\|_{X_i})
 + \sigma_i(\|x_{3-i}\|_{X_{3-i}}) + \kappa_i(\|u(0)\|_U)
\end{equation}
for all $x_i\in X_i$, $x_{3-i}\in X_{3-i}$ and $u\in \Uc$, 
where the Lie derivative of $V_i$ corresponding to the inputs 
$u\in \Uc$ and $v\in PC_b( \R_+,X_{3-i})$ with $v(0)=x_{3-i}$ 
is defined by
\begin{equation}
\label{LyapAbleitungi}
\dot{V}_i(x_i)=\mathop{\overline{\lim}} \limits_{t \rightarrow +0} {\frac{1}{t}\Big(V_i\big(\phi_i(t,x_i,v,u)\big)-V_i(x_i)\Big) }.
\end{equation}

To present a small-gain criterion for the interconnected 
system \eqref{eq:interconnection} whose components are not 
necessarily ISS, we make use of 
a generalized expression of inverse mappings on the set of 
extended non-negative numbers $\overline{\R}_+:=[0,\infty]$. 
For $\omega\in\K$, define the function $\omega^\ominus$: 
$\overline{\R}_+\to\overline{\R}_+$ as 
\[
\omega^\ominus(s)=\sup \{ v \in \R_+ : s \geq \omega(v) \}=
\begin{cases}
\omega^{-1}(s) & \text{, if } s\in \text{Im}\ \omega, \\ 
+\infty & \text{, otherwise. } 
\end{cases}
\]
A function $\omega\in\K$ is extended to 
$\omega$: $\overline{\R}_+\to\overline{\R}_+$ as 
$\omega(s) := \sup_{v\in\{y\in\R_+ \, : \, y \leq s\}} \omega(v)$. 
These notations are useful for presenting the following result succinctly. 

\begin{theorem}
\label{theorem:intercon}
Suppose that 
\begin{align}
\label{eq:alpsig}
\displaystyle\lim_{s\rightarrow\infty}\!\alpha_i(s)=\infty
\ \mbox{or} \ 
\lim_{s\rightarrow\infty}\!\sigma_{3-i}(s)\kappa_i(1)<\infty
\end{align}
is satisfied for $i=1,2$. 
If there exists $c>1$ such that 
\begin{align}
\label{eq:sg}
\psi_{11}^{-1}\circ\psi_{12}\circ
\alpha_1^\ominus\circ c\sigma_1\circ
\psi_{21}^{-1}\circ\psi_{22}\circ
\alpha_2^\ominus\circ c\sigma_2(s)
\leq s
\end{align}
holds for all $s\in\R_+$, then system \eqref{eq:interconnection} is iISS. 
Moreover, if additionally $\alpha_i\in\Kinf$ for $i=1,2$, then system \eqref{eq:interconnection} is ISS. 
Furthermore, one can construct $\lambda_i\in\K$ (see \cite[Theorem 6]{MiI15b} for details) so that
\begin{align}
\label{eq:Vsum}
V(x)=\int_0^{V_1(x_1)}\lambda_1(s)ds + \int_0^{V_2(x_2)}\lambda_2(s)ds
\end{align}
is an iISS (ISS) Lyapunov function for \eqref{eq:interconnection}.
\end{theorem}

\begin{remark}
\label{rem:Simplified_SGC} 
Condition \eqref{eq:sg} can be called an iISS small-gain condition. On the first glance it may seem a bit technical, but it simplifies considerably if $V_i(x_i) = \psi_i(\|x_i\|_{X_i})$, for some $\psi_1,\psi_2\in\Kinf$. In this case $\psi_{11}=\psi_{12}=\psi_i$, $i=1,2$ and \eqref{eq:sg} takes the form
\begin{align}
\label{eq:sg_2}
\alpha_1^\ominus\circ c\sigma_1\circ
\alpha_2^\ominus\circ c\sigma_2(s)
\leq s.
\end{align}
The term $\alpha_i^\ominus\circ \sigma_i$ can be interpreted as a Lyapunov gain of the $i$-th subsystem.
This interpretation justifies the name \q{small-gain condition} for \eqref{eq:sg}.
\end{remark}

\begin{remark}
\label{rem:Decay-rates-in-iISS-formulation} 
In Theorem~\ref{theorem:intercon} we required that the decay rates of the iISS Lyapunov functions $\alpha_i$ are $\K$-functions. 
It was shown in \cite[Theorem 1]{CAI14} that existence of an iISS Lyapunov function with such a decay rate implies not only iISS but also strong iISS. This result can be transferred similarly to infinite-dimensional systems.
Thus, in Theorem~\ref{theorem:intercon} we implicitly assume that the subsystems are not only iISS, but strongly iISS.
\end{remark}

In small-gain theory for networks of ISS systems, the Lyapunov function is usually constructed in the maximization form, see \cite{DaM13, DRW10}, etc. The use of the summation form \eqref{eq:Vsum} for systems which are not necessarily ISS is motivated by the limitation of the maximization form and clarified in \cite{IDW12} for finite-dimensional systems. 

Having Theorem~\ref{theorem:intercon} in mind, the following strategy can be used to study ISS/iISS of coupled systems: first to construct ISS or iISS Lyapunov functions for subsystems, next to verify the small-gain condition and to apply small-gain results to justify ISS or iISS of the interconnection respectively.

However, even construction of Lyapunov functions for subsystems can be a complex task, and not only in view of a higher complexity in dealing with Lyapunov functions in infinite dimensions, but also due to the necessity to choose the state spaces in a right way and to match them with the state and input spaces for other subsystems. 
In particular, for construction of iISS and ISS Lyapunov functions for some classes of nonlinear parabolic systems, it is necessary to exploit Sobolev spaces as state spaces. An example of using Theorem~\ref{theorem:intercon}
together with Proposition~\ref{ConverseLyapunovTheorem_BilinearSystems} and other results on constructions of ISS Lyapunov functions to study the stability of coupled highly nonlinear parabolic systems one can find in \cite{MiI15b}.

\subsection{Cascade interconnections}
\label{sec:Cascade interconnections}

Consider an interconnection of two evolution equations of the form \eqref{eq:interconnection}, where the right-hand side $f_2$ does not depend on $x_1$. In other words, the second subsystem does not depend on the dynamics of the first system, but it influences the dynamics of the first system.
Such interconnections are called \emph{cascade interconnections}. 

The fact that a \emph{cascade interconnection of two input-to-state stable ODE systems is itself ISS}, was already shown in \cite{Son89}, and can be obtained as a trivial consequence of Theorem~\ref{thm:ISS_SGT} for cascade interconnections of an arbitrary finite number of systems of a rather general nature.
In contrast to this, a cascade interconnection of two iISS systems is iISS only under some additional conditions, see \cite{AAS02, ChA08}. If these conditions are not met, a cascade interconnection of two iISS systems is not necessarily iISS, as shown in \cite[Example 1]{AAS02}. The failure of this important property was one of the motivations to introduce the strong iISS property in \cite{CAI14}. In \cite{CAI14b} it was shown that cascade interconnections of strongly iISS ODE systems are again strongly iISS.

\subsection{Example}
\label{sec:Example-for-SGT}

In this section, we show how the iISS small-gain theorem can be applied for the analysis of stability of coupled parabolic PDEs.
Consider the system 
\begin{equation}
\label{GekoppelteNonLinSyst}
\left
\{
\begin{array}{l} 
{\frac{\partial x_1}{\partial t}(z,t) = \frac{\partial^2 x_1}{\partial z^2}(z,t) + x_1(z,t)x^4_2(z,t),} \\[1ex]
{x_1(0,t) = x_1(\pi,t)=0,} \\[1ex]
{\frac{\partial x_2}{\partial t}(z,t) =  \frac{\partial^2 x_2}{\partial z^2}(z,t) + ax_2(z,t) 
- bx_2(z,t) \Big(\frac{\partial x_2}{\partial z}(z,t)\Big)^{\!\!2} + \Big( \frac{x^2_1(z,t)}{1+x^2_1(z,t)} \Big)^{\!\frac{1}{2}}}
, \\[1ex]
{x_2(0,t) = x_2(\pi,t)=0.} 
\end{array}
\right.    
\end{equation}
defined on the region $(z,t) \in (0,\pi) \times (0,\infty)$. 
The state spaces for subsystems we choose as $X_1:=L_2(0,\pi)$ for $x_1(\cdot,t)$ and $X_2:=H^1_0(0,\pi)$ for $x_2(\cdot,t)$.

Our aim is to analyze for which values of parameters $a,b$ the overall system is 0-UGAS.
It may be difficult to find a Lyapunov function ensuring this property directly, thus we exploit the iISS small-gain theorem. We omit most of the computations and refer to \cite{MiI15b} for the full analysis.
\begin{itemize}
    \item[(i)] $x_1$-subsystem is a generalized bilinear system as in    \eqref{BiLinSys}, and thus it is iISS and possesses an iISS Lyapunov function, given by
\begin{eqnarray}
V_1(y_1):=\ln\Big(1+\|y_1\|^2_{L_{2}(0,\pi)} \Big),
\label{eq:V1}
\end{eqnarray}
with the corresponding Lie derivative:
\begin{eqnarray}
\dot{V}_1(y_1) \leq - \frac{2\|y_1\|_{L_2(0,\pi)}^2}{1+\|y_1\|_{L_2(0,\pi)}^2}  + 8  \|y_2\|^4_{H^1_{0}(0,\pi)}.
\label{eq:Z1dot_Final}
\end{eqnarray}

\item[(ii)] $x_2$-subsystem is ISS, which can be proved by construction of an ISS Lyapunov function
\begin{eqnarray}
V_2(y_2)=  \int_0^{\pi}{\Big(\frac{\partial y_2}{\partial l}(z)\Big)^2  dz} = 
 \|y_2\|^2_{H^1_0(0,\pi)},
\label{eq:LF_H10_Example}
\end{eqnarray}
satisfying for any $\omega>0$ the dissipation inequality
\begin{align}
\dot{V}_2(y_2) \leq  -2(1-a-\frac{\omega}{2})\|y_2\|^2_{H^1_0(0,\pi)}
-\frac{2b}{3\pi}\|y_2\|^4_{H^1_0(0,\pi)}
{+} \frac{\pi}{\omega}\Big( 
\frac{\|y_1\|^2_{L_{2}(0,\pi)}}{1+\|y_1\|^2_{L_{2}(0,\pi)}}\Big) . 
\label{eq:x2iss}
\end{align}

\item[(iii)] 
Now we collect the findings of (i) and (ii).
Assume that $a<1$ and $b\geq 0$. 
For the space $X=L_2(0,\pi)\times H^1_0(0,\pi)$, 
the Lyapunov functions defined as \eqref{eq:V1} and 
\eqref{eq:LF_H10_Example} for the two subsystems satisfy 
\eqref{LyapFunk_1Eigi} 
with the class $\K_\infty$ functions 
$\psi_{11}=\psi_{12}: s \mapsto \ln(1+s^2)$ and $\psi_{21}=\psi_{22}: s \mapsto s^2$. 
Due to \eqref{eq:Z1dot_Final} and \eqref{eq:x2iss}, 
we have \eqref{GainImplikationi} for 
\begin{align}
&
\alpha_1(s)=\frac{2s^2}{1+s^2} 
, \quad  
\sigma_1(s)=8s^4
, \quad  
\kappa_1(s)=0,
\\
&
\alpha_2(s)=2\left(1\!-\!a\!-\!\frac{\omega}{2}\right)\!s^2 + \frac{2b}{3\pi}\!s^4
, \hspace{1ex} 
\sigma_2(s)=\frac{\pi}{\omega}\left(\!\frac{s^2}{1+s^2}\!\right), \;
\kappa_2(s)=0, 
\end{align}
defined for any $\omega\in(0,2(1-a)]$. 
For these functions, condition \eqref{eq:sg} holds 
for all $s\in\R_+$ if and only if 
\begin{align}
\frac{12c^2\pi^2}{b\omega}\left(\!\frac{s^2}{1+s^2}\!\right) 
\leq 
\frac{2s^2}{1+s^2} 
, \quad \forall s\in\R_+
\label{eq:exsg}
\end{align}
is satisfied. Thus, there exists $c>1$ such that \eqref{eq:sg} holds 
if and only if ${6\pi^2}/{b}<\omega$ holds. 
Combining this with $\omega\in(0,2(1-a)]$, $a<1$ and $b\geq 0$, 
Theorem \ref{theorem:intercon} establishes UGAS
of $x=0$ for the whole system \eqref{GekoppelteNonLinSyst} when 
\begin{align}
a+\frac{3\pi^2}{b}<1 , \quad  b\geq 0. 
\label{eq:expibassum}
\end{align}

\end{itemize}

\subsection{Interconnections of an infinite number of systems}
\label{sec:ISS-SGT-infinite-interconnections}

Stability and control of infinite interconnections have received significant attention during the last decades. In particular, a large body of literature is devoted to spatially invariant systems consisting of an infinite number of components, interconnected with each other by means of the same pattern \cite{BPD02, BaV05, BeJ17, CIZ09}, etc.
Input-to-state stability theory can be applied to general infinite interconnections with nonlinear components without the assumption of the spatial invariance, which is the subject of an active ongoing research, see recent papers \cite{KMS19, DMS19a, DaP17,Mir19d,NMK20b}.

\textbf{Infinite networks with linear gains.} In \cite{KMS19} tight small-gain conditions for networks consisting of a countably infinite number of finite-dimensional continuous-time systems have been developed.
The main assumption in \cite{KMS19} is that each subsystem is exponentially ISS with respect to internal and external inputs and possesses an exponential ISS Lyapunov function, which is given in dissipative form. 
The associated gain functions reflecting the interaction with neighbors are assumed to be linear. The corresponding gain operator, which collects all the information about the internal gains and is denoted by $\Psi$, is given in a sum form and hence is linear as well.
Note that for infinite networks $\Psi$ acts in an infinite-dimensional space, in contrast to couplings of $N\in\N$ systems of arbitrary nature (possibly infinite-dimensional), studied in Section~\ref{sec:ISS-SGT-Lyapunov-Formulation}. 

The main result of \cite{KMS19} states that \emph{if $r(\Psi) < 1$, then the whole interconnection is exponentially ISS and it is possible to construct a \emph{coercive} exponential ISS Lyapunov function for the overall network as a weighted sum of ISS Lyapunov functions of subsystems}.
This result is a non-trivial generalization of \cite[Proposition 3.3]{DIW11} from finite networks to infinite networks. 
The argument for finite networks in~\cite{DIW11} is based on the Perron-Frobenius theorem.
However, existing infinite-dimensional versions of the Perron-Frobenius theorem including the Krein-Rutman theorem \cite{KrR48}, are \emph{not} applicable to infinite couplings as they require at least quasi-compactness of the gain operator, which is a quite strong assumption.
To overcome this obstacle, in \cite{KMS19} alternative methods of the spectral theory of linear positive operators in ordered Banach spaces are applied in \cite{KMS19}.

In \cite{NMK20b} the small-gain theorem from \cite{KMS19} has been extended to address exponential input-to-state stability with respect to closed sets, which in turn was applied to the stability analysis of infinite \emph{time-varying} networks, to \emph{dynamic weighted average consensus} in infinite multi-agent systems, as well as to the design of \emph{distributed observers} for infinite networks.

\textbf{Infinite networks with nonlinear gains.} In contrast to the case of linear gains, described in the previous passage, there are no decisive small-gain results for infinite networks with nonlinear gains. Next we describe some partial results, which reveal the challenges arising on this way.
In~\cite{DaP19} it is shown that a countably infinite network of ISS systems is ISS, provided that the gain functions capturing the influence from the neighboring subsystems are all less than identity which is a rather conservative condition.
By means of examples, it is shown in~\cite{DMS19a} that \emph{classic max-form small-gain conditions developed for finite-dimensional systems~\cite{DRW10} do not imply ISS of the interconnection in the case of infinite networks, even for linear ones}.
To address this issue, more restrictive \emph{robust strong small-gain conditions} are developed in~\cite{DMS19a}.
While the small-gain theorems in~\cite{DMS19a, DaP19} are formulated in terms of ISS Lyapunov functions, a trajectory-based small-gain theorem for infinite networks is provided in~\cite{Mir19d}, where the key role is played by a kind of \q{monotone invertibility} of $\id - \Gamma$, where $\Gamma$ is the gain operator.

\begin{openprob}
\label{ob:infinite-interconnections}
Small-gain results presented in this section use various types of the small-gain conditions: spectral small-gain condition \cite{KMS19},
robust strong small-gain condition \cite{DMS19a}, existence of a so-called $\Omega$-path for a gain operator $\Gamma$ \cite{DMS19a}, as well as 
\q{monotone invertibility} of $\id - \Gamma$ \cite{Mir19d}.
The relations between these conditions are known for finite networks, but are unclear for infinite networks. Understanding of these relationships and using of this knowledge to obtain ISS small-gain theorems for infinite networks with nonlinear gains is a challenging open problem.
\end{openprob}

\section{Input-to-state stability of time-delay systems}
\label{sec:TDS}

The study of ISS of time-delay systems has been initiated in 1998 in a seminal paper \cite{Tee98}, and now it is a rich theory with a broad range of applications. In this section, we give a rather brief overview of the ISS theory of retarded time-invariant differential equations, with an emphasis on the relationship between the ISS theory of delay systems and the ISS theory for general infinite-dimensional systems, described previously. Such important topics as ISS of time-variant retarded differential equations, systems with varying delays, neutral differential equations as well as iISS theory of delay systems are outside of the scope of this survey. 
Applications of ISS theory to robust control of delay systems are not covered here as well, please see, e.g., \cite{KMM16} for some results in this direction.

\subsection{Retarded differential equations}

We consider \emph{retarded differential equations} of the form
\begin{eqnarray}
\dot{x}(t)=f(x_t,u_t), \quad t>0,
\label{eq:time-delay}
\end{eqnarray}
where we denote $x_t(s):=x(t+s)$, $u_t(s)=u(t+s)$, $s\in[-T_d,0]$, for all $t\geq 0$,
and $T_d>0$ is the fixed (maximal) time-delay.

For \eqref{eq:time-delay} we choose for a certain $n\in\N$ the state space $X:= C([-T_d,0],\R^n)$, endowed with the usual supremum norm, defined for any $x\in X$ by 
\[
\|x\|_X:=\sup_{t\in[-T_d,0]}|x(t)|.
\]

We assume here that $U:=\R^m$ and that input $u$ belongs to the space
$\Uc:=L_\infty([-T_d,+\infty),U)$ of globally essentially bounded, measurable functions\\  $u:[-T_d,+\infty) \to U$. The norm of $u \in \Uc$ is given by
$\|u\|_{\Uc}:=\esssup_{t \geq -T_d}|u(t)|$. 

\begin{definition}
\label{def:Solution-of-TDS-system} 
We say that $\zeta \in C([-T_d,\tau],\R^n)$, $\tau>0$ is a \emph{solution of \eqref{eq:time-delay} on $[-T_d,\tau)$} subject to an initial condition $x \in X$ and an input $u\in\Uc$, if $\zeta$ is absolutely continuous on $[-T_d,\tau]$, satisfies the initial condition $\zeta(s) = x(s)$ for $s\in[-T_d,0]$ and the equation $\dot{\zeta}(t) = f(\zeta_t,u_t)$ holds almost everywhere on $(0,\tau)$.

We say that $\zeta \in C([-T_d,+\infty),\R^n)$ is a \emph{solution of \eqref{eq:time-delay} on $[-T_d,+\infty)$} subject to an initial condition $x \in X$ and an input $u\in\Uc$, if $\zeta$ is a solution of \eqref{eq:time-delay} on $[-T_d,s)$ for each $s>0$.
\end{definition}

For any function $g:\R\to S$, where $S$ is an arbitrary set, define by $g|_{[a,b]}$ a restriction of $g$ to the interval $[a,b]\subset \R$, that is a function $g|_{[a,b]}:[a,b]\to S$ with $g|_{[a,b]}(s) = g(s)$, $s\in[a,b]$.

\begin{proposition}
\label{prop:TDS-control system} 
Assuming that for each $x\in X$ and each $u\in\Uc$ there exists the unique (maximal) solution $\zeta$ as in Definition~\ref{def:Solution-of-TDS-system} of \eqref{eq:time-delay} with the interval of existence $[-T_d,t^*(x,u))$, consider a map $\phi:D_\phi\to X$ by setting $D_\phi:=\cup_{x\in X,\ u\in\Uc}[0,t^*(x,u))\times\{(x,u)\}$, and $\phi(t,x,u):=\zeta|_{[t-T_d,t]}$, $t\in[0,t^*(x,u))$. 
Then $\Sigma:=(X,\Uc,\phi)$ is a control system according to Definition~\ref{Steurungssystem}.
\end{proposition}

\begin{proof}
Direct verification of the axioms of Definition~\ref{Steurungssystem}. 
\end{proof}

To ensure local existence and uniqueness of solutions of \eqref{eq:time-delay} one can pose the following assumption on the nonlinearity $f$.
\begin{Ass}
\label{ass:TDS:Assumption1} 
We suppose that:
\begin{enumerate}[(i)]  
    \item $f:X \times L_\infty([-T_d,0],\R^m) \to \R^n$ is Lipschitz continuous in $x$ on bounded
subsets of $X$ and $L_\infty([-T_d,0],\R^m)$, i.e.  for
all $C>0$, there exists a $L_f(C)>0$, such that for all $ x,y \in B_C $
and for all $v \in L_\infty([-T_d,0],\R^m)$, $\|v\|_\infty \leq C$, it holds that
\begin{eqnarray}
|f(x,v)-f(y,v)| \leq L_f(C) \|x-y\|_X.
\label{eq:TDS:Lipschitz}
\end{eqnarray}
    \item $f(x,\cdot)$ is continuous for all $x \in X$.
		\item For each $u\in\Uc$ and each $x \in X$ the map $t \mapsto f(x,u_t)$ is measurable in $t$.
    \item For each $u\in\Uc$ there is a locally Lebesgue integrable function $\mu_u:\R\to\R_+$ so that $|f(0,u_t)|\leq \mu_u(t)$ for all $t\in\R_+$.
\end{enumerate}
\end{Ass}

\begin{proposition}
\label{prop:TDS-as-control-systems} 
If Assumption~\ref{ass:TDS:Assumption1} holds, then $\Sigma:=(X,\Uc,\phi)$ is a control system.
\end{proposition}

\begin{proof}
Pick any $u\in\Uc$, and consider the system
\begin{eqnarray}
\dot{x} = f_u(x_t,t):=f(x_t,u_t).
\label{eq:Hilfssystem-TDS}
\end{eqnarray}
Clearly, $f_u$ is Lipschitz continuous w.r.t. the first argument (on bounded subsets of $X$).
Item (iii) of Assumption~\ref{ass:TDS:Assumption1} ensures that for each fixed $x\in X$ the map $t\mapsto f_u(x,t)$ is measurable.
Items (i) and (iv) of Assumption~\ref{ass:TDS:Assumption1} ensure that for each given $R>0$ and all $x\in B_R$ it follows that
\begin{eqnarray}
\label{eq:Caratheodory-conditions}
|f_u(x,t)| = |f(x,u_t)| &\leq& |f(x,u_t) - f(0,u_t)| + |f(0,u_t)| \leq L_f(R) R + \mu_u(t).
\end{eqnarray}
As $\mu_u$ is a locally Lebesgue integrable function, estimate \eqref{eq:Caratheodory-conditions} shows that 
the right-hand side $f_u$ of \eqref{eq:Hilfssystem-TDS} satisfies Caratheodory conditions, see \cite[Section 2.6]{HaV93},
and results in \cite[Sections 2.2, 2.6]{HaV93} show that for any $x\in X$ and any $u\in\Uc$ the equation \eqref{eq:time-delay} possesses for given initial condition $x$ and an input $u$ the unique solution.
By Proposition~\ref{prop:TDS-control system} we obtain that $\Sigma$ is a control system.
\end{proof}

An important special class of systems \eqref{eq:time-delay} are delay systems, which do not depend on past values of inputs:
\begin{eqnarray}
\dot{x}(t)=f\big(x_t,u(t)\big), \quad t>0.
\label{eq:time-delay-2}
\end{eqnarray}
For such systems the item {\color{blue}(iv)} in Assumption~\ref{ass:TDS:Assumption1} becomes redundant, as $t\mapsto \big|f\big(0,u(t)\big)\big|$ is a measurable function (as a composition of a continuous function $s\mapsto |f(0,s)|$ and a measurable function $u$.
Besides generality, one of the reasons to study systems  \eqref{eq:time-delay} is a possibility to study small-gain theorems for delay systems (for which it is necessary to consider the states of other subsystems as internal inputs).

Before we proceed to an overview of ISS results for delay systems \eqref{eq:time-delay}, let us mention that the ISS property for delay systems is usually restated in the literature in an equivalent form:
\begin{proposition}
\label{prop:ZeitVerz_ISS}
    System \eqref{eq:time-delay} is ISS if and only if there exist functions $\beta\in {\KL}$ and $\gamma\in\Kinf$, such that for every $x \in X$, every $u\in\Uc$, and all $t\in\R_+$, it holds that
\begin{align}
\label{eq:Gleichung_Verzoeg_LISS}
\left|x(t)\right|\leq \beta(\left\|x\right\|_X,t) + \gamma(\left\|u\right\|_{\infty}).
\end{align}
\end{proposition}

\begin{proof}
See, e.g., \cite[Proposition 1.4.2]{Mir12}.
\end{proof}

\begin{openprob}
\label{op:ISS-characterizations-TDS} 
Although the ISS Superposition theorem (Theorem~\ref{thm:UAG_equals_ULIM_plus_LS}) is valid for time-delay systems, it is possible that for time-delay systems tighter results can be obtained. In particular, it is not known whether forward completeness and boundedness of reachability sets are equivalent properties for time-delay systems (for the ODE case such a claim is true \cite{LSW96} and for general infinite-dimensional systems it is false \cite{MiW18b}), and whether LIM and ULIM properties are equivalent notions, see \cite{MiW17e} for more detailed discussions on this topic.
\end{openprob}

\subsection{ISS Lyapunov theory for time-delay systems}
\label{sec:ISS-theory-for-TDS}

\emph{As time-delay systems \eqref{eq:time-delay} and \eqref{eq:time-delay-2} are a special case of control systems considered in this paper, all results presented for general control systems, as direct Lyapunov theorem (Theorem~\ref{LyapunovTheorem}), ISS superposition theorem (Theorem~\ref{thm:UAG_equals_ULIM_plus_LS}) and small-gain theorem in terms of trajectories (Theorem~\ref{thm:ISS_SGT}), remain valid for delay equations.
    Linear time-delay systems also fall into the class of abstract linear systems studied in Section~\ref{sec:Linear_systems} (see, e.g., \cite[Section 2.4]{CuZ95},\cite{BaB78}, \cite{BaP05}), and thus the criteria for ISS of linear systems can be applied to linear delay systems.}

At the same time, Lyapunov theory for delay systems can be substantially refined, partly due to the fact that the state of retarded differential equations changes with time in a very specific way: at a given moment of time a nontrivial change of the state $x\in X = C([-T_d,0],\R^n)$ occurs only at zero time, and the history is just \q{continuously shifted back}. 
Two different types of Lyapunov techniques are mostly used for retarded differential equations: Lyapunov-Krasovskii functionals and Lyapunov-Razumikhin functions.

\subsubsection{Lyapunov-Krasovskii functionals}

We start with a more general Lya-punov-Krasovskii methodology.
Thanks to the special kind of dynamics of retarded differential equations, one can compute the derivative of a 
continuous functional $V: X \to \R_+$ with respect to the system \eqref{eq:time-delay-2} for any $x \in X$ and any continuous input $u$ as follows (\cite[p. 1007]{PeJ06}):

\begin{eqnarray}
\dot{V}_u(x)=\mathop{\overline{\lim}} \limits_{t \rightarrow +0} {\frac{1}{t}\Big(V\big(\phi^*(t,x,u)\big)-V(x)\Big)},
\label{eq:Driver's-derivative}
\end{eqnarray}
where \qquad
$\phi^*(t,x,u)(s) = 
\begin{cases}
x(s+t), &s \in [-T_d,-t],\\
x(0) + f(x,u(0))(t+s), & s \in [-t,0].
\end{cases}
$

The expression in \eqref{eq:Driver's-derivative} is also called Driver derivative, and its importance is that it gives a possibility to compute the derivative of the control system with respect to an input without knowledge of a future trajectory of the system (which is needed for general control systems, see \eqref{ISS_LyapAbleitung}). See \cite{Pep07} for more on Driver derivatives.

\begin{definition}
\label{def:noncoercive_ISS_LK-functional}
A continuous function $V:X \to \R_+$ is called an \textit{ISS Lyapunov-Krasovskii functional} for a system \eqref{eq:time-delay}, if there exist $\psi_1,\psi_2,\alpha \in \Kinf$ and $\chi \in \K$ such that
\begin{equation}
\label{LyapFunk_1Eig_nc_ISS_LK}
\psi_1(|x(0)|) \leq V(x) \leq \psi_2(\|x\|_X), \quad \forall x \in X
\end {equation}
and for all $x \in X$, any constant input $u\in \Uc$ the Driver derivative of $V$ along the trajectories of $\Sigma$ satisfies
\begin{equation}
\label{Delay-systems-DissipationIneq}
V(x) \geq \chi(|u(0)|) \qrq \dot{V}_u(x) \leq -\alpha(|x(0)|).
\end{equation}
\end{definition}

Following \cite{PeJ06} (assumption $H_{P1}$, p. 1007), we assume that Lyapunov-Krasovskii functionals satisfy the following hypothesis:
\begin{Ass}\label{ass:10:pepe}
For any $x\in X$, any $u\in \Uc$ and the corresponding absolutely continuous solution $\phi(\cdot,x,u)$
of \eqref{eq:time-delay-2} over a maximal interval $[0,b)$ the following holds for a function $w:[0,b)\to\R_+$, defined by $w(t):=V(\phi(t,x,u))$:
\begin{itemize}
    \item $w$ is locally absolutely continuous in $[0,b)$
    \item it holds that 
\[
    \mathop{\overline{\lim}} \limits_{h \rightarrow +0} {\frac{w(t+h) - w(t)}{h}} = \dot{V}_u\big(\phi(t,x,u)\big).
\]
\end{itemize}
\end{Ass}

It has been shown in \cite{Pep07} that the second bullet of Assumption \ref{ass:10:pepe} holds almost everywhere if $V$ is Lipschitz on bounded sets (see \cite[Theorem 2]{Pep07}).

As $|x(0)|\leq \|x\|_X$ and $ -\alpha(\|x\|_X) \leq  -\alpha(|x(0)|)$, coercive Lyapunov functions introduced in Definition~\ref{def:noncoercive_ISS_LF} are also Lyapunov-Krasovskii functionals.
As the lower bound for a Lyapunov-Krasovskii functional is given in terms of $|x(0)|$, it resembles a non-coercive ISS Lyapunov function. However, the decay rate of the functional is also given in terms of $|x(0)|$, which is much weaker than the decay in terms of $V(x)$ or $\|x\|_X$, required from Lyapunov functions as defined in Definition~\ref{def:noncoercive_ISS_LF}.
Thus, Lyapunov-Krasovskii functional is a fine-tuned specifically for delay systems version of a Lyapunov function, 
which is easier to construct and to use, and at the same time existence of such a functional is sufficient for ISS of a time-delay system \eqref{eq:time-delay-2}, as the following theorem states.


\begin{theorem}
\label{thm:ISS Lyapunov-Krasovskii Theorem} 
If there is an ISS Lyapunov-Krasovskii functional for \eqref{eq:time-delay-2}, then \eqref{eq:time-delay-2} is ISS.
\end{theorem}

\begin{proof}
In a somewhat weaker formulation, this result has been shown in an influential paper \cite{PeJ06}.
The result which we stated was shown recently in \cite[Theorem 2]{KLW17}, for the proof see \cite[Corollary 4.17]{Kan17}.
%
\end{proof}

\begin{remark}
\label{rem:ISS-LKF-UGAS-perspective} 
For a special case of delay systems without inputs (i.e. with $\Uc:=\{0\}$), Theorem~\ref{thm:ISS Lyapunov-Krasovskii Theorem} 
boils down to the classical Lyapunov-Krasovskii theorem for UGAS property, see \cite[Theorem 2.1]{HaV93}.
\end{remark}

\begin{openprob}
\label{op:relaxed-LK-functionals} 
A question whether it is possible to weaken the requirement \eqref{Delay-systems-DissipationIneq} further to merely
\begin{equation}
\label{Delay-systems-DissipationIneq-3}
|x(0)| \geq \chi(\|u\|_\Uc) \qrq \dot{V}_u(x) \leq -\alpha(|x(0)|),
\end{equation}
remains open, see \cite{CPM17} for some preliminary results and \cite{ChP18} for the positive answer of the closely related question for the iISS property.
\end{openprob}

The Lyapunov-Krasovskii methodology gives necessary and sufficient conditions for ISS of retarded differential equations.
Converse Lyapunov-Krasovskii theorems for UGAS property of retarded equations with disturbances have been obtained in \cite{Kar06}. 

This motivated the following converse ISS Lyapunov-Krasovskii theorem for delay systems, which is a special case of \cite[Theorem 3.3]{KPJ08}:
\begin{theorem}
\label{thm:Converse-LK-theorem} 
If \eqref{eq:time-delay-2} is ISS, then there is an \textit{ISS Lyapunov-Krasovskii functional} $V:X \to \R_+$ for a system \eqref{eq:time-delay-2}, which is Lipschitz continuous on bounded sets, such that for some $\psi_1,\psi_2 \in \Kinf$ it holds that
\begin{equation}
\label{LyapFunk_1Eig_ISS_LK-coercive}
\psi_1(\|x\|_X) < V(x) \leq \psi_2(\|x\|_X), \quad \forall x \in X
\end {equation}
and there is $\chi \in \K$ such that 
for all $x \in X$, and all $u\in \Uc$ it holds that
\begin{equation}
\label{Delay-systems-DissipationIneq-converse-LK-theorem}
\|x\|_X \geq \chi(\|u\|_{\Uc}) \qrq \dot{V}_u(x) \leq -V(x).
\end{equation}
\end{theorem}

\begin{proof}
System \eqref{eq:time-delay-2} falls within the class of systems considered in \cite{KPJ08} in we assume in \cite{KPJ08} that the delay systems are time-invariant, $D:=\{0\}$, $Y:=X$ and $H(t,x)=x$ for all $t\ge 0$ and $x\in X$.
In this case the UIOS property from \cite{KPJ08} is precisely the ISS property which we study in this section.

Now \cite[Theorem 3.3]{KPJ08} guarantees the existence of an ISS Lyapunov-Krasovskii functional with the properties as in the claim of this theorem.
Note that for time-invariant systems the functional $V$ constructed in \cite[Theorem 3.3, item (d)]{KPJ08}
does not depend on $t$, and almost Lipschitz continuity of $V$ (see \cite[Definition 2.2]{KPJ08}) becomes the Lipschitz continuity on bounded sets.
\end{proof}

\begin{remark}
\label{rem:Direct-and-converse-LK-results} 
Please note the difference between Theorems~\ref{thm:ISS Lyapunov-Krasovskii Theorem} and \ref{thm:Converse-LK-theorem}: 
In the direct ISS Lyapunov-Krasovskii Theorem~\ref{thm:ISS Lyapunov-Krasovskii Theorem} fairly mild assumptions on ISS Lyapunov-Krasovskii functionals are imposed, and the converse Theorem~\ref{thm:Converse-LK-theorem} establishes existence of a coercive ISS Lyapunov-Krasovskii functional with an exponential decay rate.
\end{remark}

%
%

\subsubsection{Lyapunov-Razumikhin functions}

Alternative Lyapunov methodology for stability analysis of delay systems is the Lyapunov-Razumikhin approach
which reduces the problem of stability analysis for time-delay systems to the stability analysis of delay-free systems, which simplifies the problem considerably.
In \cite{Tee98} the Lyapunov-Razumikhin sufficient condition for ISS of time-delay systems has been analyzed from the viewpoint of small-gain theorems.

\begin{definition}\label{DefLRfunction}
A continuous function $V:\R^n\rightarrow\R_+$ is called an \emph{ISS Lyapu-nov-Razumikhin function} for system \eqref{eq:time-delay}, if there exist functions $\psi_1,\psi_2, \gamma_1, \gamma_2\in\Kinf$ with $\gamma_1 <\id$ such that 
\begin{align*}
    \psi_1(|z|)\leq V(z) \leq \psi_2(|z|),\quad z\in\R^n,
\end{align*}
and for all $x \in X$ and all constant inputs $u \in\Uc$
\begin{eqnarray}
V(x(0))\geq \max\{\gamma_1(\|\bar{V}(x)\|),\gamma_2(|u(0)|)\} \qrq \dot{V}_u(x) \leq -\alpha(|x(0)|),
\label{eq:Implication-Razumikhin}
\end{eqnarray}
where $\bar{V}(x)(s) = V(x(s))$, $s \in[-T_d,0]$,\ \ $\|\bar{V}(x)\|:=\sup_{s\in[-T_d,0]}|V(x(s))|$ and
\[
\dot{V}_u(x):= \mathop{\overline{\lim}} \limits_{h \rightarrow +0} {\frac{1}{h}\Big(V\big(\phi(h,x,u)(0)\big) - V(x(0))\Big)}.
\]
\end{definition}

In \cite[Theorem 1]{Tee98} the following has been shown:
\begin{theorem}
\label{thm:Lyapunov-Razumikhin theorem} 
If there is an ISS Lyapunov-Razumikhin function for \eqref{eq:time-delay}, then \eqref{eq:time-delay} is ISS.
\end{theorem}

\begin{proof}
In \cite[Theorem 1]{Tee98} under the assumptions of the Theorem the UGS and bUAG properties are shown. 
ISS of \eqref{eq:time-delay} then follows from the ISS superposition Theorem~\ref{thm:UAG_equals_ULIM_plus_LS}.
\end{proof}

In \cite{Tee98} Lyapunov-Razumikhin approach has been applied to study the robustness of the ISS property of ODEs with respect to small delays at the control input.

In many situations, ISS Lyapunov-Razumikhin functions are easier to handle than the more complex Lyapunov-Krasovskii functionals. However, the Lyapunov-Razumikhin framework provides only sufficient conditions for ISS, which are not necessary, even for systems without inputs see \cite[Section 4]{TWJ12b} and \cite{GKC03}.

\subsection{Small-gain theorems: trajectory formulation}

Considerable attention has been devoted to small-gain theorems in terms of trajectories for time-delay systems.
To the knowledge of the authors, the first attempt to obtain ISS and, more generally, IOS (input-to-output stability) small-gain results for time-delay systems has been made in \cite{PTM06}. 

\begin{definition}
\label{def:AG}
We say that a control system $\Sigma=(X,\Uc,\phi)$ has the \emph{asymptotic gain (AG) property}, if there exists 
$\gamma\in\Kinf$ so that for all $x\in X$, $u\in\Uc$  
\[ \mathop{\overline{\lim}} \limits_{t\to\infty}\|\phi(t,x,u)\|_X\leq \gamma(\|u\|_{\Uc}).\]
\end{definition}

In \cite{PTM06} small-gain theorems for couplings of $2$ time-delay systems possessing UGS and AG properties have been derived, but no small-gain theorem for ISS property. As AG $\wedge$ UGS is (possibly) weaker than ISS for time-delay systems, the \emph{ISS} small-gain theorem has not been obtained in this work. Small-gain theorems for couplings of $n$ time-delay systems with AG $\wedge$ UGS properties have been obtained in \cite{TWJ09} and \cite{PDT13}.

The first ISS small-gain theorems, applicable for time-delay systems have been achieved in \cite{KaJ07}, where the small-gain theorems in terms of trajectories (in maximum formulation) have been shown for couplings of two control systems of a rather general nature, covering in particular time-delay systems. 

The obstacle that ISS is (at least potentially) not equivalent to AG $\wedge$ UGS, was overcome in \cite{TWJ12}
where ISS small-gain theorems for couplings of $n\ge 2$ time-delay systems have been obtained by using 
a Razumikhin-type argument, motivated by \cite{Tee98}. 
In this approach, the delayed state in the right-hand side of a time-delay system is treated as an input to the system, which makes the time-delay system a delay-free system with additional input. However, the transformation of time-delay systems to the delay-free form is not always straightforward.

The small-gain theorem for couplings of $n$ infinite-dimensional systems, which we discussed earlier (Theorem~\ref{thm:ISS_SGT}) is fully applicable to well-posed time-delay systems, as they are a special case of control systems in the sense of Definition~\ref{Steurungssystem}.

\subsection{Small-gain theorems: Lyapunov formulation}

Small-gain theorems for interconnections of $n$ nonlinear time-delay systems in a Lyapunov-Krasovskii and Lyapunov-Razumikhin formulation have been shown in \cite{DaN10}, and extended to impulsive time-delay systems in \cite{DKM12}.
A distinct to \cite{DaN10, DKM12} approach for stability analysis of coupled delay systems has been developed in \cite{KaJ11}, based on the use of vector Lyapunov functions. These results are proved for a broad class of infinite-dimensional systems, which encompasses time-delay systems, see \cite[Section 4.2]{KaJ11}.

\section{Applications}
\label{sec:Applications}

There are numerous applications of ISS theory to control of PDE systems. In this section, we briefly mention some of them, with an intention to show the scope of applications, rather than to be exhaustive.

The Stefan problem represents a liquid-solid phase change phenomenon which describes the time evolution of a material's temperature profile and the liquid-solid interface position. The closed-loop objective is to stabilize the interface position at the desired position for the one-phase Stefan problem without the heat loss. The problem is modeled by a 1-D heat equation defined on a time-varying spatial domain described by an ODE with a time-varying disturbance. This control problem is solved and ISS property is described in recent papers (see \cite{koga2019delay,koga2018control, koga2019input}).

The safety factor profile control is a crucial issue in tokamaks since it is about the control of the coupling between the poloidal flux diffusion
equation, the time-varying temperature profiles and an independent total plasma current control in fusion. ISS property is fundamental not only to reject external perturbations but also to study the coupling with dynamical actuators, as current control. See \cite{APW13b,AWP12,AWP13b,AWP13} for Lyapunov design control on the safety factor dynamics. See \cite{mavkov2017distributed,mavkov2017multi} with the coupling with electron
temperature profile and finally \cite{mavkov2018experimental} for real experiments on the tokamaks using Lyapunov based control and ISS properties.

As far as delay systems are concerned, the paper \cite{medina2017exponential} solves 
asymptotic stabilization of multiple delay systems by using the logarithmic norm technique combined with the \q{freezing}
method, in presence of slowly varying coefficients and nonlinear perturbations. Certain Lyapunov functions are computed for delay systems in \cite{seuret2017wirtinger}, and it could be interesting to see whether it is possible to write LMI conditions for the design of ISS Lyapunov functions, generalizing \cite{seuret2017wirtinger}.
In \cite{CDP17} the small-gain technique has been used to robustly stabilize delayed neural fields with partial
measurement and actuation. See also \cite{1911.10761.pdf,lhachemi2019An} for a recent control design with a time-varying delay.

Monotonicity methods have been applied to ISS stabilization of linear systems with boundary inputs and actuator disturbances in \cite{MKK19}. In \cite{PiO17} the variable structure control approach has been exploited to design discontinuous feedback control laws (with point-wise sensing and actuation) for ISS stabilization of linear reaction-diffusion-advection equations w.r.t. actuator disturbances. 

ISS property is instrumental for a various of control problems, as stabilization of infinite-dimensional systems (see, e.g., \cite{Event:espitia:aut:2019,karafyllis2018sampled}) or observer designs (see, e.g., \cite{observers:tac:2019}).
In \cite{KaJ11} vector small-gain theorems for wide classes of systems satisfying weak semigroup property have been proved and applied to the stabilization of the chemostat in \cite{KaJ12}.

Finally let us cite the works dealing with event-triggering for both hyperbolic PDEs \cite{EspiGira16,EGM18,ETT17} and parabolic PDEs \cite{selivanov2016distributed} where Lyapunov methods and ISS based triggering controls are designed. 

ISS properties of higher-order nonlinear infinite-dimensional systems could be also established. As an application example, consider the semi-linear partial differential
equation governing the motion of a railway track, as done in \cite{EdM19}. In this paper, the dynamics of flexible structures with a
Kelvin-Voigt damping is studied and a Lyapunov function approach is employed. The main contribution is the proof of an ISS property of mild solutions, where the external input is defined
as the force exerted on the railway track by moving trains, active dampers, or other external force.

\section{Further topics}
\label{sec:Further-topics}

In this section, we discuss rather briefly several other topics which received the attention of researchers recently.


\subsection{Strong and weak input-to-state stability}
\label{sec:WeakISS}
Consider a strongly continuous semigroup of linear bounded operators $T:=(T(t))_{t\geq 0}$ over a Banach space $X$.
Recall that $T$ is called \emph{strongly stable} if for any $x\in X$ it holds that $T(t)x\to 0$ as $t\to\infty$.
It is well-known that strong stability of $T$ is a much weaker property than exponential stability.
As there are important classes of control systems which are merely strongly stable \cite{Oos00} (and not exponentially stable), it is of interest to study the robustness of such systems with respect to external inputs in an ISS-like manner.
To this end in \cite{MiW18b} the following concept has been introduced:
\begin{definition}
\label{def:sISS}
System $\Sigma:=(X,\Uc,\phi)$ is called \emph{ strongly input-to-state stable (sISS)}, if 
$\Sigma$ is UGS and has a \emph{strong asymptotic gain (sAG) property}, which means that
there is $ \gamma \in \Kinf$ such that for all $ \eps >0$, $x\in X$ 
there is $ \tau=\tau(\eps,x) < \infty$ for which
\begin{equation}    
\label{sAG_Absch}
x\in X \ \wedge \ u\in\Uc \ \wedge \ t \geq \tau(\eps,x) \quad \Rightarrow \quad \|\phi(t,x,u)\|_X \leq \eps + \gamma(\|u\|_{\Uc}).
\end{equation}    
\end{definition}

%

ISS implies sISS, but the converse implication doesn't hold for infinite-dimensional systems in general.
In \cite[Theorem 12]{MiW18b} it was shown, that sISS property can be restated by means of an ISS-like inequality \eqref{iss_sum} with weaker properties of $\beta$.
For ODEs, the notions of sISS and ISS coincide, see \cite[Proposition 11]{MiW18b}.
Strong ISS of linear systems with unbounded input operators has been investigated in \cite{NaS18}.


Assuming in Definition~\ref{def:sISS}, that the time $\tau$ depends on $(\varepsilon,x,u)$ (and not only on $(\varepsilon,x)$),
we arrive at the so-called weak input-to-state stability property, introduced in \cite{ScZ18}, where 
nonlinear boundary controllers for a class of port-Hamiltonian systems have been constructed, which achieve weak ISS of the closed-loop system.
Characterizations of weak ISS have been reported in \cite{Sch19}.

\subsection{Input-to-state practical stability}
\label{sec:ISpS}

In some cases it is impossible (as in quantized control) or too costly to construct a feedback, ensuring ISS behavior of the closed-loop system. To address such applications, a relaxation of the ISS concept has been proposed in \cite{JTP94}, called input-to-state practical stability (ISpS, practical ISS). 

\begin{definition}
\label{Def:ISpS_wrt_set}
A control system $\Sigma=(X,\Uc,\phi)$ is called \emph{ (uniformly) input-to-state practically stable
(ISpS)}, if there are $\beta \in \KL$, $\gamma \in \Kinf$ and $c>0$
such that 
\begin {equation}
\label{isps_sum}
x\in X \ \wedge \ u\in\Uc \ \wedge \ t \geq 0 \quad \Rightarrow \quad 
\|\phi(t,x,u)\|_X \leq \beta(\|x\|_{X},t) + \gamma( \|u\|_{\Uc}) + c.
\end{equation}
\end{definition}
ISpS property is extremely useful for stabilization of stochastic control systems \cite{ZhX13}, control under quantization errors \cite{ShL12}, sample-data control \cite{NKK15}, study of interconnections of nonlinear systems by means of small-gain theorems \cite{JMW96,JTP94}, etc. 

In the context of infinite-dimensional systems ISpS has been studied in \cite{Mir19a}, where superposition theorems for ISpS property have been derived, some of which are new already for ODE systems. From these characterizations, it follows that existence of a non-coercive ISS Lyapunov function for a control system with a BRS property implies ISpS of this system.

\subsection{Lur'e systems and circle criterion}
\label{sec:Lur'e-systems}

Lur'e systems are ubiquitous in control theory since they come from the interconnection of linear systems with a static nonlinearity. Again the computation of Lyapunov functions is a fruitful machinery when combined with the S-procedure, Kalman-Yakubovich-Popov lemma and passivity arguments. Many works on Lur'e systems deal with ISS, as reviewed in \cite{JLR11}. Circle criterion provides a classical methodology to analyze the ISS property and for the computations of gains. It allows for the interplay of frequency-domain properties of the linear component with the sector conditions on the nonlinearity. Such approaches facilitate the design of stabilizing controllers with isolated nonlinearities, as nested or non-nested saturations (see  \cite{TarPriGom:06}), hysteresis and backlash operators (see \cite{TarbouriechPrieurQueinnec:Auto:10,tarbouriech:ieeetac:2014}), or quantized feedback systems (see \cite{ferrante2015stabilization,ferrante2018sensor}). 
This constructive method exploits conditions expressed in the form of linear matrix inequalities (LMI) that are numerically tractable and could be combined with optimization criterion. Not only finite-dimensional control systems could be considered in this context, but also infinite-dimensional systems (see in particular \cite{JLR11,MaP17,MCP17,MCP18,TMP18,MPW19a,PTS16}).

For linear systems (including infinite-dimensional ones), in closed-loop with Lur'e feedback laws, ISS properties are derived in \cite{GLO19, JLR08,JLR11} by exploiting the Laplace variable and the transfer function representation of this class of infinite-dimensional systems. Again some sector bounds are used with circle criterion. It allows to deal with backlash and play hysteresis. See \cite{JLR11} for an overview of these techniques and connections with ISS.

\subsection{Numerical computation of ISS Lyapunov functions}
\label{sec:Computation of ISS Lyapunov functions}

Finding an ISS Lyapunov function and a verification of the dissipation inequalities is not straightforward, especially for coupled PDE systems.
Paper \cite{AVP16} is devoted (among other problems) to numerical construction of ISS Lyapunov functions for evolution equations by means of the \emph{sum-of-squares (SOS) programming } method, provided the nonlinearities involved in the formulation of PDEs are polynomial.
Besides ISS, also other types of dissipativity conditions are studied in \cite{AVP16} as passivity and induced input-output norm boundedness.

\subsection{ISS for monotone parabolic systems}
\label{sec:Monotone_systems}

As a rule, it is much harder to analyze ISS of PDE systems with boundary inputs than ISS of PDEs with distributed inputs.
Thus, a natural desire is to \q{transform} boundary disturbances into the distributed ones. Unfortunately, this is not possible for general systems. However, as shown in \cite{MKK19}, for monotone control systems this transformation can be achieved in 2 steps: first due to monotonicity of the trajectory with respect to inputs one can estimate the solution for any inputs employing sup- and sub-solutions with constant inputs. 
Next by using the Dirichlet lifting approach one can transform a nonlinear PDE with boundary constant inputs into a related nonlinear PDE with constant distributed inputs, which makes the analysis of PDEs much simpler, as many approaches discussed in Section~\ref{sec:ISS_analysis_linear_nonlinear_PDEs_Lyapunov_methods} can be applied for a modified problem.
Please see also \cite{ZhZ19b} for some recent progress in using weak maximum principles for nonlinear parabolic systems.

\subsection{ISS of infinite-dimensional impulsive systems}
\label{sec:impulsive-systems}

Often in the modeling of real-world phenomena, one has to consider systems, which exhibit both continuous and discontinuous behavior.
A general framework for modeling of such phenomena is a hybrid systems theory
\cite{GST12,HCN06,SaP95}. Impulsive systems are hybrid systems whose state can jump only at moments, which are given in advance and do not depend on the state of the system.
Let ${\mathcal T} = \{t_1,\ t_2,\ t_3, \ldots \}$ be a strictly increasing sequence of impulse times without finite accumulation points.

Consider an impulsive system of the form
\index{system!impulsive}
\begin{equation}
\label{ImpSystem}
\left \{
\begin {array} {l}
{ \dot{x}(t)=Ax(t) + f(x(t),u(t)),\quad t \in [t_0,\infty)  \backslash {\mathcal T},} \\
{ x(t)=g(x^-(t),u^-(t)),\quad t \in {\mathcal T},}
\end {array}
\right.
\end {equation}
where $x(t) \in X$, $u(t) \in U$, $X$ and $U$ are Banach spaces, $\Uc:=PC_b(\R_+,U)$ and $A$ is the infinitesimal generator of a strongly continuous semigroup on $X$ and $f,g:X \times U \to X$.
Equations \eqref{ImpSystem} together with the sequence of impulse times ${\mathcal T}$ define an impulsive system.
The first equation of \eqref{ImpSystem} describes the continuous dynamics of the system, and the second describes the jumps of the state at impulse times.
This system is not within the class of systems, considered in Definition~\ref{Steurungssystem}, as the trajectories of the system are discontinuous, however, all the main concepts of the ISS framework can be extended to this class of systems.

If both continuous and discrete-time dynamics of \eqref{ImpSystem} are ISS, then the impulsive system is ISS uniformly with respect to the set of all impulse time sequences. 
At the same time, \emph{if either continuous or discrete-time part of \eqref{ImpSystem} is destabilizing, then the interplay between continuous and discrete-time dynamics becomes essential} and ISS can be verified only for some classes of impulse time sequences, described through \emph{dwell-time conditions}. 
In order to determine such conditions one constructs a common ISS Lyapunov function for both continuous and discrete dynamics, and studies the hybrid dynamics of the ISS Lyapunov function. In \cite{DaM13b} such analysis has been performed for systems \eqref{ImpSystem} with both exponential and non-exponential ISS Lyapunov functions, and the corresponding dwell-time-conditions have been developed.
In \cite{DKM12} ISS of impulsive time-delay systems has been studied by means of Lyapunov-Krasovskii functionals and Lyapunov-Razumikhin functions.

\section{Open problems}
\label{sec:Open-questions}
Here we outline some important open problems and perspective research fields in the ISS theory of infinite-dimensional systems and its applications to robust control. Other open questions have been introduced throughout this survey. 

\subsection{Infinite-dimensional integral ISS theory}

For finite-dimensional systems, the class of integral ISS systems is much broader than the class of ISS systems. At the same time, the counterparts of such important theoretical results in the ISS theory as the non-Lyapunov characterizations and criteria in terms of Lyapunov functions, as well as the small-gain theorems, can be established in the case of iISS systems, which has made integral ISS theory an important milestone in the development of ISS theory.
        On the other hand, the integral ISS theory for infinite-dimensional systems remains rather unexplored.
\begin{itemize}
    \item Superposition theorems (counterparts of Theorem~\ref{thm:UAG_equals_ULIM_plus_LS}) are not available (for ODEs see \cite{ASW00,Son98}).
    \item Full understanding of relations between ISS and integral ISS for linear systems with unbounded operators is still missing.
    \item ISS Lyapunov functions are by definition iISS Lyapunov functions. Thus the existence of iISS systems which are not ISS would suggest, that there are some limitations of the Lyapunov method at least for systems with unbounded input operators. Whether such limitations exist at all and how significant they are is a completely open question right now. 
    \item Properties of non-coercive iISS Lyapunov functions have not been investigated right now.
    \item Theory of couplings of infinite-dimensional iISS systems is far less complete than in the ISS case (see \cite{Ito13} for a survey of such results in ODE case).
    \item Integral ISS theory for bilinear systems with unbounded bilinear operators is in its infancy.
    \item Strong iISS concept, introduced in \cite{CAI14} for ODEs and briefly discussed in this survey (Sections~\ref{sec:iISS}, \ref{sec:Bilinear_systems}), has not been analyzed till now.
\end{itemize}

\subsection{Input-to-output stability theory}

In many cases, it is not necessary (or is a too strong requirement) to have the stability of a system with respect to the full set of state variables. This calls for the generalization of the input-to-state stability concept to the systems with outputs.

\begin{definition}
\label{def:system with outputs} 
A (time-invariant) \emph{system with outputs} is given by a system $\Sigma:=(X,\Uc,\phi)$ (where $X,\Uc$ are Banach spaces) together with
\begin{itemize}
    \item[(i)]   A normed linear space $Y$ called the \emph{measurement-value or output-value space}; and
    \item[(ii)]  A map $h : X \times \Uc \to Y$ called the \emph{measurement map}.
\end{itemize}
\end{definition}

Systems with outputs appear in various contexts in systems and control theory.
In particular, in many cases, the state of the system cannot be measured directly, and only certain functions of the state and input are available for control purposes. These functions can be treated as outputs, which gives rise to the system with outputs.

Another interest in such systems comes from the considerations of stability with respect to outputs.
\index{IOS}
\index{stability!input-to-output}
\begin{definition}
\label{def:IOS}
A forward-complete system with outputs $\Sigma$ is called \emph{input-to-output stable (IOS)}, if there exist $\beta \in \KL$ and $\gamma \in \Kinf$ 
such that for all $x \in X$, $u\in \Uc$ and all $t\geq 0$ the following holds
\begin {equation}
\label{ios_sum}
\|h(\phi(t,x,u),u)\|_Y \leq \beta(\|x\|_X,t) + \gamma(\|u\|_{\Uc}).
\end{equation}
\end{definition}
If $h(x,u)=x$, IOS boils down to the classical ISS property. Other choices for the output function can be: tracking error,
observer error, drifting error from a targeted set, etc.
In this case, IOS represents robust stability of control systems with respect to the given errors.
There are several properties, which are closely related to output stability, as partial stability, \cite{Vor05} and stability with respect to two measures, see, e.g., \cite{TeP00}.

For finite-dimensional systems the IOS theory is quite rich, see, e.g., \cite{ASW00,SoW99,TeP00}.
There is almost no result on IOS theory in the context of infinite-dimensional systems (except for delay systems). Some exceptions include the small-gain results for IOS systems in \cite{BLJ18,KaJ11}.

\subsection{ISS of fully nonlinear PDEs}

The absolute majority of PDE systems, which have been analyzed for robustness within the ISS framework, are either linear or belong to the class of semilinear systems, which can be written in the form \eqref{InfiniteDim}. In semilinear systems, the nonlinear part is assumed to be Lipschitz continuous, and in particular, it is bounded on bounded balls. Hence the \q{unbounded part} of the system \eqref{InfiniteDim} is linear.
    
    However, many important PDEs as porous medium equation \cite{Vaz07} and Navier-Stokes equations, do not fall into this category. 
This calls for looking onto a more general class of infinite-dimensional problems, which will include fully nonlinear PDEs.
For systems without inputs, the theory of nonlinear semigroups \cite{CrL71,Kat67,Min62} gives a powerful method for the unified study of such equations. General methods for investigation of ISS for systems from this class are highly desirable.

\subsection{Robust control design}

As it has been already explained, the Lyapunov approach is useful for ISS analysis of PDE systems and for the design of boundary controllers ensuring an ISS property of the closed-loop system. In this latter context, the theory and sufficient conditions, written in terms of Lyapunov functions, could be seen as a design condition of many boundary or internal robust controllers for various control problems. Related works on robust control design include $\mathcal{H}_\infty$ design as done in, e.g.,\ 
\cite{CuG86,CuZ95,robu2011simultaneous}. See also transfer functions techniques for robust control design as presented in particular in \cite{BaC16,bastin2015stability,JLR11,litrico:book:2009modeling}. Many open control problems exist in the literature for robust control designs, as the drilling problem (see \cite{roman2019robustness} and references therein), and the flexible structure control in aircraft (see in particular \cite{demourant2017new, poussot2016gust}) to point out only a few robust control problems. Use of ISS techniques for robust control design is, to a large extent, still an open problem.

\vspace{1em}

Above described problems constitute only a small portion of the problems which can be explored in the infinite-dimensional ISS theory.
In particular, \emph{ISS of time-varying infinite-dimensional systems} is in its infancy and \emph{ISS theory for infinite-dimensional discrete-time and hybrid systems} is almost unexplored.

\section{Conclusion and discussion}
\label{sec:Conclusion and Discussion}

In this paper, we outlined the state of the art in the input-to-state stability theory of infinite-dimensional systems.
It includes the results both for linear and nonlinear systems and encompasses such distinct methods as Lyapunov theory, dynamical systems, semigroup and admissibility theory, PDE theory, etc. 

Section~\ref{sec:Fundamental-properties-ISS-systems} contains a description of a rather general class of control systems and presents fundamental criteria for ISS of infinite-dimensional systems in terms of coercive and non-coercive ISS Lyapunov functions as well as ISS superposition theorems.
Powerful functional-analytic criteria for ISS of linear systems based on semigroup and admissibility theories, with a particular emphasis on linear boundary control systems, were provided in Sections~\ref{sec:Linear_systems} and \ref{sec:Boundary_control_systems}.
Next important classes of control systems were studied in Section~\ref{sec:ISS_analysis_linear_nonlinear_PDEs_Lyapunov_methods}, where it has been shown how ISS Lyapunov functions can be exploited for ISS analysis of nonlinear PDE systems with both boundary and distributed disturbances. 
In Section~\ref{sec:Interconnected_systems} sufficient conditions for stability of networks consisting of infinite-dimensional ISS systems were given. Special attention was given to delay systems in Section \ref{sec:TDS}, in which the connections between the ISS theory for delay systems and previously presented general methods were derived. In Section~\ref{sec:Applications} a short review of broad applications of ISS theory was given. Section \ref{sec:Further-topics} contains further results in infinite-dimensional ISS theory, which do not fall in the context of the previous sections.
Throughout the survey, it was emphasized that the ISS theory of infinite-dimensional systems is not a complete subject and a number of challenging questions are open, some of which are stated throughout the text and in Section~\ref{sec:Open-questions}.

\section{Appendix}
\label{sec:Appendix}

\subsection{Inequalities}
\label{sec:Inequalities}

We collect here several inequalities, used throughout the paper.

\begin{proposition}[Cauchy's inequality with $\varepsilon$]
\label{theorem:Young}
For all $a,b \geq 0$ and all $\omega,p>0$ it holds that
\begin{eqnarray}
\label{ineq:Cauchy-with-eps}
ab \leq \frac{\varepsilon}{2}a^2 + \frac{1}{2\varepsilon} b^2.
\end{eqnarray}
\end{proposition}

\begin{proof}
See \cite[p. 20]{Mit64}.
\end{proof}

%

%

\begin{proposition}[Jensen's inequality]
\label{theorem:Jensen}
For any convex $f:\R \to \R$ and any summable $x$  
\begin{eqnarray}
\label{ineq:Jensen}
 \int_0^L f(x(l)) dl \geq L f\Big( \frac{1}{L} \int_0^L x(l)dl\Big).
\end{eqnarray}
\end{proposition}

\begin{proof}
See \cite[p. 705]{Eva10}.
\end{proof}

\begin{proposition}[Poincare's inequality]
\label{theorem:Wirtinger}
For every $x \in W^{2,1}(0,L)$ so that either $x(0)=0$ or $x(1)=0$ it holds that
\begin{align}
\label{Wirtinger_Variation_Ineq}
\frac{4 L^2}{\pi^2} \int_0^L{\left( \frac{\partial x(l)}{\partial l} \right)^2 dl} \geq \int_0^L {x^2(l)}dl.
\end{align}
\end{proposition}

%


\begin{proposition}[Agmon's inequality]
\label{theorem:Agmon}
For all $f \in H^1(0,L)$ it holds that
\begin{align}
\label{ineq:Agmon}
\|f\|^2_{L_{\infty}(0,L)} \leq |f(0)|^2 + 2\|f\|_{L_{2}(0,L)}\Big\|\frac{df}{dl}\Big\|_{L_{2}(0,L)}.
\end{align}
\end{proposition}

\begin{proof}
See \cite[Lemma 2.4., p. 20]{KrS08}.
\end{proof}

\subsection{Glossary}
Here we list the most important notation used in the paper.


\begin{longtable}[h]{r p{.85\textwidth} } 
\caption*{\textbf{Sets, numbers and analysis}}
\endfirsthead
\endhead
	$\N$   & Set of natural numbers 1, 2, 3,\ldots \\
	$\Z,\ \Q,\ \R,\ \C$  & Sets of integer, rational, real and complex numbers respectively\\
	$\R_+$  & Set of nonnegative real numbers \\
	$\overline{\R}_+$  & $:=[0,\infty]$ (extended set of nonnegative real numbers) \\
  $\re z$, $\im z$  & Real and imaginary parts of $z \in\C$\\
	$S^n$  & $:=\underbrace{S \times \ldots \times S}_{n \text{ times}}$, for any set $S$ \\
	$B_{r,W}$ & $:=\{u \in W: \|u\|_W < r\}$ (the open ball of radius $r$ around $0$ in the normed linear space $W$)\\
  $B_r$  & Open ball of radius $r$ around $0\in X$, i.e. $\{x\in X:\|x\|_X<r\}$\\						
	$\mathop{\overline{\lim}}$ & Limit superior\\
	$\overline{S}$ & $:=\{f\in {\cal L}: \exists \{f_k\}\subset S \mbox{ s.t. } \|f_k-f\|_{\cal L}\to 0,\ k\to\infty\}$ (the closure of $S$ in a  certain normed linear space $\cal L$).\\
  $\nabla f$  & Gradient of a function $f:\R^n \to \R$ \\
	$f \circ g$  & Composition of maps $f$ and $g$: $f\circ g (s)= f(g(s))$, for $s$ from the domain of definition of $g$, with $g(s)$ in the domain of definition of $f$  \\
	$\partial G$  &  Boundary of a domain $G$ \\	
	$\intt G$  &  Interior of a domain $G$ \\	
	$\mu$  & Lebesgue measure on $\R$ \\
	$\id$  & Identity operator \\
	$A^*$  & Adjoint of an operator $A$
\end{longtable}
%
\vspace{-6mm}
\begin{longtable}{r p{.85\textwidth} } 
\caption*{\textbf{Matrices and linear maps}}
\endfirsthead
\endhead
$I$ & Identity matrix   \\
$\mathcal{D}_+^n$ & The set of diagonal positive definite matrices in $\mathbb{R}^n$\\
$\sigma(M)$ & The \emph{spectrum} of an operator $M$\\
$\rho(M)$ & The \emph{resolvent set} of an operator $M$\\
  $x^\top$, $A^\top$ & Transposition of a vector $x \in \R^n$ and matrix $A \in \R^{n\times m}$	 \\
	$|\cdot|$  &  Euclidean norm in the space $\R^s$, $s \in \N$   \\
	$D(A)$   &  Domain of a linear operator $A$ \\
	$\|A\|$   &  $:=\sup_{x\neq0}\frac{|Ax|}{|x|}$ (the norm of $A \in \R^{n \times n}$, induced by Euclidean norm) \\
	$\lambda_{\max}(M)$ & $:=\max_{\lambda\in\sigma(M)}|\lambda|$ (the \emph{spectral radius} of a bounded operator $M$) \\
	$\lambda_{\min}(M)$ & $:=\min_{\lambda\in\sigma(M)}|\lambda|$ (for any closed $M$ with $\sigma(M)\neq\emptyset$)\\
	$\sym(M)$ & $:=\frac{1}{2}(M+M^\top)$ (the symmetric part of a matrix $M\in \R^{n\times n}$)\\
& Let $X,Y$ be two vector spaces. For a map $f:X \to Y$   \\
	$\ker (f)$ & $ :=\{x \in X: f(x) = 0\}$ is the \emph{kernel} (or \emph{nullspace}) of $f$\\
	$\im (f) $ & $ :=\{f(x):x\in X\}$ is the \emph{image} (or \emph{range}) of $f$ \\
  $l_p$  & Space of sequences $x=\{x_k\}_{k=1}^\infty$ so that $\|x\|_{l_p}<\infty$, where $p\in[1,+\infty]$\\
	$\|x\|_{l_p}$ &	$:=\Big(\sum_{k=1}^\infty |x_k|^p\Big)^{1/p}$, $p\in[1,+\infty)$\\											
	$\|x\|_{l_\infty}$ &	$:=\sup_{k=1}^\infty |x_k|$.												
\end{longtable}

\vspace{-6mm}
\begin{longtable}{r p{.85\textwidth} } 
\caption*{\textbf{Function and sequence spaces}}
\endfirsthead
\endhead
	$C(X,U)$  & Space of continuous functions from $X$ to $U$ with finite norm \\
	          & \phantom{space}  $\|u\|_{C(X,U)}:=\sup\limits_{x \in X}\|u(x)\|_U$ \\
	$PC_b(\R_+,U)$  & Space of globally bounded piecewise continuous (right-continuous) functions from $\R_+$ to $U$ with $\|u\|_{PC_b(\R_+,U)}=\|u\|_{C(\R_+,U)}$ \\
	$C(X)$  &  $:=C(X,X)$\\
  $C_0^k(a,b)$  & Space of $k$ times continuously differentiable functions $f:(a,b) \to \R$ with a support, compact in $(a,b)$ \\	
  $\PD$  & $:=\left\{\gamma \in C(\R_+):  \gamma(0)=0 \mbox{ and } \gamma(r)>0 \mbox{ for } r>0 \right\}$ \\
	$\K$  & $:=\left\{\gamma \in \PD : \gamma \mbox{ is strictly increasing}  \right\}$\\
	$\K_{\infty}$ & $:=\left\{\gamma\in\K: \gamma\mbox{ is unbounded}\right\}$\\
	$\LL$  & $:=\{\gamma \in C(\R_+): \gamma\mbox{ is strictly decreasing with } \lim\limits_{t\rightarrow\infty}\gamma(t)=0 \}$ \\
	$\KL$  & $:=\left\{\beta:\R_+\times\R_+\rightarrow\R_+: \beta \mbox{ is continuous, } \beta(\cdot,t)\in{\K},\ \forall t \geq 0, \right.$\\
	&\phantom{:=}$\left. \beta(r,\cdot)\in {\LL},\ \forall r > 0\right\}$\\					
	$L(X,U)$  & Space of bounded linear operators from $X$ to $U$ \\
	$L(X)$  & $:=L(X,X)$\\					
  $L_{\infty}(\R_+,\R^m)$  & The set of Lebesgue measurable functions with $\|f\|_{\infty}<\infty$ \\
	$\|f\|_{\infty}$		& $:=\esssup_{x\geq 0}|f(x)| = \inf_{D \subset \R_+,\ \mu(D)=0} \sup_{x \in \R_+ \backslash D} |f(x)|$  \\
	$L_p(a,b)$  & Space of $p$-th power integrable functions $f:(a,b) \to \R$ with \linebreak $\|f\|_{L_p(a,b)}<\infty$ \\
   $\|f\|_{L_p(a,b)}$ & $:=\left( \int_0^d{|f(x)|^p dx} \right)^{\frac{1}{p}}$ \\
	$W^{k,p}(a,b)$ & Sobolev space of functions $f \in L_p(a,b)$ (with the norm $\|\cdot\|_{W^{k,p}(a,b)}$), which have weak derivatives of order $\leq k$, all of which belong to $L_p(a,b)$ \\
	$\|f\|_{W^{k,p}(a,b)}$ & $:=\left( \int_0^d{\sum_{1 \leq s \leq k}\left|\frac{\partial^s f}{\partial x^s}(x)\right|^p dx} \right)^{\frac{1}{p}}$ \\	
	$W^{k,p}_0(a,b)$ & Closure of $C_0^k(a,b)$ in the norm of $W^{k,p}(a,b)$ \\
	$H^k(a,b)$ & $=W^{k,2}(a,b)$ \\
	$H^k_0(a,b)$ & $=W^{k,2}_0(a,b)$ \\
	$M(I,W)$ & $:= \{f: I \to W: f \text{ is strongly measurable} \}$\\
$L_p(I,W)$ & $:= \{f \in M(I,W): \|f\|_{L_p(I,W)}:=\Big(\int_I\|f(s)\|^p_Wds\Big)^{1/p} < \infty \}$\\
$L_\infty(I,W)$ & $:= \{f \in M(I,W): \|f\|_{L_\infty(I,W)}:=\esssup_{s\in I}\|f(s)\|_W < \infty \}$

\end{longtable}

%
\vspace{-6mm}
\begin{longtable}{l p{.55\textwidth} l} 
\caption*{\textbf{Acronyms}}
\endfirsthead
\endhead
0-UAS	  &  Uniform asymptotic stability at zero& Definition~\ref{def:0-UAS}, p.~\pageref{def:0-UAS}\\

AG	  & Asymptotic gain property &  p.~\pageref{def:AG}\\

BIC	  &  Boundedness-implies-continuation property  & Definition~\ref{def:BIC}, p.~\pageref{def:BIC}\\

BRS	  &  Boundedness of reachability sets& Definition~\ref{def:BRS}, p.~\pageref{def:BRS}\\

bULIM	  &  Uniform limit property on bounded sets & Definition~\ref{def:bULIM}, p.~\pageref{def:bULIM}\\
CEP	  &  Continuity at equilibrium point & Definition~\ref{def:RobustEquilibrium_Undisturbed}, p.~\pageref{def:RobustEquilibrium_Undisturbed}\\
DSS	  &  Disturbance-to-state stability & Definition~\ref{pageref:DSS}, p.~\pageref{pageref:DSS}\\
eISS	  & Exponentially input-to-state stability & Definition~\ref{def:eISS}, p.~\pageref{def:eISS}\\


iISS	  &  Integral input-to-state stability & Definition~\ref{def:iISS}, p.~\pageref{def:iISS}\\
IOS	  &   Input-to-output stability & Definition~\ref{def:IOS}, p.~\pageref{def:IOS}\\
ISpS	  & Input-to-state practical stability  & Definition~\ref{Def:ISpS_wrt_set}, p.~\pageref{Def:ISpS_wrt_set}\\
ISS	  & Input-to-state stability & Definition~\ref{def:ISS}, p.~\pageref{def:ISS}\\

LIM	  &  Limit property & Definition~\ref{def:LIM}, p.~\pageref{def:LIM}\\
LISS &  Local input-to-state stability & Definition~\ref{def:LISS}, p.~\pageref{def:LISS}\\

ODE, ODEs	  & Ordinary differential equation(s) & \\
PDE, PDEs & Partial differential equation(s) & \\

sAG	  &  Strong asymptotic gain property& Definition~\ref{def:sISS}, p.~\pageref{def:sISS}\\
sISS	  & Strong input-to-state stability & Definition~\ref{def:sISS}, p.~\pageref{def:sISS}\\


UAG	  &  Uniform asymptotic gain property& Definition~\ref{def:UAG}, p.~\pageref{def:UAG}\\

ULS	  &  Uniform local stability& Definition~\ref{def:ULS}, p.~\pageref{def:ULS}\\
ULIM	  &  Uniform limit property & Definition~\ref{def:bULIM}, p.~\pageref{def:bULIM}\\

0-UGAS	  &  Uniform global asymptotic stability at zero& Definition~\ref{def:0-UGAS}, p.~\pageref{def:0-UGAS}\\
UGS	  &  Uniform global stability & Definition~\ref{def:ULS}, p.~\pageref{def:ULS}\\

\end{longtable}

\bibliographystyle{abbrv}
\bibliography{Mir_LitList_NoMir,MyPublications}

\end{document}